\documentclass[10pt]{article}

\usepackage[margin=1in]{geometry} 
\usepackage{libertine}
\usepackage[libertine]{newtxmath}

\usepackage[english]{babel} 
\usepackage{amsmath, amsfonts, amsthm, amssymb} 
\usepackage[scr=boondoxo,cal=euler]{mathalfa} 
\usepackage{tikz-cd} 
\usepackage{indentfirst} 
\usepackage{fancyhdr} 
\usepackage{graphicx} 
\usepackage{enumitem} 
\usepackage{microtype} 
\usepackage{pdfpages} 
\usepackage{centernot} 
\usepackage{ragged2e} 
\usepackage{pinlabel}
\usepackage{caption}
\usepackage[hidelinks]{hyperref} 

\allowdisplaybreaks 
\usetikzlibrary{decorations.pathmorphing}

\theoremstyle{plain}
\newtheorem{thm}{Theorem}[section]
\newtheorem*{thm*}{Theorem}
\newtheorem{lem}[thm]{Lemma}
\newtheorem*{lem*}{Lemma}

\newtheorem*{cor*}{Corollary}
\newtheorem{prop}[thm]{Proposition}
\newtheorem*{prop*}{Proposition}

\newtheorem*{conj*}{Conjecture}

\newtheorem*{ques*}{Question}

\theoremstyle{definition}
\newtheorem{df}[thm]{Definition}
\newtheorem*{df*}{Definition}
\newtheorem*{dfs*}{Definitions}

\newtheorem*{exercise*}{Exercise}

\theoremstyle{remark}
\newtheorem{rem}[thm]{Remark}
\newtheorem*{rem*}{Remark}

\newtheorem{example}[thm]{Example}
\newtheorem*{example*}{Example}

\usepackage{etoolbox}
\patchcmd{\thmhead}{(#3)}{#3}{}{}
\makeatletter
\g@addto@macro\bfseries{\boldmath}
\makeatother

\newcommand{\cl}[1]{\mathcal{#1}}
\newcommand{\fk}[1]{\mathfrak{#1}}

\newcommand{\sr}[1]{\mathscr{#1}}
\newcommand{\Z}{\mathbf{Z}} 
\newcommand{\R}{\mathbf{R}} 
\newcommand{\C}{\mathbf{C}} 
\newcommand{\A}{\mathbf{A}} 
\newcommand{\F}{\mathbf{F}}

\newcommand{\x}{\times}

\renewcommand{\sl}{\fk{sl}}

\newcommand{\ol}[1]{\overline{#1}}

\newcommand{\pt}{\mathrm{pt}}

\newcommand{\B}{\mathrm{B}}

\DeclareMathOperator{\Hom}{Hom}

\DeclareMathOperator{\Id}{Id}

\DeclareMathOperator{\SU}{SU}
\DeclareMathOperator{\U}{U}

\DeclareMathOperator{\End}{End}

\DeclareMathOperator{\Sym}{Sym}

\DeclareMathOperator{\Gr}{Gr}

\usepackage{graphicx}
\usepackage{subcaption}
\usepackage{mathtools}
\usetikzlibrary{calc,fit,backgrounds}
\usetikzlibrary{decorations.markings}
\tikzset{
	neg/.style={postaction=decorate,
	decoration={markings,
	mark=at position 0.05cm with \node{$\circ$};
	}}
}
\usepackage{pinlabel}

\renewcommand{\B}{\mathbf{B}}

\newcommand{\wzero}{
\:\begin{gathered}
	\includegraphics[width=.06\textwidth]{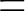}
	\vspace{-3pt}
\end{gathered}\:
}

\newcommand{\wladder}{
\:\begin{gathered}
	\includegraphics[width=.06\textwidth]{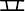}
	\vspace{-3pt}
\end{gathered}\:
}

\newcommand{\wcup}{
\:\begin{gathered}
	\includegraphics[width=.06\textwidth]{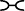}
	\vspace{-3pt}
\end{gathered}\:
}

\newcommand{\wcupgray}{
\:\begin{gathered}
	\includegraphics[width=.06\textwidth]{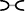}
	\vspace{-3pt}
\end{gathered}\:
}

\newcommand{\wfullcup}{
\:\begin{gathered}
	\includegraphics[width=.06\textwidth]{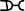}
	\vspace{-3pt}
\end{gathered}\:
}

\newcommand{\vzero}{
\:\begin{gathered}
	\includegraphics[width=.06\textwidth]{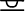}
	\vspace{-3pt}
\end{gathered}\:
}
\newcommand{\vone}{
\:\begin{gathered}
	\includegraphics[width=.06\textwidth]{ladder.pdf}
	\vspace{-3pt}
\end{gathered}\:
}
\newcommand{\vtwo}{
\:\begin{gathered}
	\includegraphics[width=.06\textwidth]{fullcup.pdf}
	\vspace{-3pt}
\end{gathered}\:
}

\newcommand{\pizerolarge}{
\:\begin{gathered}
	\includegraphics[width=.12\textwidth]{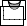}
	\vspace{-3pt}
\end{gathered}\:
}
\newcommand{\pitwolarge}{
\:\begin{gathered}
	\includegraphics[width=.12\textwidth]{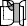}
	\vspace{-3pt}
\end{gathered}\:
}
\newcommand{\iotazerolarge}{
\:\begin{gathered}
	\includegraphics[width=.12\textwidth]{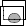}
	\vspace{-3pt}
\end{gathered}\:
}
\newcommand{\iotatwolarge}{
\:\begin{gathered}
	\includegraphics[width=.12\textwidth]{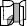}
	\vspace{-3pt}
\end{gathered}\:
}

\newcommand{\wzerolarge}{
\:\begin{gathered}
	\labellist
	\small
	\pinlabel $2$ at -.7 4.8
	\pinlabel $2$ at -.7 0.3
	\pinlabel $2$ at 12.1 4.8
	\pinlabel $2$ at 12.1 0.3
	\endlabellist
	\includegraphics[width=.12\textwidth]{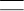}
\end{gathered}\:
}

\newcommand{\wtwolarge}{
\:\begin{gathered}
	\labellist
	\small
	\pinlabel $2$ at -.7 4.8
	\pinlabel $2$ at -.7 0.3
	\pinlabel $4$ at 5.7 3.6
	\pinlabel $2$ at 12.1 4.8
	\pinlabel $2$ at 12.1 0.3
	\endlabellist
	\includegraphics[width=.12\textwidth]{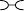}
\end{gathered}\:
}

\newcommand{\vzerolarge}{
\:\begin{gathered}
	\labellist
	\small
	\pinlabel $2$ at -.7 4.8
	\pinlabel $2$ at -.7 0.3
	\pinlabel $1$ at 2.2 2.9
	\pinlabel $1$ at 4.2 5.8
	\pinlabel $2$ at 12.1 4.8
	\pinlabel $2$ at 12.1 0.3
	\endlabellist
	\includegraphics[width=.12\textwidth]{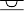}
\end{gathered}\:
}

\newcommand{\vonelarge}{
\:\begin{gathered}
	\labellist
	\small
	\pinlabel $2$ at -.7 4.8
	\pinlabel $2$ at -.7 0.3
	\pinlabel $1$ at 2.2 2.9
	\pinlabel $1$ at 4.2 5.8
	\pinlabel $2$ at 12.1 4.8
	\pinlabel $2$ at 12.1 0.3
	\pinlabel $1$ at 9.2 2.9
	\pinlabel $3$ at 6.7 1.2
	\endlabellist
	\includegraphics[width=.12\textwidth]{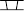}
\end{gathered}\:
}

\newcommand{\vtwolarge}{
\:\begin{gathered}
	\labellist
	\small
	\pinlabel $2$ at -.7 4.8
	\pinlabel $2$ at -.7 0.3
	\pinlabel $1$ at 2.2 2.9
	\pinlabel $1$ at 4.2 5.8
	\pinlabel $2$ at 12.1 4.8
	\pinlabel $2$ at 12.1 0.3
	\pinlabel $4$ at 7.6 3.6
	\pinlabel $3$ at 6.1 0.3
	\endlabellist
	\includegraphics[width=.12\textwidth]{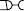}
\end{gathered}\:
}

\newcommand{\vzeroxone}{
\:\begin{gathered}
	\includegraphics[width=.06\textwidth]{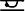}
	\vspace{-3pt}
\end{gathered}\:
}
\newcommand{\vzeroxtwo}{
\:\begin{gathered}
	\includegraphics[width=.06\textwidth]{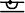}
	\vspace{-3pt}
\end{gathered}\:
}
\newcommand{\vonexone}{
\:\begin{gathered}
	\includegraphics[width=.06\textwidth]{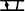}
	\vspace{-3pt}
\end{gathered}\:
}
\newcommand{\vonextwo}{
\:\begin{gathered}
	\includegraphics[width=.06\textwidth]{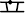}
	\vspace{-3pt}
\end{gathered}\:
}
\newcommand{\vtwoxone}{
\:\begin{gathered}
	\includegraphics[width=.06\textwidth]{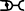}
	\vspace{-3pt}
\end{gathered}\:
}
\newcommand{\vtwoxtwo}{
\:\begin{gathered}
	\includegraphics[width=.06\textwidth]{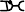}
	\vspace{-3pt}
\end{gathered}\:
}

\newcommand{\wladderone}{
\:\begin{gathered}
	\labellist
	\small
	\pinlabel $1$ at 1.7 2.6
	\pinlabel $1$ at 10 2.6
	\endlabellist
	\includegraphics[width=.06\textwidth]{ladder}
	\vspace{-3pt}
\end{gathered}\:
}
\newcommand{\wladdertwo}{
\:\begin{gathered}
	\labellist
	\small
	\pinlabel $2$ at 1.6 2.6
	\pinlabel $2$ at 10 2.6
	\endlabellist
	\includegraphics[width=.06\textwidth]{ladder}
	\vspace{-3pt}
\end{gathered}\:
}

\newcommand{\threewzerolarge}{
\:\begin{gathered}
	\labellist
	\small
	\pinlabel $3$ at -.7 4.8
	\pinlabel $3$ at -.7 0.3
	\pinlabel $3$ at 12.1 4.8
	\pinlabel $3$ at 12.1 0.3
	\endlabellist
	\includegraphics[width=.12\textwidth]{identity_large}
	\vspace{-3pt}
\end{gathered}\:
}

\newcommand{\threewthreelarge}{
\:\begin{gathered}
	\labellist
	\small
	\pinlabel $3$ at -.7 4.8
	\pinlabel $3$ at -.7 0.3
	\pinlabel $6$ at 5.7 3.6
	\pinlabel $3$ at 12.1 4.8
	\pinlabel $3$ at 12.1 0.3
	\endlabellist
	\includegraphics[width=.12\textwidth]{cup_large}
	\vspace{-3pt}
\end{gathered}\:
}

\newcommand{\threewonelarge}{
\:\begin{gathered}
	\labellist
	\small
	\pinlabel $3$ at -.7 4.8
	\pinlabel $3$ at -.7 0.3
	\pinlabel $1$ at 2.2 2.9
	\pinlabel $2$ at 4.2 5.8
	\pinlabel $3$ at 12.1 4.8
	\pinlabel $3$ at 12.1 0.3
	\pinlabel $1$ at 9.2 2.9
	\pinlabel $4$ at 6.7 1.2
	\endlabellist
	\includegraphics[width=.12\textwidth]{vone_large}
	\vspace{-3pt}
\end{gathered}\:
}

\newcommand{\threewtwolarge}{
\:\begin{gathered}
	\labellist
	\small
	\pinlabel $3$ at -.7 4.8
	\pinlabel $3$ at -.7 0.3
	\pinlabel $2$ at 2.1 2.9
	\pinlabel $1$ at 4.2 5.8
	\pinlabel $3$ at 12.1 4.8
	\pinlabel $3$ at 12.1 0.3
	\pinlabel $2$ at 9.2 2.9
	\pinlabel $5$ at 6.7 1.2
	\endlabellist
	\includegraphics[width=.12\textwidth]{vone_large}
	\vspace{-3pt}
\end{gathered}\:
}

\newcommand{\threewoneonelarge}{
\:\begin{gathered}
	\labellist
	\small
	\pinlabel $3$ at -.7 4.8
	\pinlabel $3$ at -.7 0.3
	\pinlabel $1$ at 1.7 2.9
	\pinlabel $1$ at 4.0 2.9
	\pinlabel $2$ at 3.4 5.8
	\pinlabel $1$ at 6.7 5.8
	\pinlabel $3$ at 12.1 4.8
	\pinlabel $3$ at 12.1 0.3
	\pinlabel $2$ at 9.2 2.9
	\pinlabel $4$ at 3.8 -.9
	\pinlabel $5$ at 6.7 -.9
	\endlabellist
	\includegraphics[width=.12\textwidth]{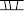}
	\vspace{-3pt}
\end{gathered}\:
}

\newcommand{\threewonetwolarge}{
\:\begin{gathered}
	\labellist
	\small
	\pinlabel $3$ at -.7 4.8
	\pinlabel $3$ at -.7 0.3
	\pinlabel $1$ at 2.2 2.9
	\pinlabel $2$ at 4.2 5.8
	\pinlabel $3$ at 12.1 4.8
	\pinlabel $3$ at 12.1 0.3
	\pinlabel $6$ at 7.6 3.6
	\pinlabel $4$ at 6.1 0.3
	\endlabellist
	\includegraphics[width=.12\textwidth]{vtwo_large}
	\vspace{-3pt}
\end{gathered}\:
}

\newcommand{\threewtwoonelarge}{
\:\begin{gathered}
	\labellist
	\small
	\pinlabel $3$ at -.7 4.8
	\pinlabel $3$ at -.7 0.3
	\pinlabel $2$ at 2.1 2.9
	\pinlabel $1$ at 4.2 5.8
	\pinlabel $3$ at 12.1 4.8
	\pinlabel $3$ at 12.1 0.3
	\pinlabel $6$ at 7.6 3.6
	\pinlabel $5$ at 6.1 0.3
	\endlabellist
	\includegraphics[width=.12\textwidth]{vtwo_large}
	\vspace{-3pt}
\end{gathered}\:
}

\newcommand{\threewoneoneonelarge}{
\:\begin{gathered}
	\labellist
	\small
	\pinlabel $3$ at -.7 4.8
	\pinlabel $3$ at -.7 0.3
	\pinlabel $1$ at 0.9 2.9
	\pinlabel $1$ at 2.8 2.9
	\pinlabel $2$ at 2.4 5.8
	\pinlabel $1$ at 4.3 5.7
	\pinlabel $3$ at 12.1 4.8
	\pinlabel $3$ at 12.1 0.3
	\pinlabel $6$ at 7.6 3.6
	\pinlabel $4$ at 2.8 -0.9
	\pinlabel $5$ at 6.1 0.3
	\endlabellist
	\includegraphics[width=.12\textwidth]{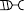}
	\vspace{-3pt}
\end{gathered}\:
}

\newcommand{\generalwrr}{
\:\begin{gathered}
	\labellist
	\small
	\pinlabel $b$ at -.7 4.8
	\pinlabel $a$ at -.7 0.2
	\pinlabel $r$ at 2.1 2.4
	\pinlabel $b-r$ at 5.7 6
	\pinlabel $c$ at 12.1 4.8
	\pinlabel $d$ at 12.1 0.4
	\pinlabel $l+r$ at 10.5 2.6
	\pinlabel $a+r$ at 5.7 -1
	\endlabellist
	\includegraphics[width=.12\textwidth]{vone_large}
	\vspace{-3pt}
\end{gathered}\:
}

\newcommand{\zipgray}{
\begin{gathered}
	\includegraphics[width=.055\textwidth]{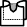}
	\vspace{-3pt}
\end{gathered}	
}

\newcommand{\unzipgray}{
\begin{gathered}
	\includegraphics[width=.055\textwidth]{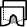}
	\vspace{-3pt}
\end{gathered}	
}

\newcommand{\twograygray}{
\begin{gathered}
	\includegraphics[width=.055\textwidth]{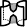}
	\vspace{-3pt}
\end{gathered}
}

\newcommand{\TRladder}{
\begin{gathered}
	\includegraphics[width=.06\textwidth]{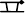}
	\vspace{-3pt}
\end{gathered}
}

\newcommand{\TLladder}{
\begin{gathered}
	\includegraphics[width=.06\textwidth]{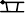}
	\vspace{-3pt}
\end{gathered}
}

\newcommand{\TRdotgray}{
\begin{gathered}
	\includegraphics[width=.045\textwidth]{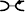}
	\vspace{-3pt}
\end{gathered}
}

\newcommand{\TLdotgray}{
\begin{gathered}
	\includegraphics[width=.045\textwidth]{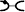}
	\vspace{-4.5pt}
\end{gathered}
}

\newcommand{\Qetwo}{
\begin{gathered}
	\includegraphics[width=.06\textwidth]{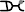}
	\vspace{-3pt}
\end{gathered}
}

\newcommand{\Qeonehone}{
\begin{gathered}
	\includegraphics[width=.06\textwidth]{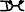}
	\vspace{-3pt}
\end{gathered}
}

\newcommand{\Qhtwo}{
\begin{gathered}
	\includegraphics[width=.06\textwidth]{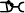}
	\vspace{-1.5pt}
\end{gathered}
}

\newcommand{\vtwodel}{
\begin{gathered}
	\includegraphics[width=.12\textwidth]{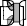}
	\vspace{-3pt}
\end{gathered}
}

\newcommand{\vzerodel}{
\begin{gathered}
	\includegraphics[width=.12\textwidth]{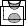}
	\vspace{-3pt}
\end{gathered}
}

\newcommand{\zonezerolarge}{
\begin{gathered}
	\includegraphics[width=.12\textwidth]{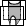}
	\vspace{-3pt}
\end{gathered}
}

\newcommand{\zzeroonelarge}{
\begin{gathered}
	\includegraphics[width=.12\textwidth]{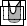}
	\vspace{-3pt}
\end{gathered}
}

\newcommand{\ztwoonelarge}{
\begin{gathered}
	\includegraphics[width=.12\textwidth]{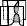}
	\vspace{-3pt}
\end{gathered}
}

\newcommand{\zonetwolarge}{
\begin{gathered}
	\includegraphics[width=.12\textwidth]{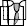}
	\vspace{-3pt}
\end{gathered}
}

\title{The minimal Rickard complexes of braids on two strands}
\author{Joshua Wang}
\date{}

\begin{document}
\maketitle

\begin{abstract}
	The Rickard complex of a braid with strands colored by positive integers is a chain complex of singular Soergel bimodules. The complex determines the colored triply-graded homology and colored $\sl_N$ homology of the braid closure, when closure is color-compatible. For each braid on two strands with any colors, we construct a minimal complex that is homotopy equivalent to its Rickard complex. It is not obtained by laborious simplification; instead, it is defined directly by explicit formulas obtained by educated guesswork and reverse engineering. 
\end{abstract}

\section{Introduction}\label{sec:introduction}

The following chain complex, which we denote by $\sr{P}_1$, is well-known and has appeared in many guises. \[
	\begin{tikzcd}[column sep=40pt]
		\wzero &[-12pt] t^{-1}q\wcupgray \ar[l,swap,"{\:\:\zipgray}" {yshift=5pt}] & t^{-2}q^3\wcupgray \ar[l,swap,"{\:\:\TRdotgray \displaystyle{-} \TLdotgray}" {yshift=10pt}] &[-12pt] t^{-3}q^5\wcupgray \ar[l,swap,"{\:\:\twograygray}" {yshift=5pt}] & t^{-4}q^7\wcupgray \ar[l,swap,"{\:\:\TRdotgray \displaystyle{-} \TLdotgray}" {yshift=10pt}] &[-12pt] \cdots \ar[l,swap,"{\:\:\twograygray}" {yshift=5pt}]
	\end{tikzcd}
\]Webs are oriented from right to left, and foams are read from top to bottom. If we ignore the gray edges and facets, then we may interpret $\sr{P}_1$ as living in Bar-Natan's dotted cobordism category \cite{MR2174270}. With this interpretation, $\sr{P}_1$ is the categorified Jones--Wenzl projector on two strands of Cooper--Krushkal \cite{MR2901969} and Rozansky \cite{MR3205575}. When $\sr{P}_1$ is interpreted to lie in the $\sl_3$ analogue of Bar-Natan's category \cite{MR2336253,MR2457839}, it is Rose's categorification \cite{MR3176309} of the quantum $\sl_3$ projector on two strands.

We explain the meaning of ``projector'' in the context of $\sl_N$ webs and foams \cite{MR3545951,MR4164001}. The Euler characteristic of $\sr{P}_1$ is \[
	\wzero - q(1 - q^2 + q^4 - q^6 + \cdots)\wcupgray \quad = \quad \wzero - \frac{1}{[2]} \wcupgray
\]where $[2] = q^{-1}+q$. This $q$-linear combination of webs diagrammatically represents an endomorphism of $V \otimes V$ where $V$ is the vector representation of the quantum group $U_q(\sl_N)$. Recalling the decomposition $V \otimes V \cong \Sym^2(V) \oplus \Lambda^2(V)$ into irreducibles, this endomorphism is the idempotent projection onto $\Sym^2(V)$, the highest-weight irreducible within $V \otimes V$. 
At the categorified level, $\sr{P}_1$ itself is remarkably also idempotent, in the sense that $\sr{P}_1 \otimes \sr{P}_1$ is homotopy equivalent to $\sr{P}_1$ \cite{MR3412353}. Our preference is to view $\sr{P}_1$ as living in the homotopy category of Soergel bimodules, which is equipped with a functor to each $\sl_N$ foam category that recovers the previous interpretations.
Idempotence in the category of Soergel bimodules, proven for $\sr{P}_1$ in \cite{MR3811775}, implies idempotence in each $\sl_N$ foam category. 

We highlight two features of $\sr{P}_1$ that distinguish it within its homotopy class. First, let $\sr{F}^k(\sr{P}_1)$ denote the subcomplex of $\sr{P}_1$ consisting of its leftmost $1+k$ terms. 
Then $\sr{F}^k(\sr{P}_1)$ is homotopy equivalent to the Rouquier complex of the braid $\sigma^k$ where $\sigma$ is the positive generator of the braid group $\mathrm{Br}_2$. When $k = 3$, for example, we have\[
	\begin{gathered}
		\includegraphics[width=.09\textwidth]{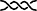}
		\vspace{-3pt}
	\end{gathered} \quad \simeq \quad \begin{tikzcd}[column sep=40pt]
		\wzero &[-12pt] t^{-1}q\wcupgray \ar[l,swap,"{\:\:\zipgray}" {yshift=5pt}] & t^{-2}q^3\wcupgray \ar[l,swap,"{\:\:\TRdotgray \displaystyle{-} \TLdotgray}" {yshift=10pt}] &[-12pt] t^{-3}q^5\wcupgray. \ar[l,swap,"{\:\:\twograygray}" {yshift=5pt}]
	\end{tikzcd}
\]Second, $\sr{P}_1$ and $\sr{F}^k(\sr{P}_1)$ are \textit{minimal} in the sense that any self homotopy equivalence is an isomorphism. A minimal complex has no contractible direct summands, and any equivalent complex admits a deformation retract onto it. Minimality determines the complex up to isomorphism within its homotopy equivalence class. 

We report the discovery of a family of complexes ${}^b_a\sr{P}^c_d$ where $a + b = c + d$ generalizing $\sr{P}_1 \eqcolon {}^1_1\sr{P}^1_1$. They are complexes of singular Soergel bimodules \cite{MR2844932} but may be interpreted as living in the category of $\sl_N$ webs and foams with four fixed endpoints colored by $a,b,c,d$. For simplicity in the introduction, we only discuss the case $a = b = c = d$. For $b \ge 1$, set $\sr{P}_b = {}^b_b\sr{P}^b_b$. 

\begin{thm}\label{thm:mainIntroThm}
	The complex $\sr{P}_b$ has the following properties. \begin{enumerate}[noitemsep]
		\item $\sr{P}_b$ is idempotent up to homotopy in the sense that $\sr{P}_b\otimes \sr{P}_b\simeq \sr{P}_b$. The Euler characteristic of $\sr{P}_b$ is the highest-weight idempotent corresponding to the two-column $b \x 2$ Young diagram. 
		\item $\sr{P}_b$ is invariant up to homotopy under composition with crossings. \vspace{1pt}\[
			\sr{P}_b  \otimes \quad\begin{gathered}
				\labellist
				\pinlabel $b$ at -2 5.5
				\pinlabel $b$ at -2 0.5
				\pinlabel $b$ at 14 0.5
				\pinlabel $b$ at 14 5.5
				\endlabellist
				\includegraphics[width=.06\textwidth]{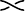}
				\vspace{-3pt}
			\end{gathered} \quad \simeq \sr{P}_b \simeq \quad\begin{gathered}
				\labellist
				\pinlabel $b$ at -2 5
				\pinlabel $b$ at -2 0.5
				\pinlabel $b$ at 14 0.5
				\pinlabel $b$ at 14 5.1
				\endlabellist
				\includegraphics[width=.06\textwidth]{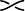}
				\vspace{-3pt}
			\end{gathered}\quad \otimes \sr{P}_b \vspace{-1pt}
		\]Furthermore, for $r \in \{1,\ldots,b\}$, the following four tensor products are contractible. \vspace{2pt}\[
			\sr{P}_b \otimes \quad \begin{gathered}
				\labellist
				\pinlabel $b$ at -2 5.5
				\pinlabel $b$ at -2 0.5
				\pinlabel $r$ at 1.5 2.5
				\pinlabel $b+r$ at 11.5 0.5
				\pinlabel $b-r$ at 11.5 5.5
				\endlabellist
				\includegraphics[width=0.035\textwidth]{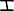}
				\vspace{-3pt}
			\end{gathered} \hspace{25pt} \simeq 0 \hspace{20pt} \sr{P}_b \otimes \quad \begin{gathered}
				\labellist
				\pinlabel $b$ at -2 5.5
				\pinlabel $b$ at -2 0.5
				\pinlabel $r$ at 1.5 2.5
				\pinlabel $b-r$ at 11.5 0.5
				\pinlabel $b+r$ at 11.5 5.5
				\endlabellist
				\includegraphics[width=0.035\textwidth]{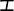}
				\vspace{-3pt}
			\end{gathered} \hspace{25pt} \simeq 0 \hspace{20pt} \hspace{25pt} \begin{gathered}
				\labellist
				\pinlabel $b-r$ at -5 5.5
				\pinlabel $b+r$ at -5 0.5
				\pinlabel $r$ at 1.5 2.5
				\pinlabel $b$ at 8.5 0.5
				\pinlabel $b$ at 8.5 5.5
				\endlabellist
				\includegraphics[width=0.035\textwidth]{rung_diag}
				\vspace{-3pt}
			\end{gathered} \quad \otimes \sr{P}_b \simeq 0 \hspace{20pt}\hspace{25pt} \begin{gathered}
				\labellist
				\pinlabel $b+r$ at -5 5.5
				\pinlabel $b-r$ at -5 0.5
				\pinlabel $r$ at 1.5 2.5
				\pinlabel $b$ at 8.5 0.5
				\pinlabel $b$ at 8.5 5.5
				\endlabellist
				\includegraphics[width=0.035\textwidth]{rung}
				\vspace{-3pt}
			\end{gathered} \quad \otimes \sr{P}_b \simeq 0 \vspace{-0pt}
		\]
		\item $\sr{P}_b$ has an exhaustive increasing filtration $\sr{F}^0(\sr{P}_b) \subseteq \sr{F}^1(\sr{P}_b) \subseteq \cdots$ by bounded subcomplexes. The complex $\sr{F}^k(\sr{P}_b)$ is homotopy equivalent to the Rickard complex of the braid $\sigma^k \in \mathrm{Br}_2$ with both strands colored by $b$.
		\item $\sr{P}_b$ and $\sr{F}^k(\sr{P}_b)$ for $k \ge 0$ are minimal.
	\end{enumerate}
\end{thm}
The main contribution of this paper is really the discovery of the explicit minimal complexes $\sr{F}^k(\sr{P}_b)$, and they have the property that they limit to a complex $\sr{P}_b$ with the stated properties. 

\begin{rem}
	The total number of indecomposable singular Soergel bimodules appearing in $\sr{F}^k(\sr{P}_b)$ across all degrees is $1 + k + k^2 + \cdots + k^b$. There are $b+1$ isomorphism types of indecomposable bimodules, and the $b+1$ terms in the sum are the counts for each isomorphism type. Minimality of $\sr{F}^k(\sr{P}_b)$ implies that any equivalent complex has at least as many indecomposable direct summands. 
\end{rem}
\begin{rem}
	When $k = 1$, the complex $\sr{F}^1(\sr{P}_b)$ is the definitional Rickard complex assigned to the positive crossing. The $k = 2$ case corresponds to the full twist. Beliakova and Habiro constructed a complex in the setting of categorified quantum $\sl_2$ that they conjectured to be homotopy equivalent to the full twist \cite[Conjecture 1.3]{MR4264234}. Hogancamp, Rose, and Wedrich considered the analogue of their complex in the setting of singular Bott--Samelson bimodules and resolved their conjecture affirmatively in this setting \cite[Theorem 3.24]{hogancamp2021skein}. The complex $\sr{F}^2(\sr{P}_b)$ is isomorphic to Hogancamp--Rose--Wedrich's version of Beliakova--Habiro's complex. For $k \ge 3$ and $b \ge 2$, the complexes $\sr{F}^k(\sr{P}_b)$ are new. The challenge of constructing these minimal complexes was originally posed by Wedrich \cite{MR3470705}.
\end{rem}

We present $\sr{P}_2$ in detail in Example~\ref{example:P2}, and we provide a sketch of $\sr{P}_3$ in Example~\ref{example:P3}. In preparation for $\sr{P}_2$, let $W_0, V_0,V_1,V_2$, and $W_2$ be the following five webs\vspace{1pt}\[
	\begin{tikzcd}
		\wzerolarge & \vzerolarge & \vonelarge & \vtwolarge & \wtwolarge
	\end{tikzcd}
\]respectively, and consider the following foams.\[
	\begin{tikzcd}[column sep=63pt]
		\wzero \ar[r,swap,shift right=1.5,"\textstyle \iota = \iotazerolarge \phantom{= \iota}" {yshift=-5pt}] & \vzero \ar[l,swap,shift right=1.5,"\textstyle \pi = \pizerolarge \phantom{=\pi}" {yshift=5pt}] \ar[r,swap,shift right=1.5,"\textstyle Z_{10} = \zonezerolarge \phantom{=Z_{10}}" {yshift=-5pt}] & \vone \ar[l,swap,shift right=1.5,"\textstyle Z_{01} = \zzeroonelarge \phantom{= Z_{01}}" {yshift=5pt}] \ar[r,swap,shift right=1.5,"\textstyle Z_{21} = \ztwoonelarge \phantom{=Z_{21}}" {yshift=-5pt}] & \vtwo \ar[l,swap,shift right=1.5,"\textstyle Z_{12} = \zonetwolarge \phantom{= Z_{01}}" {yshift=5pt}] \ar[r,swap,shift right=1.5,"\textstyle \pi = \pitwolarge \phantom{= \pi}" {yshift=-5pt}] & \wcup \ar[l,swap,shift right=1.5,"\textstyle \iota = \iotatwolarge \phantom{=\iota}" {yshift=5pt}]
	\end{tikzcd}
\]Recall that the nil-Hecke algebra $\cl{H}_2$ is the endomorphism algebra of $\Z[x_1,x_2]$ as a module over $\Z[x_1,x_2]^{\fk{S}_2}$. It is generated by $x_1,x_2$ and the divided difference operator $\partial_1$, which sends $p(x_1,x_2) \mapsto (p(x_1,x_2) - p(x_2,x_1))/(x_1 - x_2)$. The simple transposition $s_1$ sending $p(x_1,x_2) \mapsto p(x_2,x_1)$ is $s_1 = \Id - \,(x_1 - x_2)\partial_1 \in \cl{H}_2$. 
There is an action of $\Z[x_1,x_2]$ on $V_0,V_1$, and $V_2$ that extends to an action of $\cl{H}_2$ on $V_0$ and $V_2$ via\[
	\begin{tikzcd}[column sep=10pt, row sep=10pt]
		x_1 = &[-10pt] \vzeroxone & \vonexone & \vtwoxone\\
		x_2 = &[-10pt] \vzeroxtwo & \vonextwo & \vtwoxtwo\\
	\end{tikzcd} \hspace{30pt} \begin{tikzcd}[column sep=20pt]
		\partial_1 = &[-20pt] \vzerodel & \vtwodel
	\end{tikzcd}\vspace{-5pt}
\]So $s_1 = \Id - \,(x_1 - x_2)\partial_1 \in \cl{H}_2$ thereby acts on $V_0$ and $V_2$. For both $V_0$ and $V_2$, note that $\partial_1 = \iota\,\pi$ where $\iota,\pi$ are the foams defined above, which parallels the factoring of $\partial_1\colon \Z[x_1,x_2] \to \Z[x_1,x_2]$ through $\Z[x_1,x_2]^{\fk{S}_2}$. Lastly, define endomorphisms $Q_1$ and $Q_2$ of $V_1$ and $V_2$, respectively, by \[
	Q_1 = \TRladder - \TLladder \hspace{30pt} Q_2 = \Qetwo - \Qeonehone + \Qhtwo
\]where a black dot is $e_1$ and a gray dot is $e_2$. We note that each $Z_{ij}$ has degree $2$, the maps $x_1,x_2,\partial_1,s_1,\pi,\iota$ have degrees $2,2,-2,0,-1,-1$, respectively, and $Q_1$ and $Q_2$ have degrees $2$ and $4$, respectively.

Given a foam $F\colon q^d\,W_1 \to W_2$ of degree $d$, its \emph{adjoint} foam $F^*\colon q^d\,W_2 \to W_1$, also of degree $d$, is obtained by reflecting $F$ across a horizontal mirror. We call $F\colon q^d\,W \to W$ \emph{self-adjoint} if $F^* = F$ and \emph{skew-adjoint} if $F^* = -F$. The endomorphisms $x_1,x_2,\partial_1,Q_1,Q_2$ are self-adjoint, $s_1$ is skew-adjoint, and $Z^*\!\!\!_{10} = Z_{01}$, $Z^*\!\!\!_{21} = Z_{12}$, and $\pi^* = \iota$.

\begin{example}\label{example:P2}
	$\sr{P}_2$ is the following bicomplex \vspace{-5pt} \[
		\begin{tikzcd}[column sep=24pt,row sep=30pt]
			&[-5pt] & &[15pt] & t^{-8}q^{20}\wcup \ar[d,swap,"\displaystyle Q_2\iota"] &[15pt] \cdots \ar[l,swap,"\displaystyle \pi Z_{21}Z_{12}s_1" {xshift=2pt,yshift=2pt}] \\
			& & & t^{-6}q^{14}\wcup \ar[d,swap,"\displaystyle Z_{21}Z_{12}\iota"] & t^{-7}q^{17}\wfullcup \ar[l,swap,"\displaystyle \pi Q_2 s_1" {xshift=2pt,yshift=2pt}] \ar[d,swap,"\displaystyle Z_{21}Z_{12}"] & \cdots \ar[l,swap,"\displaystyle s^*\hspace{-4pt}_1Z_{21}Z_{12}\partial_1" {xshift=2pt,yshift=2pt}]\\
			& & t^{-4}q^8\wcup \ar[d,swap,"\displaystyle Q_2\iota"] & t^{-5}q^{11}\wfullcup \ar[l,swap,"\displaystyle \pi Z_{21}Z_{12}s_1" {xshift=2pt,yshift=2pt}] \ar[d,swap,"\displaystyle Q_2"] & t^{-6}q^{13}\wfullcup \ar[l,swap,"\displaystyle s^*\hspace{-4pt}_1Q_2\partial_1" {xshift=2pt,yshift=2pt}] \ar[d,swap,"\displaystyle Q_2"] & \cdots \ar[l,swap,"\displaystyle \partial^*\hspace{-4pt}_1 Z_{21}Z_{12}s_1" {xshift=2pt,yshift=2pt}]\\
			& t^{-2}q^2\wcup \ar[d,swap,"\displaystyle Z_{12}\iota"] & t^{-3}q^5\wfullcup \ar[l,swap,"\displaystyle \pi Q_2s_1" {xshift=2pt,yshift=2pt}] \ar[d,swap,"\displaystyle Z_{12}"] & t^{-4}q^7\wfullcup \ar[l,swap,"\displaystyle s^*\hspace{-4pt}_1 Z_{21}Z_{12} \partial_1" {xshift=2pt,yshift=2pt}] \ar[d,swap,"\displaystyle Z_{12}"] & t^{-5}q^9\wfullcup \ar[l,swap,"\displaystyle \partial^*\hspace{-4pt}_1 Q_2s_1" {yshift=2pt,xshift=2pt}] \ar[d,swap,"\displaystyle Z_{12}"] & \cdots \ar[l,swap,"\displaystyle s^*\hspace{-4pt}_1 Z_{21}Z_{12}\partial_1" {xshift=2pt,yshift=2pt}] \\
			\wzero & t^{-1}q^1\wladder \ar[l,swap,"\textstyle \pi Z_{01}" {xshift=2pt,yshift=2pt}] & t^{-2}q^3\wladder \ar[l,swap,"\textstyle Q_1" {xshift=2pt,yshift=2pt}] & t^{-3}q^5\wladder \ar[l,swap,"\textstyle Z_{10}s^*\hspace{-4pt}_1 \partial_1 Z_{01}" {xshift=2pt,yshift=2pt}] & t^{-4}q^7\wladder \ar[l,swap,"\textstyle Q_1" {yshift=2pt,xshift=2pt}] & \cdots \ar[l,swap,"\textstyle Z_{10}s^*\hspace{-4pt}_1\partial_1 Z_{01}" {xshift=2pt,yshift=2pt}]
		\end{tikzcd}
	\]Using the fact that $V_2$ is isomorphic to $[2]W_2$, we see that the Euler characteristic of $\sr{P}_2$ is \[
		\wzero - \frac{1}{[2]}\wladder + \frac{1}{[3]}\wcup
	\]which is the highest-weight idempotent in $\End(\Lambda^2(V) \otimes \Lambda^2(V))$ when interpreted in the setting of $\sl_N$ webs. The subcomplex $\sr{F}^k(\sr{P}_2)$ is defined to consist of the leftmost $1 + k$ columns of the bicomplex, and $\sr{F}^k(\sr{P}_2)$ is homotopy equivalent to the Rickard complex of the braid $\sigma^k$ where both strands are colored by $2$. By counting the number of indecomposable bimodules appearing in $\sr{F}^k(\sr{P}_2)$, we see \[
		1\text{ copy of }\wzero, \quad k\text{ copies of }\wladder, \quad \text{ and } k^2 \text{ copies of }\wcup
	\]so there are $1 + k + k^2$ in total. 
\end{example}

Next, we provide a sketch of $\sr{P}_3$. Consider the following webs.\[
	\begin{tikzcd}[column sep=12pt,row sep=15pt]
		W_0 = \hspace{5pt}\threewzerolarge & W_1 = \hspace{5pt}\threewonelarge & W_2 = \hspace{5pt}\threewtwolarge & W_2^{1,1} = \hspace{5pt}\threewoneonelarge\\
		W_3 = \hspace{5pt}\threewthreelarge & W_3^{1,2} = \hspace{5pt}\threewonetwolarge & W_3^{2,1} = \hspace{5pt}\threewtwoonelarge & W_3^{1,1,1} = \hspace{5pt}\threewoneoneonelarge
	\end{tikzcd}
\]The indecomposable bimodules are $W_0,W_1,W_2$, and $W_3$. There are isomorphisms $W_2^{1,1} \cong [2]W_2$, $W_3^{1,2} \cong W_3^{2,1} \cong [3] W_3$, and $W_3^{1,1,1} \cong [3][2] W_3$.

\begin{example}\label{example:P3}
	$\sr{P}_3$ is a tricomplex of the following form, with an explicit differential given in section~\ref{sec:constructionOfP}. \vspace{-5pt}\[
		\begin{tikzcd}[nodes={inner sep=2pt},column sep={39pt,between origins},row sep={25pt,between origins}]
			& & & & & & & & & t^{-9}q^{27}\,W_3 \ar[ddd] & & \cdots \ar[ll]\\
			& & & & & & & & &\\
			& & & & & & & & &\\
			& & & & & & & & & t^{-8}q^{23}\,W_3^{2,1} \ar[ddd] \ar[ddl] & & \cdots \ar[ll]\\
			& & & & & & & & &\\
			& & & & & & t^{-6}q^{15}\,W_3 \ar[ddd] & & t^{-7}q^{19}\,W_3^{1,2} \ar[ddd] \ar[ll] & & \cdots \ar[ll,crossing over]\\
			& & & & & & & & & t^{-7}q^{17}\,W_3^{2,1} \ar[ddd] \ar[ddl] & & \cdots \ar[ll]\\
			& & & & & & & & &\\
			& & & & & & t^{-5}q^{11}\,W_3^{2,1} \ar[ddd] \ar[ddl] & & t^{-6}q^{14}\,W_3^{1,1,1} \ar[ddd] \ar[ll] \ar[ddl] & & \cdots \ar[ll,crossing over]\\
			& & & & & & & & & t^{-6}q^{14}\,W_2 \ar[ddl] & & \cdots \ar[ll]\\
			& & & t^{-3}q^3\,W_3 \ar[ddd] & & t^{-4}q^7\,W_3^{1,2} \ar[ddd] \ar[ll] & & t^{-5}q^9\,W_3^{1,2} \ar[ll,crossing over] & & \cdots \ar[ll,crossing over]\\
			& & & & & & t^{-4}q^8\,W_2 \ar[ddl] & & t^{-5}q^{11}\,W_2^{1,1} \ar[ll] \ar[ddl] & & \cdots \ar[ll]\\
			& & & & & & & & &\\
			& & & t^{-2}q^2\,W_2 \ar[ddl] & & t^{-3}q^5\,W_2^{1,1} \ar[ll] \ar[ddl] & & t^{-4}q^7\,W_2^{1,1} \ar[ll] \ar[ddl] \ar[from=uuu,crossing over] & & \cdots \ar[ll]\\
			& & & & & & & & &\\
			W_0 & & t^{-1}q^1\,W_1 \ar[ll] & & t^{-2}q^3\,W_1 \ar[ll] & & t^{-3}q^5\,W_1 \ar[ll] & & \cdots \ar[ll]
		\end{tikzcd}
	\]The Euler characteristic of $\sr{P}_3$ is \[
		\wzero - \frac{1}{[2]}\wladderone + \frac{1}{[3]}\wladdertwo - \frac{1}{[4]}\wcup.
	\]The subcomplex $\sr{F}^k(\sr{P}_3)$ is defined to consist of the leftmost $1 + k$ layers of the tricomplex. Within $\sr{F}^k(\sr{P}_3)$, we see \[
		1\text{ copy of }\wzero, \quad k\text{ copies of }\wladderone,\quad k^2 \text{ copies of }\wladdertwo,\quad\text{ and } k^3\text{ copies of }\wcup
	\]so there are $1 + k + k^2 + k^3$ in total.
\end{example}

Lastly, we highlight that the construction of ${}^b_a\sr{P}^c_d$ makes use of another complex that we introduce, which we denote by ${}^b_a\sr{K}^c_d$ where we again set $\sr{K}_b \coloneqq {}^b_b\sr{K}^b_b$. 
The complex $\sr{K}_1$, shown below, is Hogancamp's two-strand compact projector \cite{MR3811775}, which plays a key role in the celebrated computation of the triply-graded homology of torus knots and links \cite{MR3880028,hogancamp2017khovanovrozanskyhomologyhighercatalan,MR4404874,hogancamp2019toruslinkhomology}. \[
	\begin{tikzcd}[column sep=50pt]
		\wzero &[-10pt] t^{-1}q\wcupgray \ar[l,swap,"{\:\:\zipgray}" {yshift=1pt}] & t^{-2}q^3\wcupgray \ar[l,swap,"{\:\:\TRdotgray \displaystyle{-} \TLdotgray}" {yshift=9pt}] &[-10pt] t^{-3}q^4\wzero \ar[l,swap,"{\:\:\unzipgray}" {yshift=1pt}]
	\end{tikzcd}
\]It is related to $\sr{P}_1$ by Koszul duality. When $b = 2$, $\sr{K}_2$ is the following bicomplex. \[
	\begin{tikzcd}[column sep={115pt,between origins}, row sep={55pt,between origins}]
		t^{-3}q^6\vzero \ar[d,swap,"\textstyle Z_{10}"] & t^{-4}q^8\vone \ar[l,swap,"\textstyle s_1\: Z_{01}"] \ar[d,swap,"\textstyle Z_{21}"] & t^{-5}q^{10}\vone \ar[l,swap,"\textstyle Q_1"] \ar[d,swap,"\textstyle Z_{21}"] & t^{-6}q^{10}\vzero \ar[l,swap,"\textstyle Z_{10}\:\partial^*\hspace{-4pt}_1"] \ar[d,swap,"\textstyle Z_{10}"]\\
		t^{-2}q^{4}\vone \ar[d,swap,"\textstyle Q_1"] & t^{-3}q^6\vtwo \ar[l,swap,"\textstyle Z_{12}\:s_1"] \ar[d,swap,"\textstyle Q_2"] & t^{-4}q^8\vtwo \ar[l,swap,"\textstyle s^*\hspace{-4pt}_1\: Q_2 \: \partial_1"] \ar[d,swap,"\textstyle Q_2"] & t^{-5}q^8\vone \ar[l,swap,"\textstyle \partial^*\hspace{-4pt}_1\: Z_{21}"] \ar[d,swap,"\textstyle Q_1"]\\
		t^{-1}q^{2}\vone \ar[d,swap,"\textstyle Z_{01}"] & t^{-2}q^{2}\vtwo \ar[l,swap,"\textstyle Z_{12}\:\partial_1"] \ar[d,swap,"\textstyle Z_{12}"] & t^{-3}q^{4}\vtwo \ar[l,swap,"\textstyle \partial^*\hspace{-4pt}_1 \: Q_2 \: s_1"] \ar[d,swap,"\textstyle Z_{12}"] & t^{-4}q^6\vone \ar[l,swap,"\textstyle s^*\hspace{-4pt}_1 \: Z_{21}"] \ar[d,swap,"\textstyle Z_{01}"]\\
		\vzero & t^{-1}\vone \ar[l,swap,"\textstyle \partial_1\:Z_{01}"] & t^{-2}q^{2}\vone \ar[l,swap,"\textstyle Q_1"] & t^{-3}q^{4}\vzero \ar[l,swap,"\textstyle Z_{10}\: s^*\hspace{-4pt}_1"]
	\end{tikzcd}
\]The complex $\sr{K}_b$ is highly structured with many remarkable properties, and we return to it in future work. 

\begin{rem}
	We summarize how the author came to discover the formulas for the differentials of $\sr{P}_b$ and $\sr{K}_b$. We emphasize that the formulas were \textit{not} obtained by a laborious bookkeeping of the differential through a simplification procedure, which remains infeasible. Instead, they arose by educated guesswork and reverse engineering. 

	First, there is a conjectural isomorphism between the colored $\sl_N$ homology of $2$-stranded torus knots and links and the cohomology of certain spaces of $\SU(N)$ representations of their knot groups that the author verified for the trefoil and the Hopf link in \cite{MR4903257}. Based on the relationship between these $\SU(N)$ representation spaces and the principal angles $\pi/2 \ge \theta_b \ge \cdots \ge \theta_1 \ge 0$ between $b$-dimensional subspaces of $\C^N$, the author guessed that the minimal complex equivalent to $\sigma^k$ with strands colored by $b$ should have the shape of a $b$-dimensional simplex. Using the work of \cite{MR3470705,hogancamp2021skein}, the author could make a precise guess of the shape of the complex and all of its objects, but lacked a formula for the differential. Compatibility with the conjecture related to $\SU(N)$ representations gave some hints at the differential, but only for those components that survive the procedure of forming a braid closure. By thoroughly working through the case $b = 2$, the author saw that the components of the differential repeated in a way that all of the data had a chance of being encoded in a bounded complex having the shape of a $3 \x 3$ square. With wishful thinking, the author guessed further symmetries of this square-shaped complex and was ultimately able to pin down exact formulas for $\sr{K}_2$ amenable to generalization. With the formulas in hand, the author could construct and prove correct all of the earlier guesses. Many of these guesses could not have been made without discussions with Matthew Hogancamp, Matthew Stoffregen, and Michael Willis, in particular.

	Complete computations of the colored $\sl_N$ homology of $2$-stranded torus knots and links including a proof of the conjectural connection to $\SU(N)$ representation spaces of the knot group are provided in forthcoming work, where a number of additional applications are also provided. 
\end{rem}

Preliminaries in section~\ref{sec:preliminaries} include an exposition of singular Soergel bimodules through the lens of equivariant cohomology and Bott--Samelson varieties. In section~\ref{sec:constructionOfK}, we construct ${}^b_a\sr{K}^c_d$, and in section~\ref{sec:constructionOfP}, we use it to construct ${}^b_a\sr{P}^c_d$. The main theorem (Theorem~\ref{thm:mainThm}) for ${}^b_a\sr{P}^c_d$ that specializes to Theorem~\ref{thm:mainIntroThm} is stated and proved in section~\ref{sec:main_theorem}. 


\theoremstyle{definition}
\newtheorem*{ack}{Acknowledgements}
\begin{ack}
	I thank William Ballinger, Elijah Bodish, Luke Conners, Eugene Gorsky, Matthew Hogancamp, Mikhail Khovanov, Peter Kronheimer, Tomasz Mrowka, David Rose, Raphael Rouquier, Lev Rozansky, Matthew Stoffregen, Catharina Stroppel, Joshua Sussan, Emmanual Wagner, Paul Wedrich, and Michael Willis for many helpful and inspiring conversations. This work is partially supported by the NSF MSPRF grant DMS-2303401 and the Simons Collaboration on ``New Structures in Low-Dimensional Topology''.
\end{ack}


\section{Preliminaries}\label{sec:preliminaries}

In sections~\ref{subsec:symmetric_polynomials} and \ref{subsec:the_nil_hecke_algebra}, we review symmetric polynomials in differences in alphabets and the nil-Hecke algebra together with geometric interpretations. In section~\ref{subsec:bott_samelson_bimodules}, we review singular Bott--Samelson bimodules through the lens of the equivariant cohomology rings of partial flag manifolds and Bott--Samelson varieties. In section~\ref{subsec:maps_between_singular_bott_samelson_bimodules}, we discuss maps of singular Bott--Samelson bimodules. 


\subsection{Symmetric polynomials in differences of alphabets}\label{subsec:symmetric_polynomials}

We review symmetric polynomials in differences in alphabets, which is the algebra relevant to Chern classes of virtual vector bundles. Our exposition leans heavily on \cite[Section 2.1]{hogancamp2021skein}. 

\begin{df}
	An \textit{alphabet} $\A = \{x_1,\ldots,x_a\}$ is a finite set of indeterminates. Let $\Z[\A]\coloneqq \Z[x_1,\ldots,x_a]$ denote the ring of polynomials in $\A$ with integer coefficients, and let $\Sym(\A) = \Z[x_1,\ldots,x_a]^{\fk{S}_a}$ be the ring of symmetric polynomials in $\A$. The elementary symmetric polynomials $e_i(\A)$ and the complete homogeneous symmetric polynomials $h_i(\A)$ are determined by their generating functions \[
		\sum_{i=0}^\infty e_i(\A)t^i = \prod_{x \in \A} (1 + xt) \qquad\qquad \sum_{i=0}^\infty h_i(\A)t^i = \prod_{x \in \A}\frac{1}{1 - xt}.
	\]Recall that $\Sym(\A)$ is isomorphic to the polynomial ring $\Z[e_1(\A),\ldots,e_a(\A)]$ by the fundamental theorem of symmetric polynomials. The $q$-degree of each $x_i$ is defined to be $2 \in \Z$, so that $e_i(\A)$ and $h_i(\A)$ have $q$-degree $2i \in \Z$.
\end{df}
\begin{df}
	Given alphabets $\A$ and $\B$ that are not necessarily disjoint, define the polynomials $e_i(\A + \B), h_i(\A + \B), e_i(\A - \B)$, and $h_i(\A - \B)$ in $\Z[\A \cup \B]$ by the generating functions \begin{align*}
		&\sum_{i=0}^\infty e_i(\A + \B)t^i = \prod_{x \in A} (1 + xt) \prod_{y \in \B} (1 + yt) &&\sum_{i=0}^\infty h_i(\A + \B) t^i = \prod_{x \in \A} \frac{1}{1 -xt}\prod_{y\in\B}\frac{1}{1 -yt}\\
		&\sum_{i=0}^\infty e_i(\A - \B) t^i = \prod_{x \in \A} (1 + xt) \prod_{y \in \B} \frac{1}{1 + yt} &&\sum_{i=0}^\infty h_i(\A - \B) t^i = \prod_{x \in \A} \frac{1}{1 - xt} \prod_{y \in \B} (1 - yt).
	\end{align*}Note that if $\A$ and $\B$ are disjoint, then $e_i(\A + \B) = e_i(\A \cup \B)$ and $h_i(\A + \B) = h_i(\A \cup \B)$. If $\B$ is a subset of $\A$, then $e_i(\A -\B) = e_i(\A\setminus \B)$ and $h_i(\A-\B) = h_i(\A\setminus\B)$. In general, we have the following formulas \begin{align*}
		&e_i(\A + \B) = \sum_{j=0}^i e_{i-j}(\A)e_{j}(\B) &&h_i(\A + \B) = \sum_{j=0}^i h_{i-j}(\A)h_j(\A)\\
		&e_i(\A - \B) = \sum_{j=0}^i (-1)^j e_{i-j}(\A)h_j(\B) &&h_i(\A - \B) = \sum_{j=0}^i (-1)^j h_{i-j}(\A)e_j(\B)
	\end{align*}and we note that $e_i(\A - \B) = (-1)^i h_i(\B - \A)$. These definitions extend to elementary symmetric polynomials and complete homogeneous symmetric polynomials in $\Z$-linear combinations of alphabets by the same formulas. 
\end{df}

The geometric perspective on this algebra is as follows. The alphabets $\A$ and $\B$ correspond to complex vector bundles $A$ and $B$ of ranks $a$ and $b$, respectively, over the same space. The polynomials $e_1(\A),\ldots,e_a(\A)$ represent the Chern classes of $A$, while $x_1,\ldots,x_a$ are its Chern roots. The $q$-degree is just the cohomological degree. The polynomials $(-1)^i h_i(\A)$ are the Segre classes of $A$. Interpreting $\A + \B$ as the direct sum of $A$ and $B$, and interpreting $\A - \B$ as the virtual vector bundle $A - B$, then $e_i(\A + \B)$ and $e_i(\A - \B)$ are just the Chern classes of $A + B$ and $A - B$ expressed in terms of the Chern and Segre classes of $A$ and $B$.

\subsection{The nil-Hecke algebra}\label{subsec:the_nil_hecke_algebra}

We review the nil-Hecke algebra and its connection to the equivariant cohomology rings of partial flag manifolds. 

\begin{df}
	For $n \ge 1$, the \textit{nil-Hecke algebra} $\cl{H}_n$ is the endomorphism algebra of the polynomial ring $\Z[x_1,\ldots,x_n]$ viewed as a module over the ring of symmetric polynomials $\Z[x_1,\ldots,x_n]^{\fk{S}_n}$. As an algebra, $\cl{H}_n$ is generated by the the endomorphisms $x_1,\ldots,x_n,\partial_1,\ldots,\partial_{n-1}$ where $\partial_i$ is the \textit{divided difference operator} or \textit{Demazure operator} given by \[
		\partial_i(P) = \frac{P - s_iP}{x_{i}-x_{i+1}} \qquad \text{ for }P \in \Z[x_1,\ldots,x_n]
	\]where $s_i$ is the simple transposition swapping $x_i$ and $x_{i+1}$. Of course, $s_i$ itself is an element of $\cl{H}_n$, and it may be expressed in terms of the generators as $s_i = \Id - \,(x_i - x_{i+1})\partial_i$. The following is a complete list of relations: \[
	\begin{gathered}
		x_ix_j = x_jx_i \qquad \partial_i\partial_i = 0 \qquad \partial_i \partial_{j}\partial_i = \partial_{j}\partial_i\partial_{j}\text{ for }|i - j| = 1 \qquad \partial_i\partial_j = \partial_j \partial_i \text{ for } |i - j| > 1\\
		\partial_ix_i = \Id + \,x_{i+1}\partial_i \qquad \partial_ix_{i+1} = -\Id + \,x_i\partial_i \qquad \partial_ix_j = x_j \partial_i \text{ for }j \neq i,i+1.
	\end{gathered}
	\]We caution the reader that here, $\partial_ix_j$ denotes the composition of the endomorphisms $\partial_i$ and $x_j$, not the application of $\partial_i$ to the polynomial $x_j$. Divided difference operators satisfy the Leibniz rule with a twist \[
		\partial_i(PQ) = \partial_i(P)s_i(Q) + P \partial_i(Q) = \partial_i(P)Q + s_i(P)\partial_i(Q)
	\]and we note that $s_i\partial_i = \partial_i$ and $\partial_is_i = -\partial_i$. Furthermore, the image of $\partial_i$ and kernel of $\partial_i$ agree and coincide with the set of polynomials that are symmetric in $x_i$ and $x_{i+1}$. We refer to the identities $\partial_i\partial_j = \partial_j\partial_i$, $\partial_is_j = s_j\partial_i$, and $s_is_j = s_js_i$ for $|i-j| > 1$ as \emph{far commutativity}.
\end{df}

\begin{lem}\label{lem:mixedBraidRelation}
	The mixed braid relations \[
		s_is_{i+1}\partial_i = \partial_{i+1}s_is_{i+1} \qquad \partial_is_{i+1}s_i = s_{i+1}s_i\partial_{i+1} \qquad s_i \partial_{i+1} s_i = s_{i+1} \partial_i s_{i+1}
	\]hold in $\cl{H}_n$ for $i = 1,\ldots,n-2$. 
\end{lem}
\begin{proof}
	By applying simple transpositions to both sides of the equation, all three relations are equivalent to the identity $s_is_{i+1}\partial_is_{i+1}s_i = \partial_{i+1}$ which we now verify. Given a polynomial $P \in \Z[x_1,\ldots,x_n]$, we have \[
		s_is_{i+1}\partial_is_{i+1}s_iP = s_is_{i+1}\left( \frac{s_{i+1}s_iP - s_is_{i+1}s_iP}{x_i - x_{i+1}} \right) = \frac{P - s_{i+1}P}{x_{i+1} - x_{i+2}} = \partial_{i+1}P
	\]where we have used the ordinary braid relation $s_is_{i+1}s_i = s_{i+1}s_is_{i+1}$. 
\end{proof}

If $G = \U(n)$ is the unitary group, then the ring $\Z[x_1,\ldots,x_n]^{\fk{S}_n}$ arises as the $G$-equivariant Borel cohomology of a point. The polynomial ring $\Z[x_1,\ldots,x_n]$, as a module over $\Z[x_1,\ldots,x_n]^{\fk{S}_n}$, arises as the $G$-equivariant cohomology of the full flag manifold $X \coloneqq \mathrm{Fl}(\C^n)$. Given a sequence $a_1,\ldots,a_m$ of positive integers for which $a_1 + \cdots + a_m = n$, consider the partial flag manifold $\mathrm{Fl}(a_1,\ldots,a_m;n)$ consisting of pairwise orthogonal $m$-tuples of vector subspaces of $\C^n$ with the given dimensions. The equivariant cohomology of $\mathrm{Fl}(a_1,\ldots,a_m;n)$ may be identified with the ring of polynomials that are invariant under $\fk{S}_{a_1} \x \cdots \x \fk{S}_{a_m} \subseteq \fk{S}_n$. 

Consider the partial flag manifold $X_i\coloneqq \mathrm{Fl}(1,\ldots,1,2,1,\ldots,1;n)$ where the $i$th entry is $2$. Its equivariant cohomology ring is the ring of polynomials that are invariant under the simple transposition $s_i \in \fk{S}_n$. Let $\pi_i\colon X \to X_i$ send $(\Lambda_1,\ldots,\Lambda_n)$ to $(\Lambda_1,\ldots,\Lambda_{i-1},\Lambda_i\oplus\Lambda_{i+1},\Lambda_{i+2},\ldots,\Lambda_n)$. Then the equivariant pullback map $(\pi_i)^*\colon H^*_G(X_i) \to H^*_G(X)$ is the inclusion $\Z[x_1,\ldots,x_n]^{s_i} \hookrightarrow \Z[x_1,\ldots,x_n]$ while the equivariant pushforward map $(\pi_i)_*\colon H^*_G(X) \to H^{*-2}_G(X_i)$ sends $P \in \Z[x_1,\ldots,x_n]$ to $(P - s_iP)/(x_i - x_{i+1}) \in \Z[x_1,\ldots,x_n]^{s_i}$. So the divided difference operator $\partial_i$ on $\Z[x_1,\ldots,x_n]$ is the equivariant push-pull map $(\pi_i)^*(\pi_i)_*$ on $H^*_G(X)$. See for example \cite{MR4655919}. 


\subsection{Singular Bott--Samelson bimodules}\label{subsec:bott_samelson_bimodules}

In this section, we review singular Bott--Samelson bimodules following \cite{MR2844932,hogancamp2021skein}. The geometric interpretations follow from \cite{MR4655919}. See also \cite{MR2776785}.  

\begin{df}
	A \textit{(braid-like) web} $\Gamma$ is an oriented graph smoothly embedded in $[0,1] \x \R$ where each edge is assigned a positive integer called its color, subject to the following conditions: \begin{itemize}[noitemsep]
		\item $\Gamma \cap \{0,1\} \x \R$ consists of degree $1$ vertices of the graph $\Gamma$. All other vertices of $\Gamma$ are trivalent (degree $3$). 
		\item The restriction of the projection map $[0,1] \x \R \to [0,1]$ to each edge of $\Gamma$ has no critical points. Furthermore, each edge is oriented from right to left (from $1 \x \R$ towards $0 \x \R$). 
		\item For each trivalent vertex of $\Gamma$, the sum of the colors of the incoming edges equals the sum of the colors of the outgoing edges. Hence each trivalent vertex has either two incoming edges and one outgoing edge, in which case it is a \textit{merge vertex}, or it has one incoming edge and two outgoing edges, in which case it is a \textit{split vertex}.
	\end{itemize}See Figure~\ref{fig:exampleWeb} for an example. We identify webs that are isotopic rel boundary through webs satisfying these conditions. So without loss of generality, we may assume that the projection map $[0,1] \x \R \to [0,1]$ is injective when restricted to the set of trivalent vertices of $\Gamma$. For all but finitely many $t \in [0,1]$, the vertical line $t \x \R$ intersects $\Gamma$ transversely. The sum of the colors of the edges that $t \x \R$ intersects is independent of $t$. We refer to this sum as the \textit{width} of $\Gamma$. As a convention, an edge labeled zero should be erased and the resulting bivalent vertices smoothed out.
\end{df}

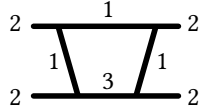
\begin{figure}[!ht]
	\centering
	\labellist
	\pinlabel $2$ at -2 1
	\pinlabel $2$ at -2 38
	\pinlabel $2$ at 21 38
	\pinlabel $2$ at 21 1
	\pinlabel $3$ at 10 9
	\pinlabel $1$ at 10 45
	\pinlabel $1$ at 3 19
	\pinlabel $1$ at 17 19
	\endlabellist
	\includegraphics[width=2cm,height=1cm]{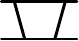}
	\captionsetup{width=.8\linewidth}
	\caption{A web. Edges are always oriented from right to left.}
	\label{fig:exampleWeb}
	\vspace{-10pt}
\end{figure}

Given a braid-like web $\Gamma$, we record the colors of the edges incident to the left vertical line $0 \x \R$ as a tuple $c_L = (a_1,\ldots,a_m)$ ordered from bottom to top. Similarly, the colors of the edges incident to the right vertical line $1 \x \R$ are recorded as a tuple $c_R = (b_1,\ldots,b_l)$, again ordered from bottom to top. We note that $a_1 + \cdots + a_m = b_1 + \cdots + b_l \eqqcolon n$ is the width of $\Gamma$. We say that $\Gamma$ is a web \emph{with boundary data} $c_L,c_R$. 

\begin{df}\label{def:BSbimodule}
	Let $\Gamma$ be a braid-like web with boundary data $c_L,c_R$. The \emph{singular Bott--Samelson bimodule} $B_\Gamma$ \emph{associated to} $\Gamma$ is constructed in the following way. First, assign to each edge $f$ of $\Gamma$ an alphabet $\A_f$ whose size is equal to the color of $f$. Then consider the tensor product $R_\Gamma \coloneqq \bigotimes_{f} \Sym(\A_f)$ over $\Z$ indexed by all edges $f$ of $E$. Next, let $v$ be a trivalent vertex of $\Gamma$, and let $\A^v,\B^v,\C^v$ be the alphabets assigned to the three edges incident to $v$, and assume that $|\C^v| = |\A^v| + |\B^v|$. Let $I_\Gamma$ be the ideal generated by the elements $e_i(\A^v + \B^v) - e_i(\C^v) \in R_\Gamma$ for all $i \ge 1$ as $v$ ranges over all trivalent vertices of $\Gamma$. By homogeneity of the relations, the $q$-grading on $R_\Gamma$ descends to the quotient $R_\Gamma/I_\Gamma$. The quotient $R_\Gamma/I_\Gamma$ is a $(R_L,R_R)$-bimodule where $R_L \coloneqq \bigotimes_f \Sym(A_f)$ where $f$ ranges over the edges incident to left endpoints of $\Gamma$, which lie on $0 \x \R$, while $R_R \coloneqq \bigotimes_f \Sym(A_f)$ where $f$ ranges over the edges incident to right endpoints, which lie on $1 \x \R$. The rings $R_L$ and $R_R$ only depend on the boundary data $c_L,c_R$. 

	The bimodule $B_\Gamma$ is defined to be a particular $q$-grading shift of $R_\Gamma/I_\Gamma$. If $v$ is a merge vertex of $\Gamma$, let $a(v)$ and $b(v)$ be the labels of the two incoming edges of $v$. Then \[
		B_\Gamma \coloneqq q^{-\sum_v a(v)b(v)} R_\Gamma/I_\Gamma
	\]where the sum is over all merge vertices of $\Gamma$. So the element $1 \in B_\Gamma$ has $q$-degree $-\sum_v a(v)b(v) \in \Z$. 
\end{df}

\begin{example}
	The alphabets assigned to the edges of the web of Figure~\ref{fig:exampleWeb} are given names in Figure~\ref{fig:alphabets}. The sizes of these alphabets are $|\A| = |\B| = |\C| = |\mathbf{D}| = 2$, $|\mathbf{X}| = |\mathbf{Y}| = |\mathbf{Z}| = 1$, and $|\mathbf{W}| = 3$. Because $\Sym(\mathbf{K}) = \Z[e_1(\mathbf{K}),\ldots,e_k(\mathbf{K})]$ where $k = |\mathbf{K}|$, the ring $R_\Gamma = \Sym(\A) \otimes \Sym(\B) \otimes \Sym(\C) \otimes \Sym(\mathbf{D}) \otimes \Sym(\mathbf{X}) \otimes \Sym(\mathbf{Y}) \otimes \Sym(\mathbf{Z}) \otimes \Sym(\mathbf{W})$ is a polynomial ring in $2 + 2 + 2 + 2 + 1 + 1 + 1 + 3 = 14$ variables. Let $a_1,a_2,b_1,b_2,c_1,c_2,d_1,d_2,x_1,y_1,z_1,w_1,w_2,w_3$ be the elementary symmetric polynomials of the corresponding alphabets. These are the 14 indeterminates of $R_\Gamma$. The upper left vertex yields the relations $b_1 = x_1 + y_1$ and $b_2 = x_1y_1$ and the upper right vertex yields the relations $c_1 = x_1 + z_1$ and $c_2 = x_1z_1$. The lower two vertices yield the six relations $a_1 + y_1 = w_1 = d_1 + z_1, a_2 + a_1y_1 = w_2 = d_2 + d_1z_1$, and $a_2y_1 = w_3 = d_2z_1$. So $B_\Gamma$ is a $q$-grading shift of the quotient of $R_\Gamma = \Z[a_1,a_2,b_1,b_2,c_1,c_2,d_1,d_2,x_1,y_1,z_1,w_1,w_2,w_3]$ by these ten relations, and it is a $(R_L,R_R)$-bimodule where $R_L = \Sym(\A) \otimes \Sym(\B) = \Z[a_1,a_2,b_1,b_2]$ and $R_R = \Sym(\C) \otimes \Sym(\mathbf{D}) = \Z[c_1,c_2,d_1,d_2]$. The $q$-grading shift is $q^{-1\cdot 2 - 1\cdot 1} = q^{-3}$.
\end{example}

\begin{figure}[!ht]
	\centering
	\vspace{5pt}
	\labellist
	\pinlabel $\A$ at -2 1
	\pinlabel $\B$ at -2 38
	\pinlabel $\C$ at 21 38
	\pinlabel $\mathbf{D}$ at 21 1
	\pinlabel $\mathbf{W}$ at 10 9
	\pinlabel $\mathbf{X}$ at 10 44
	\pinlabel $\mathbf{Y}$ at 2.5 18
	\pinlabel $\mathbf{Z}$ at 17 18
	\endlabellist
	\includegraphics[width=2cm,height=1cm]{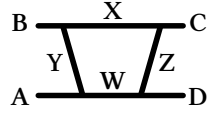}
	\vspace{5pt}
	\captionsetup{width=.8\linewidth}
	\caption{Alphabets assigned to the edges of the web of Figure~\ref{fig:exampleWeb}.}
	\label{fig:alphabets}
\end{figure}

\begin{df}\label{def:BSVariety}
	Let $\Gamma$ be a braid-like web of width $n$. The \textit{Bott--Samelson variety} $V_\Gamma$ associated to $\Gamma$ may be described in the following two ways. 
	\begin{enumerate}
		\item Assign to each edge $e$ of $\Gamma$ of color $a$ an $a$-dimensional vector subspace of $\C^n$. Require that for each vertical line $t \x \R$ that intersects $\Gamma$ transversely, the subspaces assigned to the edges that intersect $t \x \R$ are pairwise orthogonal. The variety $V_\Gamma$ is the space of such configurations of vector subspaces of $\C^n$ indexed by edges of $\Gamma$.
		\item Assign to each region (connected component) of the complement of $\Gamma$ within $[0,1] \x \R^2$ a vector subspace of $\C^n$. Require that for each vertical line $t \x \R$ that intersects $\Gamma$ transversely, the sequence of vector subspaces assigned to the regions that meet the line starting from bottom to top form a partial flag of vector subspaces $W_1 \subseteq \cdots \subseteq W_k$ within $\C^n$ such that the jumps in dimension between two adjacent steps of the partial flag are given by the color of the edge separating the two corresponding regions. The unbounded region below the web is assigned the zero subspace while the unbounded region above the web is assigned $\C^n$. The variety $V_\Gamma$ is the space of such configurations of vector subspaces of $\C^n$ indexed by regions of the complement of $\Gamma$.
	\end{enumerate}See Figure~\ref{fig:exampleBSVariety} for an example. The equivalence between these two descriptions is a straightforward extension of the usual correspondence between a tuple of pairwise orthogonal vector subspaces that span and a partial flag. The boundary data $c_L = (a_1,\ldots,a_m)$ and $c_R = (b_1,\ldots,b_l)$ of $\Gamma$ determine two partial flag manifolds $\mathrm{Fl}_L \coloneqq \mathrm{Fl}(a_1,\ldots,a_m;n)$ and $\mathrm{Fl}_R\coloneqq \mathrm{Fl}(b_1,\ldots,b_l;n)$. The variety $V_\Gamma$ is equipped with forgetful maps to $\mathrm{Fl}_L$ and $\mathrm{Fl}_R$. 
\end{df}

\begin{figure}[!ht]
	\vspace{5pt}
	\centering
	\begin{subfigure}{.49\textwidth}
		\centering
		\labellist
		\pinlabel $\Lambda$ at -2 1
		\pinlabel $\Omega$ at -2 38
		\pinlabel $\Psi$ at 21 38
		\pinlabel $\Phi$ at 21 1
		\pinlabel $A$ at 10 9
		\pinlabel $\alpha$ at 10 43
		\pinlabel $\beta$ at 2.5 18
		\pinlabel $\gamma$ at 17 18
		\endlabellist
		\includegraphics[width=2cm,height=1cm]{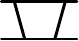}
		\vspace{5pt}
		\captionsetup{width=.97\linewidth}
		\caption{\RaggedRight \normalsize Each edge is assigned a vector subspace of $\C^4$. $\Omega,\Lambda,\Psi,\Phi$ are $2$-dimensional; $\alpha,\beta,\gamma$ are $1$-dimensional; $A$ is $3$-dimensional. The vector subspaces within each of the following five tuples are required to be pairwise orthogonal: $(\Lambda,\Omega), (\Lambda,\beta,\alpha), (A,\alpha), (\Phi,\gamma,\alpha), (\Phi,\Psi)$. }
	\end{subfigure}
	\begin{subfigure}{.49\textwidth}
		\centering
		\labellist
		\pinlabel $\C^4$ at 10.5 48
		\pinlabel $A$ at 10 18
		\pinlabel $\Lambda$ at 1 18
		\pinlabel $\Phi$ at 18.5 18
		\pinlabel $0$ at 10 -8
		\endlabellist
		\vspace{-5pt}
		\includegraphics[width=2cm,height=1cm]{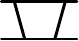}
		\vspace{10pt}
		\captionsetup{width=.97\linewidth}
		\caption{\RaggedRight \normalsize Each region is assigned a vector subspace of $\C^4$. $\Lambda,\Phi$ are $2$-dimensional; $A$ is $3$-dimensional. Each of the following five tuples is required to be a partial flag: $(0, \Lambda, \C^4), (0,\Lambda,A,\C^4), (0,A,\C^4),(0,\Phi,A,\C^4),(0,\Phi,\C^4)$.}
	\end{subfigure}
	\captionsetup{width=.8\linewidth}
	\caption{The singular Bott--Samelson variety assigned to the web of Figure~\ref{fig:exampleWeb}.}
	\label{fig:exampleBSVariety}
\end{figure}

The singular Bott--Samelson bimodule $B_\Gamma$ of a braid-like web $\Gamma$ of width $n$ can be identified with a grading shift of the $G$-equivariant cohomology of the Bott--Samelson variety $V_\Gamma$ where $G = \U(n)$. First, the equivariant cohomology rings of the partial flag manifolds $\mathrm{Fl}_L$ and $\mathrm{Fl}_R$ associated to the boundary data $c_L$ and $c_R$ may be identified with $R_L$ and $R_R$, respectively (see Definition~\ref{def:BSbimodule}). The maps from $V_\Gamma$ to $\mathrm{Fl}_L$ and $\mathrm{Fl}_R$ are $G$-equivariant, and their induced maps give $H^*_G(V_\Gamma)$ the structure of an $(R_L,R_R)$-bimodule. 

\begin{prop}\label{prop:H*BSVariety}
	There is an isomorphism of graded bimodules $B_\Gamma \cong q^{\dim_\C \mathrm{Fl}_L \,-\, \dim_\C V_\Gamma} \:H^*_G(V_\Gamma)$.
\end{prop}
\begin{proof}
	Following the notation of Definition~\ref{def:BSbimodule}, we construct a map $R_\Gamma/I_\Gamma \to H^*_G(V_\Gamma)$ that we subsequently show is an isomorphism. Using the first description of $V_\Gamma$ given in Definition~\ref{def:BSVariety}, we see that there is a tautological vector bundle over $V_\Gamma$ for each edge of $\Gamma$. Given an edge $f$, the fiber of the associated tautological vector bundle over a given configuration of vector subspaces is the vector space assigned to $f$. Define $R_\Gamma \to H^*_G(V_\Gamma)$ by sending the elementary symmetric polynomials in the alphabet associated to $f$ to the equivariant Chern classes of the dual of this tautological vector bundle. The relations generating the ideal $I_\Gamma$ are sent to zero by the Whitney sum formula relating the vector bundles associated to the three edges incident to a vertex. We thereby obtain a map $R_\Gamma/I_\Gamma \to H^*_G(V_\Gamma)$ that is easily seen to be a bimodule map. 

	We show that this map is an isomorphism by induction on the number of vertices of $\Gamma$. If $\Gamma$ has no vertices, then the map $V_\Gamma \to \mathrm{Fl}_L$ is an isomorphism, and $R_\Gamma/I_\Gamma \to H^*_G(V_\Gamma)$ is just the Borel presentation of $H^*_G(\mathrm{Fl}_L)$. If $\Gamma$ has vertices, consider the rightmost vertex $v$. If $v$ is a split vertex, consider the web $\Gamma'$ obtained in the following way. Let $f$ be the edge whose left endpoint is $v$ and note that the right endpoint of $f$ lies on $1 \x \R$. Let $g$ and $h$ be the other two edges incident to $v$.  Disconnect the three edges incident to $v$, and drag the loose ends of $g$ and $h$ to $1 \x \R$ by following along either side of $f$. Then erase $f$ and call the resulting web $\Gamma'$. Then $\Gamma'$ has one fewer vertex than $\Gamma$, and it is straightforward to see that there are isomorphisms of left-modules, drawn vertically below, that make the diagram \[
		\begin{tikzcd}
			R_\Gamma/I_\Gamma \ar[r] \ar[d] & H^*_G(V_\Gamma) \ar[d]\\
			R_{\Gamma'}/I_{\Gamma'} \ar[r] & H^*_G(V_{\Gamma'})
		\end{tikzcd}
	\]commute. By induction, $R_\Gamma/I_\Gamma \to H^*_G(V_\Gamma)$ is an isomorphism. Now assume that the rightmost vertex $v$ is a merge vertex. Let $g$ and $h$ be the incoming edges to $v$ with colors $a$ and $b$, and note that their right endpoints lie on $1 \x \R$. Let $f$ be the outgoing edge of $v$ with color $a + b$. Let $\Gamma'$ be obtained by disconnecting the three edges incident to $v$, dragging the loose end of $f$ to $1 \x \R$ within the region originally bounded by $g,h$, and $1 \x \R$, and then erasing $g$ and $h$. Then the natural map $V_\Gamma \to V_\Gamma'$ that forgets the subspaces assigned to $g$ and $h$ is a fiber bundle with fiber isomorphic to the Grassmannian $\Gr(a,a+b)$ of $a$-dimensional subspaces within $\C^{a + b}$. Furthermore, $\Gamma$ can be obtained by taking the Grassmannian bundle $\Gr(a,-)$ of the tautological bundle $F$ over $\Gamma'$ associated to the elongated version of $f$. By the formula for the equivariant cohomology of a Grassmannian bundle \cite[Proposition 4.5.1]{MR4655919}, we find that \[
		H^*_G(V_\Gamma) \cong \frac{H^*_G(V_{\Gamma'}) \otimes \Sym(\A) \otimes \Sym(\B)}{(c_i(F^\vee) - e_i(\A + \B) \text{ for }i \ge 1)} \cong \frac{R_{\Gamma'}/I_{\Gamma'} \otimes \Sym(\A) \otimes \Sym(\B)}{(c_i(\mathbf{F}) - e_i(\A + \B) \text{ for }i \ge 1)} \cong R_{\Gamma}/I_\Gamma
	\]Here $\A$ and $\B$ are alphabets of size $a$ and $b$ respectively, and $\F$ is the alphabet of size $a + b$ associated to the edge $f$. The first isomorphism is as modules over $H^*_G(V_{\Gamma'})$ and the elementary symmetric polynomials in $\Sym(\A)$ and $\Sym(\B)$ are identified with the Chern classes of the duals of the tautological bundles over $V_\Gamma$ assigned to the edges $g$ and $h$. The second isomorphism is by the inductive hypothesis. The third isomorphism follows from the definitions of $R_\Gamma$ and $I_\Gamma$.

	This inductive argument also shows that $V_\Gamma$ is a tower of bundles with Grassmannians as fibers over $\mathrm{Fl}_L$. The steps of the tower are in bijection with the merge vertices of $\Gamma$, and the fiber of the bundle corresponding to the merge vertex $v$ is the Grassmannian $\Gr(a(v),a(v)+b(v))$ where $a(v)$ and $b(v)$ are the colors of the incoming edges to $v$. Since $\dim_\C \Gr(a(v),a(v)+b(v)) = a(v)b(v)$, we have \[
		\sum_v a(v)b(v) = \dim_\C V_\Gamma - \dim_\C \mathrm{Fl}_L.
	\]where the sum is over merge vertices $v$ of $\Gamma$. This identifies the grading shifts which finishes the proof.
\end{proof}

The singular Bott--Samelson bimodule $B_\Gamma$ associated to a web $\Gamma$ is given in Definition~\ref{def:BSbimodule}. In general, a \emph{singular Bott--Samelson bimodule} $B$ relative to boundary data $c_L,c_R$ is any $\Z$-graded $(R_L,R_R)$-bimodule that is isomorphic to a finite direct sum of $q$-grading shifts of singular Bott--Samelson bimodules assigned to webs. So\[
	B \cong \bigoplus_{j=1}^m q^{i_j} B_{\Gamma_j}
\]for $i_j \in \Z$ and webs $\Gamma_j$ with boundary data $c_L,c_R$. 

\subsection{Maps between singular Bott--Samelson bimodules}\label{subsec:maps_between_singular_bott_samelson_bimodules}

If $B$ and $C$ are singular Bott--Samelson bimodules relative to the same boundary data $c_L,c_R$, then \[
	\Hom(B,C) \coloneqq \bigoplus_{i \in \Z} \Hom^i(B,C)
\]where $\Hom^i(B,C)$ is the space of $(R_L,R_R)$-bimodule maps from $B$ to $C$ that are homogeneous of $q$-degree $i$. We note that $\Hom^i(B,C) = \Hom^0(q^iB,C) = \Hom^0(B,q^{-i}C)$. Using the same grading shift notation as before, we have \[
	q^i\Hom(B,C) = \Hom(q^{-i}B,C) = \Hom(B,q^iC).
\]The notation $f\colon B \to C$ is reserved for degree $0$ bimodule maps. So $g\colon q^iB \to q^jC$ denotes a map $g \in \Hom^0(q^iB,q^jC) = \Hom^{j-i}(B,C)$. 
The category of singular Bott--Samelson bimodules relative to $c_L,c_R$ is thereby a full subcategory of the category of $\Z$-graded $(R_L,R_R)$-bimodules. The category of \emph{singular Soergel bimodules} is the smallest full subcategory of $\Z$-graded $(R_L,R_R)$-bimodules containing singular Bott--Samelson bimodules that is closed under taking direct summands. In this paper, we will not encounter any singular Soergel bimodules that are not singular Bott--Samelson bimodules.

Next, let $c_L,c_M,c_R$ be tuples of positive integers that have the same sum. If $\Gamma$ and $\Delta$ are webs with boundary data $c_L,c_M$ and $c_M,c_R$, respectively, then we may glue the right endpoints of $\Gamma$ to the left endpoints of $\Delta$ to obtain a web $\Gamma\Delta$ with boundary data $c_L,c_R$. For example,\[
	\Gamma = \begin{gathered}
		\includegraphics[width=.12\textwidth]{vtwo_large}
		\vspace{-3pt}
	\end{gathered} \hspace{40pt} \Delta = \begin{gathered}
		\includegraphics[width=.12\textwidth]{vone_large}
		\vspace{-3pt}
	\end{gathered} \hspace{40pt} \Gamma\Delta = \begin{gathered}
		\includegraphics[width=.12\textwidth]{vtwo_large}
		\hspace{-2pt}
		\includegraphics[width=.12\textwidth]{vone_large}
		\vspace{-3pt}
	\end{gathered}
\]Then there is a natural isomorphism\[
	B_{\Gamma\Delta} = B_\Gamma \otimes_{R_M} B_\Delta
\]where $R_M$ is the ring associated to $c_M$. Similarly, the Bott--Samelson variety $V_{\Gamma\Delta}$ is the fiber product of $V_\Gamma$ and $V_\Delta$ over $\mathrm{Fl}_M$. The operation on bimodules distributes over finite direct sums of $q$-grading shifts by \[
	\left( \bigoplus_j q^{i_j} B_{\Gamma_j} \right) \otimes \left( \bigoplus_k q^{i_k} B_{\Delta_k} \right) = \bigoplus_{j,k} q^{i_j+i_k} B_{\Gamma_j\Delta_k}
\]and defines a bifunctor between the relevant categories of singular Bott--Samelson bimodules. 

We now discuss a number of symmetries and dualities that singular Bott--Samelson bimodules enjoy. 

\begin{df}
	If $\Gamma$ is a web with boundary data $c_L,c_R$, then let $\Gamma^\vee$ be the web with boundary data $c_R,c_L$ obtained by reflecting $\Gamma \subseteq [0,1] \x \R$ across the vertical line $\frac12 \x \R$ and reversing the orientation of all edges. For example, \[
		\Gamma = \begin{gathered}
			\includegraphics[width=0.12\textwidth]{vtwo_large}
		\end{gathered} \hspace{40pt} \Gamma^\vee = \begin{gathered}
			\includegraphics[width=0.12\textwidth]{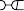}
		\end{gathered}
	\]
\end{df}

Recall that the Bott--Samelson variety $V_\Gamma$ of $\Gamma$ is equipped with a pair of equivariant fiber bundle maps $(\pi_L,\pi_R)$ to partial flag manifolds $(\mathrm{Fl}_L,\mathrm{Fl}_R)$. The Bott--Samelson variety $V_{\Gamma^\vee}$ of $\Gamma^\vee$ is the same variety except that its two maps have swapped roles. Similarly, their singular Bott--Samelson bimodules $B_\Gamma$ and $B_{\Gamma^\vee}$ differ just by swapping the left and right actions. It follows that there is an identification \[
	\Hom(B_\Gamma,B_\Delta) = \Hom(B_{\Gamma^\vee},B_{\Delta^\vee})
\]that we denote by $f \leftrightarrow \ol{f}$. We reserve the notation $f \leftrightarrow f^\vee$ for a contravariant duality \[
	\Hom(B_\Gamma,B_\Delta) \cong \Hom(B_{\Delta^\vee},B_{\Gamma^\vee})
\]that we now explain. 

Set $R_L \coloneqq H^*_G(\mathrm{Fl}_L)$ and $R_R\coloneqq H^*_G(\mathrm{Fl}_R)$ as before where $G = \U(n)$, so that $B_\Gamma$ is an $(R_L,R_R)$-bimodule. Ignoring the left action of $R_L$ for a moment, we may form the right-dual space \[
	B_\Gamma^* \coloneqq \Hom_{\Z,R_R}(B_\Gamma,R_R).
\]This right-dual has the structure of a $(R_R,R_L)$-bimodule inherited from the $(R_L,R_R)$-bimodule structure on $B_\Gamma$. This operation extends to a contravariant functor from the category of singular Bott--Samelson bimodules with boundary data $c_L,c_R$ to the category of $\Z$-graded $(R_R,R_L)$-bimodules. Similarly, we may define the left-dual \[
	{}^*\!B_\Gamma\coloneqq \Hom_{R_L,\Z}(B_\Gamma,R_L)
\]as another $(R_R,R_L)$-bimodule, yielding another functor with the same source and target categories. 

\begin{prop}\label{prop:duality}
	For each web $\Gamma$, there are $(R_R,R_L)$-bimodule isomorphisms \[
		q^{d} B_\Gamma^* \cong B_{\Gamma^\vee} \cong q^{-d} {}^*\!B_\Gamma
	\]where $d\coloneqq \dim_\C\mathrm{Fl}_L - \dim_\C\mathrm{Fl}_R$. The composite isomorphism $q^d B^*_\Gamma \cong q^{-d}{}^*\!B_\Gamma$ is a natural isomorphism of functors.
\end{prop}

In particular, the right-dual and left-dual functors land in the category of singular Bott--Samelson bimodules with boundary data $c_R,c_L$. The duality isomorphism $f \leftrightarrow f^\vee$ stated above is defined using either isomorphism between $B_{\Gamma^\vee}$ and $q^d B_\Gamma^*$ or $q^{-d}{}^*\!B_\Gamma$, and it satisfies $(f^\vee)^\vee = f$. Before proving Proposition~\ref{prop:duality}, we introduce a few maps in the following lemma. 

\begin{lem}\label{lem:unitcounit}
	Each web $\Gamma$ is equipped with bimodule maps \[
		\begin{tikzcd}[row sep=0pt]
			\varepsilon_{L} \colon q^d B_{\Gamma\Gamma^\vee} \to R_L & \varepsilon_{R} \colon q^{-d} B_{\Gamma^\vee\Gamma} \to R_R\\
			\eta_{L} \colon R_L \to q^{-d} B_{\Gamma\Gamma^\vee} & \eta_{R} \colon R_R \to q^{d} B_{\Gamma^\vee\Gamma}
		\end{tikzcd}
	\]where $d \coloneq \dim_\C \mathrm{Fl}_L - \dim_\C \mathrm{Fl}_R$ such that the composites \[
		\begin{tikzcd}[column sep=35pt,row sep=0pt]
			B_{\Gamma^\vee} \ar[r,"\Id \otimes \,\eta_{L}"] & q^{-d}B_{\Gamma^\vee\Gamma\Gamma^\vee} \ar[r,"\varepsilon_R \otimes \,\Id"] & B_{\Gamma^\vee} & B_{\Gamma^\vee} \ar[r,"\eta_R \otimes\,\Id"] & q^{d}B_{\Gamma^\vee\Gamma\Gamma^\vee} \ar[r,"\Id\otimes \,\varepsilon_L"] & B_{\Gamma^\vee}\\
			B_\Gamma \ar[r,"\Id \otimes\,\eta_R"] & q^d B_{\Gamma\Gamma^\vee\Gamma} \ar[r,"\varepsilon_L\otimes\,\Id"] & B_\Gamma & B_\Gamma \ar[r,"\eta_L\otimes\,\Id"] & q^d B_{\Gamma\Gamma^\vee\Gamma} \ar[r,"\Id\otimes\,\varepsilon_R"] & B_\Gamma
		\end{tikzcd}
	\]are the identity maps. In particular, the four maps are the equivariant push-pull maps in both directions of the following two correspondences \[
		\begin{tikzcd}
			\mathrm{Fl}_L & V_{\Gamma} \ar[l,swap,"\pi_L"] \ar[r,"\delta"] & V_{\Gamma} \x_{\mathrm{Fl}_R} V_{\Gamma^\vee} & V_{\Gamma^\vee} \x_{\mathrm{Fl}_L} V_\Gamma & V_\Gamma \ar[l,swap,"\delta"] \ar[r,"\pi_R"] & \mathrm{Fl}_R
		\end{tikzcd}
	\]where $\delta$ denotes the diagonal embedding.
\end{lem}
\begin{proof}
	By definition, the map $(\Id\otimes\,\varepsilon_R)\circ(\eta_L\otimes\Id)$ is the composite of the maps induced by the following maps \[
		\begin{tikzcd}[column sep={55pt,between origins}]
			& V_{\Gamma} \x_{\mathrm{Fl}_L} V_\Gamma \ar[dl,swap,"\pi_L \x\, \Id"] \ar[dr,"\delta\,\x\,\Id"] & & V_\Gamma \x_{\mathrm{Fl}_R} V_\Gamma \ar[dl,swap,"\Id \x \,\delta"] \ar[dr,"\Id \x\, \pi_R"] \\
			\mathrm{Fl}_L \x_{\mathrm{Fl}_L} V_{\Gamma} & & V_{\Gamma} \x_{\mathrm{Fl}_R} V_{\Gamma^\vee} \x_{\mathrm{Fl}_L} V_{\Gamma} & & V_{\Gamma} \x_{\mathrm{Fl}_R} \mathrm{Fl}_R
		\end{tikzcd}
	\]from left to right, using the equivariant push or pull depending on the direction of the arrow. We add to the diagram the fiber product of the two maps in the center to obtain \[
		\begin{tikzcd}[column sep={55pt,between origins}]
			& & V_\Gamma \ar[dl,swap,"\delta"] \ar[dr,"\delta"] & &\\
			& V_{\Gamma} \x_{\mathrm{Fl}_L} V_\Gamma \ar[dl,swap,"\pi_L \x\, \Id"] \ar[dr,"\delta\,\x\,\Id"] & & V_\Gamma \x_{\mathrm{Fl}_R} V_\Gamma \ar[dl,swap,"\Id \x \,\delta"] \ar[dr,"\Id \x\, \pi_R"] \\
			\mathrm{Fl}_L \x_{\mathrm{Fl}_L} V_{\Gamma} & & V_{\Gamma} \x_{\mathrm{Fl}_R} V_{\Gamma^\vee} \x_{\mathrm{Fl}_L} V_{\Gamma} & & V_{\Gamma} \x_{\mathrm{Fl}_R} \mathrm{Fl}_R
		\end{tikzcd}
	\]By trading the lower push-pull for the upper pull-push in the pullback square, we see that the overall composite is the identity map. Minor variations of this argument prove the other three identities. 
\end{proof}

\begin{proof}[Proof of Proposition~\ref{prop:duality}]
	The maps $\varepsilon_L$ and $\varepsilon_R$ defined in Lemma~\ref{lem:unitcounit} are just the Poincar\'e pairings for the fiber bundles $\pi_L\colon V_\Gamma \to \mathrm{Fl}_L$ and $\pi_R\colon V_\Gamma \to \mathrm{Fl}_R$. In particular, they are induced by the maps \[
		q^{2\dim_\C\mathrm{Fl}_L - 2\dim_\C V_\Gamma} H^*_G(V_{\Gamma}) \x H^*_G(V_{\Gamma^\vee}) \to H^*_G(\mathrm{Fl}_L) \hspace{40pt} q^{2\dim_\C\mathrm{Fl}_R - 2\dim_\C V_\Gamma} H^*_G(V_{\Gamma^\vee}) \x H^*_G(V_\Gamma) \to H^*_G(\mathrm{Fl}_R)
	\]given by $(x,y) \mapsto (\pi_L)_*(x \cup y)$ and $(x,y) \mapsto (\pi_R)_*(x \cup y)$ respectively. Nondegeneracy of these pairings \cite[section 3.7]{MR4655919} allows us to identify \[
		B_\Gamma^* = \Hom_{\Z,R_R}(q^{\dim_\C \mathrm{Fl}_L - \dim_\C V_\Gamma} H^*_G(V_\Gamma),R_R) = q^{2\dim_\C\mathrm{Fl}_R - \dim_\C V_\Gamma - \dim_\C\mathrm{Fl}_L}H^*_G(V_{\Gamma^\vee}) = q^{\dim_\C \mathrm{Fl}_R - \dim_\C\mathrm{Fl}_L} B_{\Gamma^\vee}
	\]and \[
		{}^*\!B_\Gamma = \Hom_{R_L,\Z}(q^{\dim_\C \mathrm{Fl}_L - \dim_\C V_\Gamma}H^*_G(V_\Gamma),R_L) = q^{\dim_\C\mathrm{Fl}_L - \dim_\C V_\Gamma} H_G^*(V_{\Gamma^\vee}) = q^{\dim_\C \mathrm{Fl}_L - \dim_\C\mathrm{Fl}_R} B_{\Gamma^\vee}
	\]which gives the desired isomorphisms. 

	We now prove naturality. Suppose $\Delta$ is another web with boundary data $c_L,c_R$ and let $f\colon q^c H^*_G(V_\Gamma) \to H^*_G(V_\Delta)$ be a bimodule map. By nondegeneracy of the pairings, there are unique maps $g_L,g_R\colon q^{c - \dim_\C V_\Delta + \dim_\C V_\Gamma} H^*_G(V_\Delta) \to H^*_G(V_\Gamma)$ for which \[
		(\pi_L)_*(f(x) \cup y) = (\pi_L)_*(x \cup g_L(y)) \hspace{40pt} (\pi_R)_*(f(x) \cup y) = (\pi_R)_*(x \cup g_R(y))
	\]for all $x \in H^*_G(V_\Gamma)$ and $y \in H^*_G(V_\Delta)$. Our goal is to show that $g_L = g_R$. To do so, we use the nondegenerate Poincar\'e pairings \[
		q^{2\dim_\C V_\Gamma} H^*_G(V_\Gamma) \x H^*_G(V_\Gamma) \to H^*_G(\pt) \hspace{40pt} q^{2\dim_\C V_\Delta} H^*_G(V_\Delta) \x H^*_G(V_\Delta) \to H^*_G(\pt)
	\]induced by the maps $\pi\colon V_\Gamma \to \pt$ and $\pi\colon V_\Delta \to \pt$. There is a unique map $g\colon q^{c - \dim_\C V_\Delta + \dim_\C V_\Gamma}H^*_G(V_\Delta) \to H^*_G(V_\Gamma)$ for which \[
		\pi_*(f(x) \cup y) = \pi_*(x \cup g(y))
	\]for all $x \in H^*_G(V_\Gamma)$ and $y \in H^*_G(V_\Delta)$. By applying the pushforward map under $\mathrm{Fl}_L \to \pt$ and $\mathrm{Fl}_R \to \pt$ to the identities characterizing $g_L$ and $g_R$, we see that $g_L = g = g_R$ by uniqueness of $g$. 
\end{proof}

Lemma~\ref{lem:unitcounit} also provides the left and right adjoints of the tensor product operation on singular Bott--Samelson bimodules. 

\begin{prop}\label{prop:adjunction}
	Let $\Gamma$ be a web with boundary data $c_L,c_R$. There are natural isomorphisms of $\Z$-graded abelian groups \[
		\begin{tikzcd}[row sep=0]
			\Hom(B_{\Theta\Gamma},B_\Delta) \cong q^{-d} \Hom(B_\Theta, B_{\Delta\Gamma^\vee}) & \Hom(B_{\Phi},B_{\Gamma\Psi}) \cong q^{d}\Hom(B_{\Gamma^\vee\Phi},B_{\Psi})\\
			\Hom(B_\Delta,B_{\Theta\Gamma}) \cong q^{-d} \Hom(B_{\Delta\Gamma^\vee},B_\Theta) & \Hom(B_{\Gamma\Psi},B_{\Phi}) \cong q^{d}\Hom(B_{\Psi},B_{\Gamma^\vee\Phi})
		\end{tikzcd}
	\]for webs $\Theta,\Delta,\Phi,\Psi$ with appropriate boundary data, where $d \coloneqq \dim_\C\mathrm{Fl}_L - \dim_\C\mathrm{Fl}_R$. 
\end{prop}
\begin{proof}
	This is just the construction of adjoint functors from a unit and counit, which are provided by Lemma~\ref{lem:unitcounit}. For example, a map $f \in \Hom^i(B_{\Theta\Gamma},B_\Delta)$ is sent to the composite \[
		\begin{tikzcd}[column sep=35pt]
			B_\Theta \ar[r,"\Id \otimes \,\eta_L"] & q^{-d}B_{\Theta\Gamma\Gamma^\vee} \ar[r,"f\otimes \,\Id"] & q^{-d-i}B_{\Delta\Gamma^\vee}
		\end{tikzcd}
	\]in $\Hom^{i+d}(B_\Theta,B_{\Delta\Gamma^\vee})$. The inverse isomorphism is given by sending $g \in \Hom^{i+d}(B_\Theta,B_{\Delta\Gamma^\vee})$ to \[
		\begin{tikzcd}[column sep=35pt]
			B_{\Theta\Gamma} \ar[r,"g\,\otimes\,\Id"] & q^{-i-d} B_{\Delta\Gamma^\vee\Gamma} \ar[r,"\Id\otimes \,\varepsilon_R"] & q^{-i} B_{\Delta}
		\end{tikzcd}
	\]The fact that these are inverses follows from the identities in Lemma~\ref{lem:unitcounit}. The other three isomorphisms are defined similarly.
\end{proof}

Given webs $\Gamma,\Delta$ with the same boundary data, we now have two identifications \[
	\Hom(B_\Gamma,B_\Delta) = \Hom(B_{\Gamma^\vee},B_{\Delta^\vee}) \qquad \Hom(B_\Gamma,B_\Delta) \cong \Hom(B_{\Delta^\vee},B_{\Gamma^\vee})
\]denoted $f \leftrightarrow \ol{f}$ and $f \leftrightarrow f^\vee$, respectively. Both operations are involutive. We define one more involutive duality isomorphism \[
	\Hom(B_\Gamma,B_\Delta) \cong \Hom(B_\Delta,B_\Gamma)
\]by $f \leftrightarrow f^*$ where $f^* = (\ol{f})^\vee = \ol{(f^\vee)}$, and we refer to $f^*$ as the \emph{adjoint} of $f$, as mentioned in the introduction. 

We briefly explain foams, which are essentially cobordisms for webs, and refer to \cite{MR3545951,queffelec2018annularevaluationlinkhomology,MR3982970,MR4164001,hogancamp2021skein} for a more detailed account of the theory. In our setting, foams live within $[0,1] \x [0,1] \x \R$, are read from top to bottom, and have the property that generic horizontal slices (intersections with $t \x [0,1] \x \R$) are braid-like webs. 
Certain bimodule maps between singular Bott--Samelson bimodules associated to webs are graphically represented by foams. Foams that are isotopic rel boundary through foams satisfying the horizontal slice condition represent the same bimodule map. Here are the basic examples, from which all other foams can be created. 

For any web $\Gamma$, the identity bimodule map $B_\Gamma \to B_\Gamma$ is represented by the product foam $[0,1] \x \Gamma \subset [0,1] \x [0,1] \x \R$. \[
	\begin{gathered}
		\includegraphics[width=.12\textwidth]{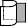}
		\vspace{-3pt}
	\end{gathered}
\]If $f$ is an edge of $\Gamma$ with associated alphabet $\A_f$, then the endomorphism of $B_\Gamma$ given by multiplication by $e_i(\A_f)$ is represented by the identity foam with a dot labeled by $i$ on the facet of the foam corresponding to $f$. A further shorthand for this foam is simply a picture of $\Gamma$ with a dot labeled by $i$ on the edge $f$. \[
	\begin{gathered}
		\labellist
		\pinlabel $i$ at 6 10.8
		\endlabellist
		\includegraphics[width=.12\textwidth]{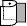}
		\vspace{-3pt}
	\end{gathered} \qquad \eqqcolon \qquad \begin{gathered}
		\labellist
		\small
		\pinlabel $i$ at 5.5 5.7
		\endlabellist
		\includegraphics[width=.12\textwidth]{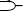}
		\vspace{-3pt}
		\end{gathered}
\]The following foam has a tetrahedral point.\vspace{5pt}\[
	\begin{gathered}
		\includegraphics[width=.12\textwidth]{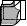}
		\vspace{-3pt}
	\end{gathered} \vspace{5pt}
\]The Bott--Samelson varieties of the webs on the top and bottom of this foam are canonically identified, and this foam represents the induced isomorphism of bimodules. The tensor product of two maps represented by foams is represented by the foam obtained by gluing the two foams together, extending the operation of gluing webs. There is another tensor product operation that we have not discussed that corresponds to taking webs of width $n$ and $m$ and placing one above the other to obtain a web of width $n + m$. The juxtaposition of two foams in a similar manner represents the tensor product of the bimodule maps. 

The adjunction isomorphisms of Proposition~\ref{prop:adjunction} transform foams in the following way. Let $F \subseteq [0,1] \x [0,1] \x \R$ be a foam representing a bimodule map in $\Hom(B_{\Theta\Gamma},B_\Delta)$. Then the boundary of $F$ is naturally segmented into four pieces. On the top is $\Theta\Gamma = F \cap 1 \x [0,1] \x \R$ and on the bottom is $\Delta = F \cap 0 \x [0,1] \x \R$. On the left $F \cap [0,1] \x 0 \x \R$ and on the right $F \cap [0,1] \x 1 \x \R$ are a union of vertical lines. There is an isotopy of $\partial([0,1] \x [0,1] \x \R)$ that rotates $\Gamma$ from lying on the top first to the right and then to the bottom, so that at the end we see $\Theta$ on top and $\Delta\Gamma^\vee$ on the bottom. By isotopy extension, we may drag the foam $F$ along with it, and the resulting foam from $\Theta$ to $\Delta\Gamma^\vee$ represents the corresponding bimodule map in $q^{-d}\Hom(B_{\Gamma},B_{\Delta\Gamma^\vee})$ given by adjunction. The other isomorphisms are also given by dragging $\Gamma$ between the top and bottom by passing either to the right or the left. This is a version of the ``bending trick'' described in \cite[2.3.3]{MR3877770}. 

\begin{example}\label{example:bending}
	Let $\Gamma$ be the web\[
		\begin{gathered}
			\labellist
			\small
			\pinlabel $b$ at -.8 4.8
			\pinlabel $a$ at -.8 0.2
			\pinlabel $a+b$ at 13 2.6
			\endlabellist
			\includegraphics[width=.12\textwidth]{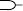}
		\end{gathered}
	\]with boundary data $c_L = (a,b)$ and $c_R = (a+b)$. Then the isomorphisms given in Proposition~\ref{prop:adjunction} \[
		\begin{tikzcd}[column sep={120pt,between origins}]
			\Hom^{ab}(B_{\smash{\Gamma\ol{\Gamma}}},R_L) \ar[r,no head] \ar[rd,no head] & \Hom^0(B_\Gamma,B_\Gamma) \ar[r,no head] \ar[rd,no head] & \Hom^{-ab}(B_{\hspace{0.5pt}\smash{\ol{\Gamma}\Gamma}},R_R)\\
			\Hom^{ab}(R_L,B_{\smash{\Gamma\ol{\Gamma}}}) \ar[r,no head] \ar[ru, no head] & \Hom^0(B_{\hspace{0.5pt}\smash{\ol{\Gamma}}},B_{\hspace{0.5pt}\smash{\ol{\Gamma}}}) \ar[r,no head] \ar[ru, no head] & \Hom^{-ab}(R_R,B_{\hspace{0.5pt}\smash{\ol{\Gamma}\Gamma}})
		\end{tikzcd}
	\]identify the bimodule maps represented by the following foams \[
		\begin{tikzcd}[column sep={120pt,between origins},row sep=0pt]
			\begin{gathered}
				\includegraphics[width=.12\textwidth]{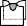}
				\vspace{-3pt}
			\end{gathered} \ar[r,no head] \ar[rd,no head] & \begin{gathered}
				\includegraphics[width=.12\textwidth]{identityFork1.pdf}
				\vspace{-3pt}
			\end{gathered} \ar[r,no head] \ar[rd,no head] & \begin{gathered}
				\includegraphics[width=.12\textwidth]{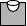}
				\vspace{-3pt}
			\end{gathered}\\
			\begin{gathered}
				\includegraphics[width=.12\textwidth]{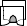}
				\vspace{-3pt}
			\end{gathered} \ar[r,no head] \ar[ru, no head] & \begin{gathered}
				\includegraphics[width=.12\textwidth]{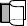}
				\vspace{-3pt}
			\end{gathered} \ar[r,no head] \ar[ru, no head] & \begin{gathered}
				\includegraphics[width=.12\textwidth]{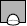}
				\vspace{-3pt}
			\end{gathered}
		\end{tikzcd}
	\]The following formulas for these maps are drawn from \cite[Appendix A]{hogancamp2021skein}. Assign the alphabets $\A = \{x_1,\ldots,x_a\}$, $\B = \{x_{a+1},\ldots,x_{a+b}\}$, and $\A \sqcup \B$ to the three edges of $\Gamma$. Then $R_L = \Z[x_1,\ldots,x_{a+b}]^{\fk{S}_a \x \fk{S}_b}$ and $R_R = \Z[x_1,\ldots,x_{a+b}]^{\fk{S}_{a+b}}$. The six bimodule maps are given by\[
		\begin{tikzcd}[column sep={120pt,between origins},row sep=5pt]
			\substack{\textstyle R_L \otimes_{R_R} R_L' \to R_L\\[2pt]\textstyle f \otimes g' \mapsto fg'} & \Id\colon R_L \to R_L & \substack{\textstyle R_L \to R_R\\[2pt]\textstyle f \mapsto \partial_{a,b}f}\\
			\substack{\textstyle R_L \to R_L \otimes_{R_R} R_L'\\[2pt]\textstyle 1 \mapsto \fk{s}_{b^a}(\A - \B')} & \Id \colon R_L \to R_L & R_R \hookrightarrow R_L
		\end{tikzcd}
	\]where $\partial_{a,b} \coloneqq (\partial_b\cdots\partial_1)(\partial_{b+1}\cdots\partial_2)\cdots(\partial_{a+b-1}\cdots\partial_a)$ and $\fk{s}_{b^a}(\A - \B')$ denotes the extension of the Schur polynomial $\fk{s}_{b^a}$ to differences of alphabets, which can again be understood in terms of characteristic classes of vector bundles. We will only make use of these formulas in the special cases where either $a = 1$ or $b = 1$ where $\fk{s}_{1^a} = e_a$ and $\fk{s}_{b^1} = h_b$. 
\end{example}

The three dualities given by $f \mapsto \ol{f}, f^\vee, f^*$ have the following interpretations in terms of foams. In short, the three dualities act on foams via the Klein four-group action on the square $[0,1] \x [0,1]$. Let $F$ be a foam that represents $f$. \begin{itemize}
	\item The foam $\ol{F}$ obtained by reflecting $F$ across $[0,1] \x \frac12 \x \R$ represents $\ol{f}$. In Example~\ref{example:bending}, this duality swaps the foams in the middle column but fixes each of the four other foams. 
	\item The foam $F^\vee$ obtained by rotation $F$ by $180^\circ$ within the $[0,1] \x [0,1]$ factor of $[0,1] \x [0,1] \x \R$ represents $f^\vee$. In Example~\ref{example:bending}, this duality swaps the two foams within each of the three columns. 
	\item The foam $F^*$ obtained by reflecting $F$ across $\frac12 \x [0,1] \x \R$ represents $f^*$. In Example~\ref{example:bending}, this duality swaps the two foams within the right column and within the left column but fixes each of the foams in the middle column.
\end{itemize}
Proposition~\ref{prop:duality} is encoded by the fact that rotation by $180^\circ$ results in the same foam whether the rotation is clockwise or counterclockwise. 

Any bimodule map between singular Bott--Samelson bimodules turns out to be a $\Z$-linear combination of maps representable by foams, which follows from \cite{MR3177365,MR3682839,MR3545951}. 
Furthermore, there is a functor from the category of singular Bott--Samelson bimodules to the category of $\sl_N$ webs and foams. It sends the bimodule associated to a web to that web, viewed as an object in the $\sl_N$ foam category, and it sends the bimodule associated to a foam to that foam, viewed as morphism in the $\sl_N$ foam category. See \cite{MR3545951,MR3982970,hogancamp2021skein}.

We state the correspondence between foams and bimodule maps for the following four foams that we will later use. \[
	\begin{gathered}
		\includegraphics[width=.12\textwidth]{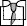}
		\vspace{-3pt}
	\end{gathered} \hspace{40pt} \begin{gathered}
		\includegraphics[width=.12\textwidth]{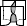}
		\vspace{-3pt}
	\end{gathered} \hspace{40pt} \begin{gathered}
		\includegraphics[width=.12\textwidth]{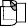}
		\vspace{-3pt}
	\end{gathered} \hspace{40pt} \begin{gathered}
		\includegraphics[width=.12\textwidth]{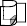}
		\vspace{-3pt}
	\end{gathered}
\]From left to right, these four foams represent the equivariant pullback and pushforward of the map of Bott--Samelson varieties \[
	\begin{gathered}
	 	\labellist
		\small
		\pinlabel $\Omega$ at -1 4.8
		\pinlabel $\Lambda \oplus \Psi$ at -2.5 0.2
		\pinlabel $\Lambda$ at 4.6 2.5
		\pinlabel $\Omega \oplus \Lambda$ at 14.2 4.8
		\pinlabel $\Psi$ at 12.5 0.2
		\endlabellist
		\includegraphics[width=.12\textwidth]{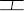}
		\vspace{-3pt}
	\end{gathered}\hspace{30pt} \longmapsto \hspace{30pt}\begin{gathered}
		\labellist
		\small
		\pinlabel $\Omega$ at -1 4.8
		\pinlabel $\Lambda \oplus \Psi$ at -2.5 0.2
		\pinlabel $\Lambda \oplus \Omega \oplus \Psi$ at 8.8 1.3
		\pinlabel $\Lambda\oplus\Omega$ at 20.2 4.8
		\pinlabel $\Psi$ at 18.5 0.2
		\endlabellist
		\includegraphics[width=.19\textwidth]{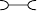}
		\vspace{-3pt}
	\end{gathered}\vspace{5pt}
\]and the equivariant pushforward and pullback of \vspace{5pt}\[
	\begin{gathered}
	 	\labellist
		\small
		\pinlabel $\Omega \oplus \Lambda$ at -2.5 4.8
		\pinlabel $\Psi$ at -.9 0.2
		\pinlabel $\Lambda$ at 4.3 2.5
		\pinlabel $\Omega$ at 7.5 5
		\pinlabel $\Omega \oplus \Lambda \oplus \Psi$ at 15.2 2.5
		\pinlabel $\Lambda \oplus \Psi$ at 8.4 -0.4
		\endlabellist
		\includegraphics[width=.12\textwidth]{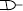}
		\vspace{-3pt}
	\end{gathered}\hspace{45pt} \longmapsto \hspace{25pt}\begin{gathered}
		\labellist
		\small
		\pinlabel $\Omega \oplus \Lambda$ at -2.5 4.8
		\pinlabel $\Psi$ at -.9 0.2
		\pinlabel $\Omega \oplus \Lambda \oplus \Psi$ at 15.2 2.5
		\endlabellist
		\includegraphics[width=.12\textwidth]{pitchfork2}
		\vspace{-3pt}
	\end{gathered}\hspace{30pt}\vspace{5pt}
\]where $\Omega,\Lambda,\Psi$ are pairwise orthogonal subspaces. 

\begin{rem}
	From now on, we abuse notation by confusing a web with its associated singular Bott--Samelson bimodule. For example, an action of an algebra on a web $\Gamma$ means an action on $B_\Gamma$ through bimodule maps. Additionally, we confuse a foam with the bimodule map it represents. 
\end{rem}


\section{Construction of $\sr{K}$}\label{sec:constructionOfK}

Fix positive integers $a,b,c,d$ for which $a + b = c + d$ and $b = \min(a,b,c,d)$. We focus on webs with boundary data $c_L = (a,b)$ and $c_R = (d,c)$. The purpose of this section is to construct $\sr{K} \coloneqq {}^b_a \sr{K}^c_d$. It is a bounded chain complex of singular Bott--Samelson bimodules with boundary data $c_L,c_R$. In section~\ref{subsec:webs_and_foams_in_K}, we introduce all of the relevant webs and foams needed to construct $\sr{K}$. In section~\ref{subsec:the_shape_of_K}, we explain how the objects and the components of the differential of $\sr{K}$ are formally modeled on the vertices and edges of the $b$-dimensional cube $[0,3]^b$. In section~\ref{subsec:the_objects_of_K}, we define the objects of $\sr{K}$, and in section~\ref{subsec:the_differential_of_K}, we define the components of the differential of $\sr{K}$.

Let $n$ be the common sum $a + b = c + d$, which is also the width of any web with boundary data $c_L,c_R$. Let $l$ be the common difference $c - b = a - d$ which is nonnegative by the requirement $b = \min(a,b,c,d)$. 


\subsection{The webs and foams in $\sr{K}$}\label{subsec:webs_and_foams_in_K}

For organizational purposes, we first summarize the webs, foams, and relevant identities before providing the definitions.
In Definition~\ref{def:websfoamsK}, we define \begin{itemize}[noitemsep]
	\item a web $V_r$ with boundary data $c_L,c_R$ for $r \in \{0,1,\ldots,b\}$, 
	\item an endomorphism $Q_t \in \Hom^{2l+2t}(V_r,V_r)$ for $t \in \{1,\ldots,b\}$, and
	\item a pair of adjoint foams $Z_{(r+1)r}\in \Hom^d(V_r,V_{r+1})$ and $Z_{r(r+1)}\in\Hom^d(V_{r+1},V_r)$ for $r \in \{0,\ldots,b-1\}$.
\end{itemize}

\begin{lem}\label{lem:relationsK}
	The webs and foams defined in Definition~\ref{def:websfoamsK} satisfy the following properties.\begin{itemize}[noitemsep]
		\item The web $V_r$ is equipped with an action of the subalgebra of the nil-Hecke algebra $\cl{H}_b$ generated by $x_1,\ldots,x_b$ and $\partial_1,\ldots,\partial_{r-1},\partial_{r+1},\ldots,\partial_{b-1}$. In particular, $s_1,\ldots,s_{r-1},s_{r+1},\ldots,s_{b-1} \in \cl{H}_b$ act on $V_r$. The stated generators act on $V_r$ by self-adjoint endomorphisms while the simple transpositions act by skew-adjoint endomorphisms.
		\item The endomorphism $Q_t$ is self-adjoint, commutes with $x_1,\ldots,x_b$, and satisfies \[
			\partial_iQ_t = \begin{cases}
				Q_t\partial_i & i \neq t-1\\
				Q_{t-1}s_t + Q_t \partial_{t-1} & i = t-1
			\end{cases}
		\]for $i \in \{1,\ldots,r-1,r+1,\ldots,b-1\}$. If $t > r$, then $Q_t$ is actually the zero endomorphism of $V_r$. 
		\item The foams $Z_{(r+1)r}$ and $Z_{r(r+1)}$ commute with $Q_t$ for $t \in \{1,\ldots,b\}$ and with $x_1,\ldots,x_b,\partial_1,\ldots,\partial_{r-1},\partial_{r+2},\ldots,\partial_{b-1}$. For $r \in \{1,\ldots,b-1\}$, both $Z_{(r+1)r}\,Z_{r(r-1)}$ and its adjoint $Z_{(r-1)r}\,Z_{r(r+1)}$ commute with $\partial_r$, and we have the identity \[
			Z_{r(r+1)}\,s_r\,Z_{(r+1)r} = Z_{r(r-1)}\,s_r\,Z_{(r-1)r}.
		\]
	\end{itemize}
\end{lem}

\begin{df}\label{def:websfoamsK}
	For $r \in \{0,\ldots,b\}$, let $V_r$ be the web given by\vspace{5pt}\[
		\begin{gathered}
			\labellist
			\pinlabel $a$ at -0.5 0.1
			\pinlabel $b$ at -0.5 7.7
			\pinlabel $c$ at 22.3 7.7
			\pinlabel $d$ at 22.4 0.1
			\pinlabel $1$ at 1.3 6
			\pinlabel $1$ at 2.85 6
			\pinlabel $\cdots$ at 4.6 5.8
			\pinlabel $1$ at 5.8 6
			\pinlabel $1$ at 13.1 1
			\pinlabel $1$ at 13.1 2.5
			\pinlabel $\vdots$ at 13.1 4.5
			\pinlabel $1$ at 13.1 5.5
			\pinlabel $1$ at 13.1 7
			\pinlabel $l+r$ at 21.2 6
			\pinlabel $a+r$ at 13.1 {-.6}
			\endlabellist
			\includegraphics[width=.45\textwidth]{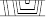}
		\end{gathered}
	\]with the following alphabets assigned to edges: \vspace{5pt}\[
		\begin{gathered}
			\labellist
			\pinlabel $\A$ at -0.5 0.1
			\pinlabel $\B$ at -0.5 7.7
			\pinlabel $\C$ at 22.3 7.7
			\pinlabel $\mathbf{D}$ at 22.4 0.1
			\pinlabel $x_1$ at 1.2 6
			\pinlabel $x_2$ at 2.7 6
			\pinlabel $\cdots$ at 4.4 5.9
			\pinlabel $x_r$ at 5.7 6
			\pinlabel $x_{r+1}$ at 13.1 1
			\pinlabel $x_{r+2}$ at 13.1 2.5
			\pinlabel $\vdots$ at 13.1 4.5
			\pinlabel $x_{b-1}$ at 13.1 5.5
			\pinlabel $x_b$ at 13.1 7
			\pinlabel $\mathbf{E}_r$ at 20.8 6
			\pinlabel $\F_r$ at 13.1 {-.7}
			\endlabellist
			\includegraphics[width=.45\textwidth]{Vgeneral}
		\end{gathered}\vspace{5pt}
	\]For $t \in \{1,\ldots,b\}$, let $Q_t$ be the endomorphism of $V_r$ given by multiplication with \[
		e_{l+t}(\C - x_t - x_{t+1} - \cdots - x_b).
	\]For $r \in \{0,\ldots,b-1\}$, define $Z_{(r+1)r}$ to be the following foam. For clarity, we only draw the portion of the foam in the region near the edges with alphabets $x_{r+1},\mathbf{E}_r,\F_r$ and $x_{r+1},\mathbf{E}_{r+1},\F_{r+1}$. Away from this region, the foam agrees with the identity foam.\[
		\begin{gathered}
			\includegraphics[width=.163\textwidth]{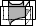}
			\vspace{-3pt}
		\end{gathered} \hspace{10pt} = \hspace{10pt} \begin{gathered}
			\includegraphics[width=.163\textwidth]{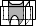}
			\\
			\includegraphics[width=.163\textwidth]{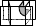}
			\vspace{-3pt}
		\end{gathered} \hspace{10pt} = \hspace{10pt} \begin{gathered}
			\includegraphics[width=.163\textwidth]{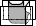}
			\\
			\includegraphics[width=.163\textwidth]{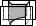}
			\vspace{-3pt}
		\end{gathered}\vspace{5pt}
	\]Let $Z_{r(r+1)}$ be the adjoint of $Z_{(r+1)r}$. Algebraic formulas for these bimodule maps are given in the proof of Lemma~\ref{lem:relationsK}. 
\end{df}

\begin{proof}[Proof of Lemma~\ref{lem:relationsK}]
	The singular Bott--Samelson bimodule associated to $V_r$ is a grading shift of the quotient of \[
		\Z[x_1,\ldots,x_b] \otimes \Sym(\A) \otimes \Sym(\B) \otimes \Sym(\C) \otimes \Sym(\mathbf{D}) \otimes \Sym(\mathbf{E}_r) \otimes \Sym(\F_r)
	\]by the ideal $I_r$ generated by the relations \[
		\begin{tikzcd}[row sep = 0pt]
			e_i(\B) = e_i(x_1 + \cdots + x_b) & e_i(\C) = e_i(\mathbf{E}_r + x_{r+1} + \cdots + x_b)\\
			e_i(\A) = e_i(\F_r - x_1 - \cdots - x_r) & e_i(\mathbf{D}) = e_i(\F_r - \mathbf{E}_r)
		\end{tikzcd}
	\]for $i \ge 1$ using the notation explained in section~\ref{subsec:symmetric_polynomials}. For $j \neq r$, the divided difference operator $\partial_j$ of $\Z[x_1,\ldots,x_b]$ preserves $I_r$ so it descends to an endomorphism of the quotient. It commutes with the actions of the elementary symmetric polynomials in the alphabets $\A,\B,\C$, and $\mathbf{D}$ so it is a bimodule endomorphism. If $i \in \{1,\ldots,r-1\}$, then the foam representing $\partial_i$ is given locally by \[
		\begin{gathered}
			\includegraphics[width=.13\textwidth]{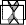}
		\end{gathered}
	\]For $i = r+1,\ldots,b-1$, it is given by a similar local foam except that the front white sheet is rotated around a vertical axis to the right, which drags along and stretches the shaded facets into an arc. In particular, the foam $\partial_i$ factors through the web obtained by merging the parallel edges with alphabets $x_i$ and $x_{i+1}$ into a single edge colored by $2$ with alphabet $\{x_i,x_{i+1}\}$. The map $\partial_i$ is the equivariant push-pull map induced by the natural projection from the Bott--Samelson variety of $V_r$ to that of this web, so it is self-adjoint. Each $x_i$ is self-adjoint as is multiplication by an elementary symmetric polynomial in an alphabet assigned to any edge. The endomorphism $s_i = \Id - \,(x_i - x_{i+1})\partial_i$ is skew-adjoint because \[
		s^*\hspace{-4pt}_i = \Id^* - \,\partial^*\hspace{-4pt}_i(x^*\hspace{-4pt}_i - x^*\hspace{-4pt}_{i+1}) = \Id - \,\partial_i x_i + \partial_i x_{i+1} = \Id - \,(\Id + \,x_{i+1}\partial_i) + (-\Id + \,x_i\partial_i) = -s_i.
	\]

	The endomorphism $Q_t$ clearly commutes with $x_1,\ldots,x_b$, and since $e_{l+t}(\C - x_t - \cdots - x_b)$ is symmetric in $x_1,\ldots,x_{t-1}$ and in $x_t,\ldots,x_b$, it follows from the Leibniz rule for $\partial_i$ that $Q_t$ commutes with $\partial_i$ whenever $i\neq t-1$. By applying $\partial_{t-1}$ to $e_{l+t}(\C - x_t - \cdots - x_b)$, we obtain \begin{align*}
		\partial_{t-1}\left(\sum_{j=0}^{l+t} (-1)^j e_{l+t-j}(\C)h_{j}(x_t,\ldots,x_b) \right) &= \sum_{j=0}^{l+t} (-1)^j e_{l+t-j}(\C)\partial_{t-1} h_{j}(x_t,\ldots,x_b)\\
		&= \sum_{j=1}^{l+t}(-1)^{j+1} e_{l+t-j}(\C)h_{j-1}(x_{t-1},x_t,\ldots,x_b)
	\end{align*}which is $e_{l+t-1}(\C - x_{t-1}-x_t-\cdots-x_b)$. The second equality uses the identity $\partial_{t-1}h_j(x_t,\ldots,x_b) = -h_{j-1}(x_{t_1},x_t,\ldots,x_b)$ for $j \ge 1$, which follows from the computation \begin{align*}
		\frac{h_j(x_t,x_{t+1},\cdots,x_b) - h_j(x_{t-1},x_{t+1},\ldots,x_b)}{x_{t-1}-x_t} &= \sum_{m=0}^j \left( \frac{x_t^m-x_{t-1}^m}{x_{t-1}-x_t} \right) h_{j-m}(x_{t+1},\ldots,x_b)\\
		&= -\sum_{m=1}^j h_{m-1}(x_{t-1},x_t)h_{j-m}(x_{t+1},\ldots,x_b).
	\end{align*}Hence, the Leibniz rule for $\partial_{t-1}$ implies that $\partial_{t-1} Q_t = Q_{t-1}s_{t-1} + Q_t \partial_{t-1}$. If $t = r + k$ for $k \ge 1$, then \begin{align*}
		e_{l+r+k}(\C - x_{r+k}-\cdots-x_b) &= e_{l+r+k}(\mathbf{E} + x_{r+1} + \cdots + x_{r+k-1})\\
		&= \sum_{j=0}^{l+r+k} e_{l+r+k-j}(\mathbf{E})e_{j}(x_{r+1},\ldots,x_{r+k-1}) = 0
	\end{align*}where the last equality follows from the observation that $e_{l+r+k-j}(\mathbf{E}) = 0$ when $j < k$ while $e_j(x_{r+1},\ldots,x_{r+k-1}) = 0$ when $j \ge k$. 

	Next, we note that $Z_{r(r+1)}$ is the ring map \[
		\begin{tikzcd}
			\Z[x_1,\ldots,x_b] \otimes \Sym(\A) \otimes \Sym(\B) \otimes \Sym(\C) \otimes \Sym(\mathbf{D}) \otimes \Sym(\mathbf{E}_{r+1}) \otimes \Sym(\F_{r+1})/I_{r+1} \ar[d]\\
			\Z[x_1,\ldots,x_b] \otimes \Sym(\A) \otimes \Sym(\B) \otimes \Sym(\C) \otimes \Sym(\mathbf{D}) \otimes \Sym(\mathbf{E}_r) \otimes \Sym(\F_r)/I_r
		\end{tikzcd}
	\]that is linear over the first five tensor factors and sends $e_i(\mathbf{E}_{r+1}) \mapsto e_i(\mathbf{E}_r + x_{r+1})$ and $e_i(\F_{r+1}) \mapsto e_i(\F_r + x_{r+1})$. Its adjoint $Z_{(r+1)r}$ is given by the map \[
		\begin{tikzcd}
			q^{2d}\,\Z[x_1,\ldots,x_b] \otimes \Sym(\A) \otimes \Sym(\B) \otimes \Sym(\C) \otimes \Sym(\mathbf{D}) \otimes \Sym(\mathbf{E}_r) \otimes \Sym(\F_r)/I_r \ar[d]\\
			\Z[x_1,\ldots,x_b] \otimes \Sym(\A) \otimes \Sym(\B) \otimes \Sym(\C) \otimes \Sym(\mathbf{D}) \otimes \Sym(\mathbf{E}_{r+1}) \otimes \Sym(\F_{r+1})/I_{r+1}
		\end{tikzcd}
	\]that sends $1 \mapsto e_d(\mathbf{D} - x_{r+1})$, is linear over the first five tensor factors, and intertwines $e_i(\mathbf{E}_r)$ with $e_i(\mathbf{E}_{r+1}-x_{r+1})$ and $e_i(\F_r)$ with $e_i(\F_{r+1}-x_{r+1})$. 
	It is straightforward to see that $Z_{r(r+1)}$ and $Z_{(r+1)r}$ commute with the stated endomorphisms. To see that $Z_{(r-1)r}\,Z_{r(r+1)}$ commutes with $\partial_r$, note that $Z_{(r-1)r}\,Z_{r(r+1)}$ sends $e_i(\mathbf{E}_{r+1}) \mapsto e_i(\mathbf{E}_{r-1} + x_{r} + x_{r+1})$ and $e_i(\F_{r+1})\mapsto e_i(\F_{r-1} + x_{r} + x_{r+1})$. These polynomials are invariant under $s_r$ so $Z_{(r-1)r}\,Z_{r(r+1)}$ commutes with $\partial_{r}$. By taking adjoints, it follows that $Z_{(r+1)r}\,Z_{r(r-1)}$ does as well. Lastly, a direct computation shows that $Z_{r(r+1)}\,s_r\,Z_{(r+1)r}$ and $Z_{r(r-1)}\,s_r\,Z_{(r-1)r}$ are both given by the map \[
		\begin{tikzcd}
			q^{2d}\,\Z[x_1,\ldots,x_b] \otimes \Sym(\A) \otimes \Sym(\B) \otimes \Sym(\C) \otimes \Sym(\mathbf{D}) \otimes \Sym(\mathbf{E}_r) \otimes \Sym(\F_r)/I_r \ar[d]\\
			\Z[x_1,\ldots,x_b] \otimes \Sym(\A) \otimes \Sym(\B) \otimes \Sym(\C) \otimes \Sym(\mathbf{D}) \otimes \Sym(\mathbf{E}_{r}) \otimes \Sym(\F_{r})/I_{r}
		\end{tikzcd}
	\]that sends $1\mapsto e_d(\mathbf{D}-x_r)$, is linear over the last six tensor factors and the actions of $x_1,\ldots,x_{r-1},x_{r+2},\ldots,x_b$, and intertwines the action of $x_r$ with $x_{r+1}$ and the action of $x_{r+1}$ with $x_r$. 
\end{proof}


\subsection{The shape of $\sr{K}$}\label{subsec:the_shape_of_K}

Consider the $b$-dimensional cube $[0,3]^b \subset \R^b$ with edges of length $3$. The standard cubulation of a $b$-dimensional cube has $2^b$ vertices, $2^{b-1}b$ edges, $2^{b-2}\binom{b}{2}$ faces, and $2^{b-k}\binom{b}{k}$ facets of dimension $k$. The cube $[0,3]^b$ can be thought of as having either this standard cubulation or the finer cubulation with $4^b$ vertices, $4^{b-1} 3b$ edges, $4^{b-2}9b$ faces, and $4^{b-k}3^k \binom{b}{k}$ facets of dimension $k$ arising from the standard cubulation of $\R^b$. The $4^b$ vertices are the integer lattice points within $[0,3]^b$. We refer to the standard cubulation as the \textit{coarse} cubulation and the finer cubulation as the \textit{fine} cubulation. See Figure~\ref{fig:cubulations}.

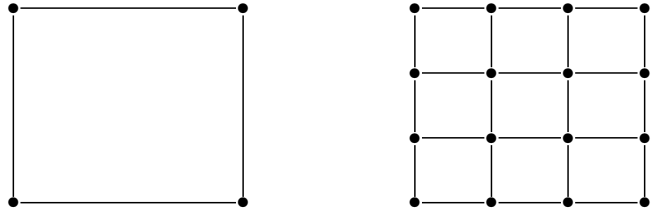
\begin{figure}[!ht]
	\centering
	\vspace{-5pt}\[
		\begin{tikzcd}[nodes={inner sep=0pt},row sep=20pt]
			\bullet \ar[rrr, no head] & \phantom{\bullet} & \phantom{\bullet} & \bullet\\
			\phantom{\bullet} & \phantom{\bullet} & \phantom{\bullet} & \phantom{\bullet}\\
			\phantom{\bullet} & \phantom{\bullet} & \phantom{\bullet} & \phantom{\bullet}\\
			\bullet \ar[rrr, no head] \ar[uuu, no head] & \phantom{\bullet} & \phantom{\bullet} & \bullet \ar[uuu, no head]
		\end{tikzcd} \qquad\qquad\qquad\begin{tikzcd}[nodes={inner sep=0pt},row sep=20pt]
			\bullet \ar[r, no head] & \bullet \ar[r, no head] & \bullet \ar[r, no head] & \bullet\\
			\bullet \ar[r, no head] \ar[u, no head] & \bullet \ar[r, no head] \ar[u, no head] & \bullet \ar[r, no head] \ar[u, no head] & \bullet \ar[u, no head]\\
			\bullet \ar[r, no head] \ar[u, no head] & \bullet \ar[r, no head] \ar[u, no head] & \bullet \ar[r, no head] \ar[u, no head] & \bullet \ar[u, no head]\\
			\bullet \ar[r, no head] \ar[u, no head] & \bullet \ar[r, no head] \ar[u, no head] & \bullet \ar[r, no head] \ar[u, no head] & \bullet \ar[u, no head]
		\end{tikzcd}\vspace{-5pt}
	\]
	\captionsetup{width=.8\linewidth}
	\caption{The coarse and fine cubulations of $[0,3]^b \subset \R^b$ for $b = 2$.}
	\label{fig:cubulations}
\end{figure}

We will define the chain complex $\sr{K} = {}^b_a\sr{K}^c_d$ in the following way. To each vertex $\varepsilon = (\varepsilon_1,\ldots,\varepsilon_b) \in [0,3]^b \cap \Z^b$ of the fine cubulation, we assign an object $V(\varepsilon)$, which is just one of the webs $V_0,\ldots,V_b$ from section~\ref{subsec:webs_and_foams_in_K} with a quantum grading shift. 
The complex $\sr{K}$ is the direct sum $\bigoplus_\varepsilon t^{-|\varepsilon|} V(\varepsilon)$ over all $4^b$ vertices $\varepsilon \in [0,3]^b \cap \Z^b$ where $|\varepsilon| \coloneqq \sum_{i=1}^b \varepsilon_i$. We then define the differential explicitly by components. The nontrivial components are precisely the ones lying along the $4^{b-1}3b$ edges of the fine cubulation. In particular, the differential decrements a coordinate of $\varepsilon$ by one. The differential squares to zero when traveling along consecutive edges in the same direction. Each of the $4^{b-2}9b$ faces of the fine cubulation yields a commutative square, which is made anti-commutative with appropriate signs added later. 
So $\sr{K}$ is a $b$-fold complex, where a $2$-fold complex is a bicomplex and a $3$-fold complex is a tricomplex. The component of the differential assigned to an edge parallel to the $i$th coordinate direction is negated when forming the total complex if the sum of the first $i-1$ coordinates is odd. 

The involutive symmetry $\iota\colon [0,3]^b \to [0,3]^b$ of the cube given by $\iota(x_1,\ldots,x_b) = (3-x_1,\ldots,3-x_b)$ will play an important role. Note that it induces an involution on the set of $k$-dimensional facets of the fine cubulation. On facets, it is a fixed-point-free involution except for the central $b$-dimensional facet. 

\begin{df}
	Two $k$-dimensional facets of the fine cubulation are \textit{dual} if they are paired by the involution. Given a vertex $\varepsilon = (\varepsilon_1,\ldots,\varepsilon_b) \in [0,3]^b \cap \Z^b$, we denote its dual vertex by $\varepsilon^* \coloneqq (3-\varepsilon_1,\ldots,3-\varepsilon_b)$.
\end{df}

\subsection{The objects of $\sr{K}$}\label{subsec:the_objects_of_K}
 
Let $\varepsilon = (\varepsilon_1,\ldots,\varepsilon_b) \in [0,3]^b \cap \Z^b$ be a vertex of the fine cubulation. Let $r(\varepsilon) \in \{0,1,\ldots,b\}$ be the number of coordinates of $\varepsilon$ that are equal to either $1$ or $2$. In symbols, we have $r(\varepsilon) = \sum_{i=1}^b \delta(\varepsilon_i - 1) + \delta(\varepsilon_i - 2)$ where $\delta\colon \Z \to \{0,1\}$ is the Dirac delta function. Note that $r(\varepsilon)$ is the dimension of the unique facet of the coarse cubulation of $[0,3]^b$ that contains $\varepsilon$ in its interior. Set \[
	V(\varepsilon) \coloneqq q^{\,G(\varepsilon)} V_{r(\varepsilon)}
\] where the grading shift function $G\colon [0,3]^b \cap \Z^b \to \Z$ has $G(0,\ldots,0) = 0$ and satisfies the following rule. 

While specifying the rule, we define a function $\upsilon$ that assigns an integer to each oriented edge of the fine cubulation of $[0,3]^b$. Fix $1 \leq i \leq b$ and $\varepsilon_1,\ldots,\varepsilon_{i-1},\varepsilon_{i+1},\ldots,\varepsilon_b \in \{0,1,2,3\}$. Let $\varepsilon^j \coloneqq (\varepsilon_1,\ldots,\varepsilon_{i-1},j,\varepsilon_{i+1},\ldots,\varepsilon_b) \in [0,3]^b \cap \Z^b$ for $j = 0,1,2,3$, and let $\varepsilon^{[0,1]},\varepsilon^{[1,2]},\varepsilon^{[2,3]}$ be the following three oriented edges of the fine cubulation \[
	\begin{tikzcd}[column sep={70pt,between origins}]
		\varepsilon^0 & \varepsilon^1 \ar[l,swap,"\varepsilon^{[0,1]}"] & \varepsilon^2 \ar[l,swap,"\varepsilon^{[1,2]}"] & \varepsilon^3 \ar[l,swap,"\varepsilon^{[2,3]}"]
	\end{tikzcd}
\]The function $G$ is required to satisfy the following equations which also define the values of $\upsilon$ on these three edges. \begin{align*}
	G(\varepsilon^0) - G(\varepsilon^1) &= \upsilon(\varepsilon^{[0,1]}) \coloneqq 2\left(\sum_{k=i+1}^b \delta(\varepsilon_k - 0) + \delta(\varepsilon_k-1)\right) - d\\
	G(\varepsilon^1) - G(\varepsilon^2) &= \upsilon(\varepsilon^{[1,2]}) \coloneqq -2\left(\sum_{k=1}^{i-1} \delta(\varepsilon_k - 1) + \delta(\varepsilon_k - 2)\right) - 2 - 2l\\
	G(\varepsilon^2) - G(\varepsilon^3) &= \upsilon(\varepsilon^{[2,3]}) \coloneqq 2\left(\sum_{k=i+1}^b \delta(\varepsilon_k - 2) + \delta(\varepsilon_k-3)\right) - d.
\end{align*}So to obtain $G(\varepsilon^0)$ from $G(\varepsilon^1)$, subtract $d$ and add twice the number of $0$'s and $1$'s among $\varepsilon_{i+1},\ldots,\varepsilon_b$. To obtain $G(\varepsilon^1)$ from $G(\varepsilon^2)$, subtract $2 + 2l$ and then subtract twice the number of $1$'s and $2$'s among $\varepsilon_{1},\ldots,\varepsilon_{i-1}$. To obtain $G(\varepsilon^2)$ from $G(\varepsilon^3)$, subtract $d$ and add twice the number of $2$'s and $3$'s among $\varepsilon_{i+1},\ldots,\varepsilon_b$. 
See Figures~\ref{fig:bequals2objects} and \ref{fig:bequals3objects} for examples. 

\begin{figure}[!ht]
	\centering
	\[
		\begin{tikzcd}[nodes={inner sep=4pt},row sep={27pt,between origins}, column sep={27pt,between origins}]
			03 \ar[dotted,r, no head] & 13 \ar[dotted,r, no head] & 23 \ar[dotted,r, no head] & 33\\
			02 \ar[dotted,r, no head] \ar[dotted,u, no head] & 12 \ar[dotted,r, no head] \ar[dotted,u, no head] & 22 \ar[dotted,r, no head] \ar[dotted,u, no head] & 32 \ar[dotted,u, no head]\\
			01 \ar[dotted,r, no head] \ar[dotted,u, no head] & 11 \ar[dotted,r, no head] \ar[dotted,u, no head] & 21 \ar[dotted,r, no head] \ar[dotted,u, no head] & 31 \ar[dotted,u, no head]\\
			00 \ar[dotted,r, no head] \ar[dotted,u, no head] & 10 \ar[dotted,r, no head] \ar[dotted,u, no head] & 20 \ar[dotted,r, no head] \ar[dotted,u, no head] & 30 \ar[dotted,u, no head]
		\end{tikzcd} \qquad\qquad \begin{tikzcd}[nodes={inner sep=4pt},row sep={27pt,between origins}, column sep={27pt,between origins}]
			0 \ar[dotted,r, no head] & 1 \ar[dotted,r, no head] & 1 \ar[dotted,r, no head] & 0\\
			1 \ar[dotted,r, no head] \ar[dotted,u, no head] & 2 \ar[dotted,r, no head] \ar[dotted,u, no head] & 2 \ar[dotted,r, no head] \ar[dotted,u, no head] & 1 \ar[dotted,u, no head]\\
			1 \ar[dotted,r, no head] \ar[dotted,u, no head] & 2 \ar[dotted,r, no head] \ar[dotted,u, no head] & 2 \ar[dotted,r, no head] \ar[dotted,u, no head] & 1 \ar[dotted,u, no head]\\
			0 \ar[dotted,r, no head] \ar[dotted,u, no head] & 1 \ar[dotted,r, no head] \ar[dotted,u, no head] & 1 \ar[dotted,r, no head] \ar[dotted,u, no head] & 0 \ar[dotted,u, no head]
		\end{tikzcd} \qquad\qquad \begin{tikzcd}[nodes={inner sep=4pt},row sep={27pt,between origins}, column sep={27pt,between origins}]
			6 \ar[dotted,r, no head] & 8 \ar[dotted,r, no head] & 10 \ar[dotted,r, no head] & 10\\
			4 \ar[dotted,r, no head] \ar[dotted,u, no head] & 6 \ar[dotted,r, no head] \ar[dotted,u, no head] & 8 \ar[dotted,r, no head] \ar[dotted,u, no head] & 8 \ar[dotted,u, no head]\\
			2 \ar[dotted,r, no head] \ar[dotted,u, no head] & 2 \ar[dotted,r, no head] \ar[dotted,u, no head] & 4 \ar[dotted,r, no head] \ar[dotted,u, no head] & 6 \ar[dotted,u, no head]\\
			0 \ar[dotted,r, no head] \ar[dotted,u, no head] & 0 \ar[dotted,r, no head] \ar[dotted,u, no head] & 2 \ar[dotted,r, no head] \ar[dotted,u, no head] & 4 \ar[dotted,u, no head]
		\end{tikzcd}
		\vspace{-10pt}
	\]
	\captionsetup{width=.8\linewidth}
	\caption{On the left are the vertices $\varepsilon \in [0,3]^b \cap \Z^b$ for $b = 2$. In the middle are the values of $r(\varepsilon) \in \{0,1,2\}$. On the right are the values of $G(\varepsilon) \in \Z$ for $a = b = c = d = 2$.}
	\label{fig:bequals2objects}
\end{figure}

\begin{figure}[!ht]
	\[
		\begin{tikzcd}[nodes={inner sep=3pt},column sep={29pt,between origins},row sep={4pt,between origins}]
			& & & 033 \ar[dotted,no head, dddl] \ar[dotted,no head, dddd] & & & & 133 \ar[dotted,no head, dddl] \ar[dotted,no head, dddd] & & & & 233 \ar[dotted,no head, dddl] \ar[dotted,no head, dddd] & & & & 333 \ar[dotted,no head, dddl] \ar[dotted,no head, dddd]\\
			& & & & & & & & & & & & & & &\\
			& & & & & & & & & & & & & & &\\
			& & 023 \ar[dotted,no head, dddl] \ar[dotted,no head, dddd] & & & & 123 \ar[dotted,no head, dddl] \ar[dotted,no head, dddd] & & & & 223 \ar[dotted,no head, dddl] \ar[dotted,no head, dddd] & & & & 323 \ar[dotted,no head, dddl] \ar[dotted,no head, dddd] &\\
			& & & 032 \ar[dotted,no head, dddl] \ar[dotted,no head, dddd] & & & & 132 \ar[dotted,no head, dddl] \ar[dotted,no head, dddd] & & & & 232 \ar[dotted,no head, dddl] \ar[dotted,no head, dddd] & & & & 332 \ar[dotted,no head, dddl] \ar[dotted,no head, dddd]\\
			& & & & & & & & & & & & & & &\\
			& 013 \ar[dotted,no head, dddl] \ar[dotted,no head, dddd] & & & & 113 \ar[dotted,no head, dddl] \ar[dotted,no head, dddd] & & & & 213 \ar[dotted,no head, dddl] \ar[dotted,no head, dddd] & & & & 313 \ar[dotted,no head, dddl] \ar[dotted,no head, dddd] & &\\
			& & 022 \ar[dotted,no head, dddl] \ar[dotted,no head, dddd] & & & & 122 \ar[dotted,no head, dddl] \ar[dotted,no head, dddd] & & & & 222 \ar[dotted,no head, dddl] \ar[dotted,no head, dddd] & & & & 322 \ar[dotted,no head, dddl] \ar[dotted,no head, dddd] &\\
			& & & 031 \ar[dotted,no head, dddl] \ar[dotted,no head, dddd] & & & & 131 \ar[dotted,no head, dddl] \ar[dotted,no head, dddd] & & & & 231 \ar[dotted,no head, dddl] \ar[dotted,no head, dddd] & & & & 331 \ar[dotted,no head, dddl] \ar[dotted,no head, dddd]\\
			003 \ar[dotted,no head, dddd] & & & & 103 \ar[dotted,no head, dddd] & & & & 203 \ar[dotted,no head, dddd] & & & & 303 \ar[dotted,no head, dddd] & & &\\
			& 012 \ar[dotted,no head, dddl] \ar[dotted,no head, dddd] & & & & 112 \ar[dotted,no head, dddl] \ar[dotted,no head, dddd] & & & & 212 \ar[dotted,no head, dddl] \ar[dotted,no head, dddd] & & & & 312 \ar[dotted,no head, dddl] \ar[dotted,no head, dddd] & &\\
			& & 021 \ar[dotted,no head, dddl] \ar[dotted,no head, dddd] & & & & 121 \ar[dotted,no head, dddl] \ar[dotted,no head, dddd] & & & & 221 \ar[dotted,no head, dddl] \ar[dotted,no head, dddd] & & & & 321 \ar[dotted,no head, dddl] \ar[dotted,no head, dddd] &\\
			& & & 030 \ar[dotted,no head, dddl] & & & & 130 \ar[dotted,no head, dddl] & & & & 230 \ar[dotted,no head, dddl] & & & & 330 \ar[dotted,no head, dddl]\\
			002 \ar[dotted,no head, dddd] & & & & 102 \ar[dotted,no head, dddd] & & & & 202 \ar[dotted,no head, dddd] & & & & 302 \ar[dotted,no head, dddd] & & &\\
			& 011 \ar[dotted,no head, dddl] \ar[dotted,no head, dddd] & & & & 111 \ar[dotted,no head, dddl] \ar[dotted,no head, dddd] & & & & 211 \ar[dotted,no head, dddl] \ar[dotted,no head, dddd] & & & & 311 \ar[dotted,no head, dddl] \ar[dotted,no head, dddd] & &\\
			& & 020 \ar[dotted,no head, dddl] & & & & 120 \ar[dotted,no head, dddl] & & & & 220 \ar[dotted,no head, dddl] & & & & 320 \ar[dotted,no head, dddl] &\\
			& & & & & & & & & & & & & & &\\
			001 \ar[dotted,no head, dddd] & & & & 101 \ar[dotted,no head, dddd] & & & & 201 \ar[dotted,no head, dddd] & & & & 301 \ar[dotted,no head, dddd] & & &\\
			& 010 \ar[dotted,no head, dddl] & & & & 110 \ar[dotted,no head, dddl] & & & & 210 \ar[dotted,no head, dddl] & & & & 310 \ar[dotted,no head, dddl] & &\\
			& & & & & & & & & & & & & & &\\
			& & & & & & & & & & & & & & &\\
			000 & & & & 100 & & & & 200 & & & & 300 & & &
		\end{tikzcd}
	\]\vspace{-10pt}
	\[
		\begin{tikzcd}[nodes={inner sep=3pt},column sep={29pt,between origins},row sep={4pt,between origins}]
			& & & 0 \ar[dotted,no head, dddl] \ar[dotted,no head, dddd] & & & & 1 \ar[dotted,no head, dddl] \ar[dotted,no head, dddd] & & & & 1 \ar[dotted,no head, dddl] \ar[dotted,no head, dddd] & & & & 0 \ar[dotted,no head, dddl] \ar[dotted,no head, dddd]\\
			& & & & & & & & & & & & & & &\\
			& & & & & & & & & & & & & & &\\
			& & 1 \ar[dotted,no head, dddl] \ar[dotted,no head, dddd] & & & & 2 \ar[dotted,no head, dddl] \ar[dotted,no head, dddd] & & & & 2 \ar[dotted,no head, dddl] \ar[dotted,no head, dddd] & & & & 1 \ar[dotted,no head, dddl] \ar[dotted,no head, dddd] &\\
			& & & 1 \ar[dotted,no head, dddl] \ar[dotted,no head, dddd] & & & & 2 \ar[dotted,no head, dddl] \ar[dotted,no head, dddd] & & & & 2 \ar[dotted,no head, dddl] \ar[dotted,no head, dddd] & & & & 1 \ar[dotted,no head, dddl] \ar[dotted,no head, dddd]\\
			& & & & & & & & & & & & & & &\\
			& 1 \ar[dotted,no head, dddl] \ar[dotted,no head, dddd] & & & & 2 \ar[dotted,no head, dddl] \ar[dotted,no head, dddd] & & & & 2 \ar[dotted,no head, dddl] \ar[dotted,no head, dddd] & & & & 1 \ar[dotted,no head, dddl] \ar[dotted,no head, dddd] & &\\
			& & 2 \ar[dotted,no head, dddl] \ar[dotted,no head, dddd] & & & & 3 \ar[dotted,no head, dddl] \ar[dotted,no head, dddd] & & & & 3 \ar[dotted,no head, dddl] \ar[dotted,no head, dddd] & & & & 2 \ar[dotted,no head, dddl] \ar[dotted,no head, dddd] &\\
			& & & 1 \ar[dotted,no head, dddl] \ar[dotted,no head, dddd] & & & & 2 \ar[dotted,no head, dddl] \ar[dotted,no head, dddd] & & & & 2 \ar[dotted,no head, dddl] \ar[dotted,no head, dddd] & & & & 1 \ar[dotted,no head, dddl] \ar[dotted,no head, dddd]\\
			0 \ar[dotted,no head, dddd] & & & & 1 \ar[dotted,no head, dddd] & & & & 1 \ar[dotted,no head, dddd] & & & & 0 \ar[dotted,no head, dddd] & & &\\
			& 2 \ar[dotted,no head, dddl] \ar[dotted,no head, dddd] & & & & 3 \ar[dotted,no head, dddl] \ar[dotted,no head, dddd] & & & & 3 \ar[dotted,no head, dddl] \ar[dotted,no head, dddd] & & & & 2 \ar[dotted,no head, dddl] \ar[dotted,no head, dddd] & &\\
			& & 2 \ar[dotted,no head, dddl] \ar[dotted,no head, dddd] & & & & 3 \ar[dotted,no head, dddl] \ar[dotted,no head, dddd] & & & & 3 \ar[dotted,no head, dddl] \ar[dotted,no head, dddd] & & & & 2 \ar[dotted,no head, dddl] \ar[dotted,no head, dddd] &\\
			& & & 0 \ar[dotted,no head, dddl] & & & & 1 \ar[dotted,no head, dddl] & & & & 1 \ar[dotted,no head, dddl] & & & & 0 \ar[dotted,no head, dddl]\\
			1 \ar[dotted,no head, dddd] & & & & 2 \ar[dotted,no head, dddd] & & & & 2 \ar[dotted,no head, dddd] & & & & 1 \ar[dotted,no head, dddd] & & &\\
			& 2 \ar[dotted,no head, dddl] \ar[dotted,no head, dddd] & & & & 3 \ar[dotted,no head, dddl] \ar[dotted,no head, dddd] & & & & 3 \ar[dotted,no head, dddl] \ar[dotted,no head, dddd] & & & & 2 \ar[dotted,no head, dddl] \ar[dotted,no head, dddd] & &\\
			& & 1 \ar[dotted,no head, dddl] & & & & 2 \ar[dotted,no head, dddl] & & & & 2 \ar[dotted,no head, dddl] & & & & 1 \ar[dotted,no head, dddl] &\\
			& & & & & & & & & & & & & & &\\
			1 \ar[dotted,no head, dddd] & & & & 2 \ar[dotted,no head, dddd] & & & & 2 \ar[dotted,no head, dddd] & & & & 1 \ar[dotted,no head, dddd] & & &\\
			& 1 \ar[dotted,no head, dddl] & & & & 2 \ar[dotted,no head, dddl] & & & & 2 \ar[dotted,no head, dddl] & & & & 1 \ar[dotted,no head, dddl] & &\\
			& & & & & & & & & & & & & & &\\
			& & & & & & & & & & & & & & &\\
			0 & & & & 1 & & & & 1 & & & & 0 & & &
		\end{tikzcd}
	\]\vspace{-10pt}
	\[
		\begin{tikzcd}[nodes={inner sep=3pt},column sep={29pt,between origins},row sep={4pt,between origins}]
			& & & 14 \ar[dotted,no head, dddl] \ar[dotted,no head, dddd] & & & & 17 \ar[dotted,no head, dddl] \ar[dotted,no head, dddd] & & & & 19 \ar[dotted,no head, dddl] \ar[dotted,no head, dddd] & & & & 18 \ar[dotted,no head, dddl] \ar[dotted,no head, dddd]\\
			& & & & & & & & & & & & & & &\\
			& & & & & & & & & & & & & & &\\
			& & 13 \ar[dotted,no head, dddl] \ar[dotted,no head, dddd] & & & & 16 \ar[dotted,no head, dddl] \ar[dotted,no head, dddd] & & & & 18 \ar[dotted,no head, dddl] \ar[dotted,no head, dddd] & & & & 17 \ar[dotted,no head, dddl] \ar[dotted,no head, dddd] &\\
			& & & 11 \ar[dotted,no head, dddl] \ar[dotted,no head, dddd] & & & & 14 \ar[dotted,no head, dddl] \ar[dotted,no head, dddd] & & & & 16 \ar[dotted,no head, dddl] \ar[dotted,no head, dddd] & & & & 15 \ar[dotted,no head, dddl] \ar[dotted,no head, dddd]\\
			& & & & & & & & & & & & & & &\\
			& 11 \ar[dotted,no head, dddl] \ar[dotted,no head, dddd] & & & & 12 \ar[dotted,no head, dddl] \ar[dotted,no head, dddd] & & & & 14 \ar[dotted,no head, dddl] \ar[dotted,no head, dddd] & & & & 15 \ar[dotted,no head, dddl] \ar[dotted,no head, dddd] & &\\
			& & 10 \ar[dotted,no head, dddl] \ar[dotted,no head, dddd] & & & & 13 \ar[dotted,no head, dddl] \ar[dotted,no head, dddd] & & & & 15 \ar[dotted,no head, dddl] \ar[dotted,no head, dddd] & & & & 14 \ar[dotted,no head, dddl] \ar[dotted,no head, dddd] &\\
			& & & 9 \ar[dotted,no head, dddl] \ar[dotted,no head, dddd] & & & & 10 \ar[dotted,no head, dddl] \ar[dotted,no head, dddd] & & & & 12 \ar[dotted,no head, dddl] \ar[dotted,no head, dddd] & & & & 13 \ar[dotted,no head, dddl] \ar[dotted,no head, dddd]\\
			8 \ar[dotted,no head, dddd] & & & & 9 \ar[dotted,no head, dddd] & & & & 11 \ar[dotted,no head, dddd] & & & & 12 \ar[dotted,no head, dddd] & & &\\
			& 8 \ar[dotted,no head, dddl] \ar[dotted,no head, dddd] & & & & 9 \ar[dotted,no head, dddl] \ar[dotted,no head, dddd] & & & & 11 \ar[dotted,no head, dddl] \ar[dotted,no head, dddd] & & & & 12 \ar[dotted,no head, dddl] \ar[dotted,no head, dddd] & &\\
			& & 6 \ar[dotted,no head, dddl] \ar[dotted,no head, dddd] & & & & 7 \ar[dotted,no head, dddl] \ar[dotted,no head, dddd] & & & & 9 \ar[dotted,no head, dddl] \ar[dotted,no head, dddd] & & & & 10 \ar[dotted,no head, dddl] \ar[dotted,no head, dddd] &\\
			& & & 6 \ar[dotted,no head, dddl] & & & & 7 \ar[dotted,no head, dddl] & & & & 9 \ar[dotted,no head, dddl] & & & & 10 \ar[dotted,no head, dddl]\\
			5 \ar[dotted,no head, dddd] & & & & 6 \ar[dotted,no head, dddd] & & & & 8 \ar[dotted,no head, dddd] & & & & 9 \ar[dotted,no head, dddd] & & &\\
			& 4 \ar[dotted,no head, dddl] \ar[dotted,no head, dddd] & & & & 3 \ar[dotted,no head, dddl] \ar[dotted,no head, dddd] & & & & 5 \ar[dotted,no head, dddl] \ar[dotted,no head, dddd] & & & & 8 \ar[dotted,no head, dddl] \ar[dotted,no head, dddd] & &\\
			& & 3 \ar[dotted,no head, dddl] & & & & 4 \ar[dotted,no head, dddl] & & & & 6 \ar[dotted,no head, dddl] & & & & 7 \ar[dotted,no head, dddl] &\\
			& & & & & & & & & & & & & & &\\
			3 \ar[dotted,no head, dddd] & & & & 2 \ar[dotted,no head, dddd] & & & & 4 \ar[dotted,no head, dddd] & & & & 7 \ar[dotted,no head, dddd] & & &\\
			& 1 \ar[dotted,no head, dddl] & & & & 0 \ar[dotted,no head, dddl] & & & & 2 \ar[dotted,no head, dddl] & & & & 5 \ar[dotted,no head, dddl] & &\\
			& & & & & & & & & & & & & & &\\
			& & & & & & & & & & & & & & &\\
			0 & & & & {-1} & & & & 1 & & & & 4 & & &
		\end{tikzcd}
	\]
	\captionsetup{width=.8\linewidth}
	\caption{On the top are the vertices $\varepsilon \in [0,3]^b \cap \Z^b$ for $b = 3$. In the middle are the values of $r(\varepsilon) \in \{0,1,2,3\}$. On the bottom are the values of $G(\varepsilon) \in \Z$ for $a = b = c = d = 3$.}
	\label{fig:bequals3objects}
\end{figure}
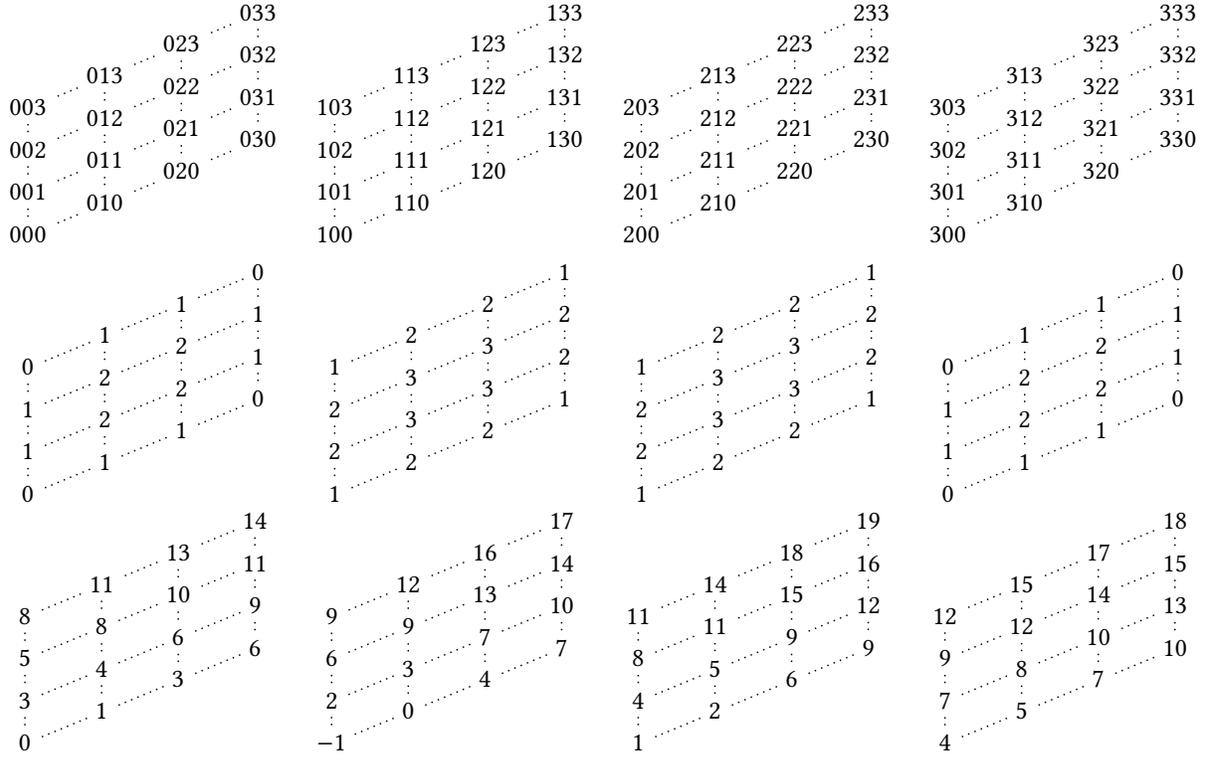

\begin{lem}\label{lem:gradingFunctionK}
	There is a unique function $G\colon [0,3]^b \cap \Z^b \to \Z$ for which $G(0,\ldots,0) = 0$ that satisfies the above rule. Furthermore, this function has the property that for any $\varepsilon \in [0,3]^b \cap \Z^b$,\[
		\frac{G(\varepsilon) + G(\varepsilon^*)}2 = \binom{a+2}{2} - \binom{a-b+2}{2}.
	\]
\end{lem}
\begin{proof}
	We view $\upsilon$ as a $1$-cochain on the cube. We must show that $\upsilon$ is a coboundary, and to do so, it suffices to show that it is coclosed. Fix $1 \leq i_1 < i_2 \leq b$ and $\varepsilon_1,\ldots,\varepsilon_{i_1 - 1}, \varepsilon_{i_1+1},\ldots,\varepsilon_{i_2-1},\varepsilon_{i_2+1},\ldots,\varepsilon_b \in \{0,1,2,3\}$. For $j_1,j_2 \in \{0,1,2,3\}$, let $\varepsilon^{j_1,j_2} \coloneqq (\varepsilon_1,\ldots,\varepsilon_{i_1 - 1}, j_1,\varepsilon_{i_1+1},\ldots,\varepsilon_{i_2-1},j_2,\varepsilon_{i_2+1},\ldots,\varepsilon_b)$. For $j_1,j_2 \in \{0,1,2\}$, consider the face of the fine cubulation with the following vertices and oriented edges. \vspace{-3pt}\[
		\begin{tikzcd}[column sep={100pt,between origins},row sep={35pt,between origins}]
			\varepsilon^{j_1,j_2+1} \ar[d,swap,"\varepsilon^{j_1,[j_2,j_2+1]}"] & \varepsilon^{j_1+1,j_2+1} \ar[l,swap,"\varepsilon^{[j_1,j_1+1],j_2+1}"] \ar[d,swap,"\varepsilon^{j_1+1,[j_2,j_2+1]}"]\\
			\varepsilon^{j_1,j_2} & \varepsilon^{j_1+1,j_2} \ar[l,swap,"\varepsilon^{[j_1,j_1+1],j_2}"]
		\end{tikzcd}
	\]We must verify that $\upsilon(\varepsilon^{j_1,[j_2,j_2+1]}) + \upsilon(\varepsilon^{[j_1,j_1+1],j_2+1}) = \upsilon(\varepsilon^{[j_1,j_1+1],j_2}) + \upsilon(\varepsilon^{j_1+1,[j_2,j_2+1]})$. Observe that the contributions from the fixed coordinates $\varepsilon_k$ for $k \in \{1,\ldots,b\}\setminus \{i_1,i_2\}$ agree so we may assume that $i_1 = 1$, $i_2 = 2$, and $b = 2$. Similarly, subtracting constants in different orders commute so we may set $l$ and $d$ to be any fixed constants we like, so we set $d = 0$ and $l = -1$. We are reduced to checking that the following $1$-cochain on $[0,3]^2$ is coclosed, which is easy to see. \[
		\begin{tikzcd}[nodes={inner sep=1pt},row sep={30pt,between origins}, column sep={35pt,between origins}]
			\varepsilon^{03} \ar[d,swap,"0"] & \varepsilon^{13} \ar[d,swap,"0"] \ar[l,swap,"0"] & \varepsilon^{23} \ar[d,swap,"0"] \ar[l,swap,"0"] & \varepsilon^{33} \ar[d,swap,"0"] \ar[l,swap,"2"]\\
			\varepsilon^{02} \ar[d,swap,"0"] & \varepsilon^{12} \ar[d,swap,"-2"] \ar[l,swap,"0"] & \varepsilon^{22} \ar[d,swap,"-2"] \ar[l,swap,"0"] & \varepsilon^{32} \ar[d,swap,"0"]\ar[l,swap,"2"]\\
			\varepsilon^{01} \ar[d,swap,"0"] & \varepsilon^{11} \ar[d,swap,"0"] \ar[l,swap,"2"] & \varepsilon^{21} \ar[d,swap,"0"] \ar[l,swap,"0"] & \varepsilon^{31} \ar[d,swap,"0"]\ar[l,swap,"0"]\\
			\varepsilon^{00} & \varepsilon^{10} \ar[l,swap,"2"] & \varepsilon^{20} \ar[l,swap,"0"] & \varepsilon^{30} \ar[l,swap,"0"]
		\end{tikzcd}
	\]Hence $\upsilon$ is a coboundary, and any two functions on vertices having $\upsilon$ as their coboundary differ by an overall constant. 

	Let $G\colon [0,3]^b \cap \Z^b \to \Z$ the unique function with $G(0,\ldots,0) = 0$ having $\upsilon$ as its coboundary. Define another function $F\colon [0,3]^b \cap \Z^b \to \Z$ by $F(\varepsilon) \coloneqq -G(\varepsilon^*)$. Then the coboundary of $F$ is also $\upsilon$. Indeed \begin{align*}
		F(\varepsilon^0) - F(\varepsilon^1) &= G((\varepsilon^1)^*) - G((\varepsilon^0)^*) = 2\left(\sum_{k=i+1}^b \delta((3-\varepsilon_k)-2) + \delta((3-\varepsilon_k)-3) \right) - d = \upsilon(\varepsilon^{[0,1]})\\
		F(\varepsilon^1) - F(\varepsilon^2) &= G((\varepsilon^2)^*) - G((\varepsilon^1)^*) = -2\left(\sum_{k=1}^{i-1} \delta((3-\varepsilon_k)-1) +\delta((3-\varepsilon_k)-2) \right) - 2 - 2l = \upsilon(\varepsilon^{[1,2]})\\
		F(\varepsilon^2) - F(\varepsilon^3) &= G((\varepsilon^3)^*) - G((\varepsilon^2)^*) = 2\left(\sum_{k=i+1}^b \delta((3-\varepsilon_k)-0) + \delta((3-\varepsilon_k)-1) \right) - d  = \upsilon(\varepsilon^{[2,3]}).
	\end{align*}It follows that $G - F$ is constant. Hence $G(\varepsilon) + G(\varepsilon^*) = G(\varepsilon) - F(\varepsilon)$ is independent of $\varepsilon$, and in particular \[
		\frac{G(\varepsilon) + G(\varepsilon^*)}{2} = \frac{G(0,\ldots,0) + G(3,\ldots,3)}{2} = \frac{G(3,\ldots,3)}{2}.
	\]We now compute $G(3,\ldots,3)$. Fix $1 \leq i \leq b$ and set $\varepsilon_1 = \cdots = \varepsilon_{i-1} = 3$ and $\varepsilon_{i+1} = \cdots = \varepsilon_b = 0$. Then \[
		G(\varepsilon^0) - G(\varepsilon^1) = 2(b-i) - d \qquad G(\varepsilon^1) - G(\varepsilon^2) = - 2 - 2l \qquad G(\varepsilon^2) - G(\varepsilon^3) = -d.
	\]Hence $G(\varepsilon^0) - G(\varepsilon^3) = -2(a-b+1+i)$ using the identity $l = a - d$, so \[
		\frac{G(3,\ldots,3)}{2} = -\frac12(G(0,\ldots,0) - G(3,\ldots,3)) = \sum_{i=1}^b (a - b + 1 + i) = \binom{a+2}{2} - \binom{a-b+2}{2}.\qedhere
	\]
\end{proof}

\subsection{The differential of $\sr{K}$}\label{subsec:the_differential_of_K}

To define the differential of $\sr{K}$, we must first introduce some notation. A \emph{descending string} $\alpha$ is defined to be a finite string in the symbols $\{\partial,s\}$ where each symbol in the string is additionally given a subscript. The subscripts are positive integers that are required to sequentially decrement from left to right. A concrete example is $\alpha = \partial_8\,\partial_7\,\partial_6\,s_5\,\partial_4\,s_3$. An \emph{ascending string} $\alpha^*$ is defined to be a finite string in the symbols $\{\partial^*, s^*\}$ equipped with positive integral subscripts that sequentially increment from left to right. An example is $\alpha^* = s^*\hspace{-4pt}_3 \,\partial^*\hspace{-4pt}_4\,s^*\hspace{-4pt}_5\,\partial^*\hspace{-4pt}_6\,\partial^*\hspace{-4pt}_7\,\partial^*\hspace{-4pt}_8$. 

Given a descending string $\alpha$, we make the following definitions. \begin{itemize}[noitemsep]
	\item Its \emph{opposite string} $\hat{\alpha}$ is obtained by replacing each $\partial$ with $s$ and each $s$ with $\partial$ while keeping the subscripts the same. The opposite string to $\alpha = \partial_8\,\partial_7\,\partial_6\,s_5\,\partial_4\,s_3$ is $\hat{\alpha} = s_8\, s_7\, s_6\, \partial_5\, s_4\,\partial_3$. 
	\item Its \emph{adjoint string} $\alpha^*$ is obtained by reversing the order of the sequence of symbols with their subscripts while also adding a superscript ${}^*$ to each symbol. The adjoint string to $\alpha = \partial_8\,\partial_7\,\partial_6\,s_5\,\partial_4\,s_3$ is $\alpha^* = s^*\hspace{-4pt}_3 \,\partial^*\hspace{-4pt}_4\,s^*\hspace{-4pt}_5\,\partial^*\hspace{-4pt}_6\,\partial^*\hspace{-4pt}_7\,\partial^*\hspace{-4pt}_8$. 
\end{itemize}Forming the adjoint string defines a bijection between descending and ascending strings, so we may uniquely denote any ascending string by $\alpha^*$ for a descending string $\alpha$. The operations of forming the opposite and adjoint strings to a given ascending string are defined in the natural way so that the operations commute and are involutions. The \textit{largest} and \textit{smallest subscripts} of a nonempty descending string $\alpha$ are its first and last subscripts, respectively, while the \emph{largest} and \emph{smallest subscripts} of a nonempty ascending string $\alpha^*$ are its last and first subscripts, respectively. Technically, for each nonnegative integer $t$, we have an empty sequence, viewed as both ascending and descending, that is defined to have smallest subscript $t+1$ and largest subscript $t$. 

Now fix $i \in \{1,\ldots,b\}$ and $\varepsilon_1,\ldots,\varepsilon_{i-1},\varepsilon_{i+1},\ldots,\varepsilon_b\in\{0,1,2,3\}$. Let $\varepsilon^j \coloneqq (\varepsilon_1,\ldots,\varepsilon_{i-1},j,\varepsilon_{i+1},\ldots,\varepsilon_b)$ for $j = 0,1,2,3$. We define the three components of the differential of $\sr{K}$\[
	\begin{tikzcd}[column sep=60pt]
		V({\varepsilon^{0}}) & V({\varepsilon^{1}}) \ar[l,swap,"\textstyle \phi"] & V({\varepsilon^{2}}) \ar[l,swap,"\textstyle \psi"] & V({\varepsilon^{3}}) \ar[l,swap,"\textstyle \chi"]
	\end{tikzcd}
\]First set $r\coloneqq r(\varepsilon^0) + 1 = r(\varepsilon^1) = r(\varepsilon^2) = r(\varepsilon^3) + 1$. Then define \[
	\phi = \alpha \, Z_{(r-1)r}\, \beta \qquad\qquad \psi = \beta^* \, Q_r \, \hat{\beta} \qquad\qquad \chi = \hat{\beta}^* \, Z_{r(r-1)} \, \hat{\alpha}^*
\]where $\alpha$ is a descending string and $\beta^*$ is an ascending string, defined in the following ways. The smallest subscript of $\alpha$ is declared to be $r$, and its string of symbols is obtained from the sequence $\varepsilon_{i+1},\ldots,\varepsilon_b$ of numbers by deleting the $1$'s and $2$'s and replacing $0$ by $\partial$ and $3$ by $s$. The largest subscript of $\beta^*$ is declared to be $r - 1$, and its string of symbols is obtained from the sequence $\varepsilon_{i+1},\ldots,\varepsilon_b$ of numbers by deleting the $0$'s and $3$'s and replacing $1$ by $\partial^*$ and $2$ by $s^*$. 

\begin{example}
	Let $\varepsilon^j = (1,j,2,0,2,3,1)$. First, we compute $r = r(\varepsilon^1) = 5$ by counting the number of $1$'s and $2$'s in ``$1120231$''. The descending string $\alpha$ is obtained by considering the tail sequence ``$20231$'', erasing the $1$'s and $2$'s to obtain ``$03$'', replacing $0$ by $\partial$ and $3$ by $s$ to obtain $\partial\,s$, and filling in the descending subscripts with smallest subscript $5$ to obtain $\alpha = \partial_6 \, s_5$. The ascending string $\beta^*$ is obtained by considering the tail sequence ``$20231$'', erasing the $0$'s and $3$'s to obtain ``$221$'', replacing $1$ by $\partial^*$ and $2$ by $s^*$ to obtain $s^*\,s^*\, \partial^*$, and filling in the ascending subscripts with largest subscript $4$ to obtain $\beta^* = s^*\hspace{-4pt}_2\: s^*\hspace{-4pt}_3\: \partial^*\hspace{-4pt}_4$. With $\alpha$ and $\beta^*$ in hand, we then have \[
		\phi = \partial_6 \, s_5 \, Z_{45} \, \partial_4 \, s_3 \, s_2 \qquad\qquad \psi = s^*\hspace{-4pt}_2\: s^*\hspace{-4pt}_3\: \partial^*\hspace{-4pt}_4\: Q_5 \, s_4 \, \partial_3 \, \partial_2 \qquad\qquad \chi = \partial^*\hspace{-4pt}_2\: \partial^*\hspace{-4pt}_3\: s^*\hspace{-4pt}_4 \: Z_{54} \, \partial^*\hspace{-4pt}_5 \: s^*\hspace{-4pt}_6.
	\]
\end{example}

\begin{example}
	See the introduction for the components of the differential of $\sr{K}$ in the cases where all four numbers $a,b,c,d$ are equal to $1$ or equal to $2$. In the case that all four are equal to $3$, the components of the differential of $\sr{K}$ are shown in Figure~\ref{fig:finitetricomplex}. 
\end{example}

\begin{figure}
	\vspace{-10pt}\[
		\begin{tikzcd}[nodes={inner sep=0pt},column sep={31pt,between origins},row sep={13pt,between origins}]
			& & & {}^{14}V_0 \ar[dddl,"\textstyle Z_{10} \partial^*\hspace{-4pt}{}_1" {xshift=5pt},sloped] \ar[dddd,"\textstyle Z_{10}" {xshift=1pt},swap] &[-10pt] & & & {}^{17}V_1 \ar[dddl,"\textstyle Z_{21} \partial^*\hspace{-4pt}{}_2" {xshift=5pt},sloped] \ar[dddd,"\textstyle Z_{21}" {xshift=1pt},swap] &[-10pt] & & & {}^{19}V_1 \ar[dddl,"\textstyle Z_{21} \partial^*\hspace{-4pt}{}_2" {xshift=5pt},sloped] \ar[dddd,"\textstyle Z_{21}" {xshift=1pt},swap] &[-10pt] & & & {}^{18}V_0 \ar[dddl,"\textstyle Z_{10} \partial^*\hspace{-4pt}{}_1" {xshift=5pt},sloped] \ar[dddd,"\textstyle Z_{10}" {xshift=1pt},swap]\\
			& & & & & & & & & & & & & & &\\
			& & & & & & & & & & & & & & &\\
			& & {}^{13}V_1 \ar[dddl,"\textstyle Q_1" {xshift=5pt},sloped] \ar[dddd,"\textstyle Z_{21}" {xshift=1pt},swap] & & & & {}^{16}V_2 \ar[dddl,"\textstyle Q_2" {xshift=5pt},sloped] \ar[dddd,"\textstyle Z_{32}" {xshift=1pt},swap] & & & & {}^{18}V_2 \ar[dddl,"\textstyle Q_2" {xshift=5pt},sloped] \ar[dddd,"\textstyle Z_{32}" {xshift=1pt},swap] & & & & {}^{17}V_1 \ar[dddl,"\textstyle Q_1" {xshift=5pt},sloped] \ar[dddd,"\textstyle Z_{21}" {xshift=1pt},swap] &\\
			& & & {}^{11}V_1 \ar[dddl,"\textstyle \partial^*\hspace{-4pt}{}_1 Z_{21}" {xshift=5pt},sloped] \ar[dddd,"\textstyle Q_1" {xshift=1pt},swap] & & & & {}^{14}V_2 \ar[dddl,"\textstyle \partial^*\hspace{-4pt}{}_2 Z_{32}" {xshift=5pt},sloped] \ar[dddd,"\textstyle Q_2" {xshift=1pt},swap] & & & & {}^{16}V_2 \ar[dddl,"\textstyle \partial^*\hspace{-4pt}{}_2 Z_{32}" {xshift=5pt},sloped] \ar[dddd,"\textstyle Q_2" {xshift=1pt},swap] & & & & {}^{15}V_1 \ar[dddl,"\textstyle \partial^*\hspace{-4pt}{}_1 Z_{21}" {xshift=5pt},sloped] \ar[dddd,"\textstyle Q_1" {xshift=1pt},swap]\\
			& & & & & & & & & & & & & & &\\
			& {}^{11}V_1 \ar[dddl,"\textstyle s_1Z_{01}" {xshift=5pt},sloped] \ar[dddd,"\textstyle Z_{21}" {xshift=1pt},swap] & & & & {}^{12}V_2 \ar[dddl,"\textstyle s_2Z_{12}" {xshift=5pt},sloped] \ar[dddd,"\textstyle Z_{32}" {xshift=1pt},swap] & & & & {}^{14}V_2 \ar[dddl,"\textstyle s_2Z_{12}" {xshift=5pt},sloped] \ar[dddd,"\textstyle Z_{32}" {xshift=1pt},swap] & & & & {}^{15}V_1 \ar[dddl,"\textstyle s_1Z_{01}" {xshift=5pt},sloped] \ar[dddd,"\textstyle Z_{21}" {xshift=1pt},swap] & &\\
			& & {}^{10}V_2 \ar[dddl,sloped,"\textstyle s^*\hspace{-4pt}_1Q_2\partial_1" {xshift=5pt}] \ar[dddd,"\textstyle Q_2" {xshift=1pt},swap] & & & & {}^{13}V_3 \ar[dddl,sloped,"\textstyle s^*\hspace{-4pt}_2Q_3\partial_2" {xshift=5pt}] \ar[dddd,"\textstyle Q_3" {xshift=1pt},swap] & & & & {}^{15}V_3 \ar[dddl,sloped,"\textstyle s^*\hspace{-4pt}_2Q_3\partial_2" {xshift=5pt}] \ar[dddd,"\textstyle Q_3" {xshift=1pt},swap] & & & & {}^{14}V_2 \ar[dddl,sloped,"\textstyle s^*\hspace{-4pt}_1Q_2\partial_1" {xshift=5pt}] \ar[dddd,"\textstyle Q_2" {xshift=1pt},swap] &\\
			& & & {}^{9}V_1 \ar[dddl,sloped,"\textstyle s^*\hspace{-4pt}_1 Z_{21}" {xshift=5pt}] \ar[dddd,swap,"\textstyle Z_{01}" {xshift=1pt}] & & & & {}^{10}V_2 \ar[dddl,sloped,"\textstyle s^*\hspace{-4pt}_2 Z_{32}" {xshift=5pt}] \ar[dddd,swap,"\textstyle Z_{12}" {xshift=1pt}] & & & & {}^{12}V_2 \ar[dddl,sloped,"\textstyle s^*\hspace{-4pt}_2 Z_{32}" {xshift=5pt}] \ar[dddd,swap,"\textstyle Z_{12}" {xshift=1pt}] & & & & {}^{13}V_1 \ar[dddl,sloped,"\textstyle s^*\hspace{-4pt}_1 Z_{21}" {xshift=5pt}] \ar[dddd,swap,"\textstyle Z_{01}" {xshift=1pt}]\\
			{}^{8}V_0 \ar[dddd,swap,"\textstyle Z_{10}" {xshift=1pt}] & & & & {}^{9}V_1 \ar[dddd,swap,"\textstyle Z_{21}" {xshift=1pt}] & & & & {}^{11}V_1 \ar[dddd,swap,"\textstyle Z_{21}" {xshift=1pt}] & & & & {}^{12}V_0 \ar[dddd,swap,"\textstyle Z_{10}" {xshift=1pt}] & & &\\
			& {}^{8}V_2 \ar[dddl,sloped,"\textstyle Z_{12}s_1" {xshift=5pt}] \ar[dddd,swap,"\textstyle Q_2" {xshift=1pt}] & & & & {}^{9}V_3 \ar[dddl,sloped,"\textstyle Z_{23}s_2" {xshift=5pt}] \ar[dddd,swap,"\textstyle Q_3" {xshift=1pt}] & & & & {}^{11}V_3 \ar[dddl,sloped,"\textstyle Z_{23}s_2" {xshift=5pt}] \ar[dddd,swap,"\textstyle Q_3" {xshift=1pt}] & & & & {}^{12}V_2 \ar[dddl,sloped,"\textstyle Z_{12}s_1" {xshift=5pt}] \ar[dddd,swap,"\textstyle Q_2" {xshift=1pt}] & &\\
			& & {}^{6}V_2 \ar[dddl,sloped,"\textstyle \partial^*\hspace{-4pt}_1 Q_2 s_1" {xshift=5pt}] \ar[dddd,swap,"\textstyle Z_{12}" {xshift=1pt}] & & & & {}^{7}V_3 \ar[dddl,sloped,"\textstyle \partial^*\hspace{-4pt}_2 Q_3 s_2" {xshift=5pt}] \ar[dddd,swap,"\textstyle Z_{23}" {xshift=1pt}] & & & & {}^{9}V_3 \ar[dddl,sloped,"\textstyle \partial^*\hspace{-4pt}_2 Q_3 s_2" {xshift=5pt}] \ar[dddd,swap,"\textstyle Z_{23}" {xshift=1pt}] & & & & {}^{10}V_2 \ar[dddl,sloped,"\textstyle \partial^*\hspace{-4pt}_1 Q_2 s_1" {xshift=5pt}] \ar[dddd,swap,"\textstyle Z_{12}" {xshift=1pt}] &\\
			& & & {}^{6}V_0 \ar[dddl,sloped,"\textstyle Z_{10}s^*\hspace{-4pt}_1" {xshift=5pt}] & & & & {}^{7}V_1 \ar[dddl,sloped,"\textstyle Z_{21}s^*\hspace{-4pt}_2" {xshift=5pt}] & & & & {}^{9}V_1 \ar[dddl,sloped,"\textstyle Z_{21}s^*\hspace{-4pt}_2" {xshift=5pt}] & & & & {}^{10}V_0 \ar[dddl,sloped,"\textstyle Z_{10}s^*\hspace{-4pt}_1" {xshift=5pt}]\\
			{}^{5}V_1 \ar[dddd,swap,"\textstyle Q_1" {xshift=1pt}] & & & & {}^{6}V_2 \ar[dddd,swap,"\textstyle Q_2" {xshift=1pt}] & & & & {}^{8}V_2 \ar[dddd,swap,"\textstyle Q_2" {xshift=1pt}] & & & & {}^{9}V_1 \ar[dddd,swap,"\textstyle Q_1" {xshift=1pt}] & & &\\
			& {}^{4}V_2 \ar[dddl,sloped,"\textstyle Z_{12}\partial_1" {xshift=5pt}] \ar[dddd,swap,"\textstyle Z_{12}" {xshift=1pt}] & & & & {}^{3}V_3 \ar[dddl,sloped,"\textstyle Z_{23}\partial_2" {xshift=5pt}] \ar[dddd,swap,"\textstyle Z_{23}" {xshift=1pt}] & & & & {}^{5}V_3 \ar[dddl,sloped,"\textstyle Z_{23}\partial_2" {xshift=5pt}] \ar[dddd,swap,"\textstyle Z_{23}" {xshift=1pt}] & & & & {}^{8}V_2 \ar[dddl,sloped,"\textstyle Z_{12}\partial_1" {xshift=5pt}] \ar[dddd,swap,"\textstyle Z_{12}" {xshift=1pt}] & &\\
			& & {}^{3}V_1 \ar[dddl,sloped,"\textstyle Q_1" {xshift=5pt}] & & & & {}^{4}V_2 \ar[dddl,sloped,"\textstyle Q_2" {xshift=5pt}] & & & & {}^{6}V_2 \ar[dddl,sloped,"\textstyle Q_2" {xshift=5pt}] & & & & {}^{7}V_1 \ar[dddl,sloped,"\textstyle Q_1" {xshift=5pt}] &\\
			& & & & & & & & & & & & & & &\\
			{}^{3}V_1 \ar[dddd,swap,"\textstyle Z_{01}" {xshift=1pt}] & & & & {}^{2}V_2 \ar[dddd,swap,"\textstyle Z_{12}" {xshift=1pt}] & & & & {}^{4}V_2 \ar[dddd,swap,"\textstyle Z_{12}" {xshift=1pt}] & & & & {}^{7}V_1 \ar[dddd,swap,"\textstyle Z_{01}" {xshift=1pt}] & & &\\
			& {}^{1}V_1 \ar[dddl,sloped,"\textstyle \partial_1 Z_{01}" {xshift=5pt}] & & & & {}^{0}V_2 \ar[dddl,sloped,"\textstyle \partial_2 Z_{12}" {xshift=5pt}] & & & & {}^{2}V_2 \ar[dddl,sloped,"\textstyle \partial_2 Z_{12}" {xshift=5pt}] & & & & {}^{5}V_1 \ar[dddl,sloped,"\textstyle \partial_1 Z_{01}" {xshift=5pt}] & &\\
			& & & & & & & & & & & & & & &\\
			& & & & & & & & & & & & & & &\\
			{}^{0}V_0 & & & & {}^{-1}V_1 & & & & {}^{1}V_1 & & & & {}^{4}V_0 & & &\\
		\end{tikzcd}
	\]\vspace{-15pt}\[
		\begin{tikzcd}[nodes={inner sep=0pt},column sep={31pt,between origins},row sep={16pt,between origins}]
			& & & {}^{14}V_0 &[-10pt] & & & {}^{17}V_1 \ar[llll,swap,"\textstyle s_2s_1 Z_{01}" {yshift=-1pt}] &[-10pt] & & & {}^{19}V_1 \ar[llll,swap,"\textstyle Q_1" {yshift=-1pt}] &[-10pt] & & & {}^{18}V_0 \ar[llll,swap,"\textstyle Z_{10}\partial^*\hspace{-4pt}_1 \partial^*\hspace{-4pt}_2" {yshift=-1pt}]\\
			& & {}^{13}V_1 & & & & {}^{16}V_2 \ar[llll,swap,"\textstyle s_2 Z_{12} s_1" {yshift=-1pt}] & & & & {}^{18}V_2 \ar[llll,swap,"\textstyle s^*\hspace{-4pt}_1 Q_2 \partial_1" {yshift=-1pt}] & & & & {}^{17}V_1 \ar[llll,swap,"\textstyle \partial^*\hspace{-4pt}_1 Z_{21}\partial^*\hspace{-4pt}_2" {yshift=-1pt}] &\\
			& {}^{11}V_1 & & & & {}^{12}V_2 \ar[llll,swap,"\textstyle s_2 Z_{12} \partial_1" {yshift=-1pt}] & & & & {}^{14}V_2 \ar[llll,swap,"\textstyle \partial^*\hspace{-4pt}_1 Q_2 s_1" {yshift=-1pt}] & & & & {}^{15}V_1 \ar[llll,swap,"\textstyle s^*\hspace{-4pt}_1 Z_{21} \partial^*\hspace{-4pt}_2" {yshift=-1pt}] & &\\
			{}^{8}V_0 & & & & {}^{9}V_1 \ar[llll,swap,"\textstyle \partial_2 s_1 Z_{01}" {yshift=-1pt}] & & & & {}^{11}V_1 \ar[llll,swap,"\textstyle Q_1" {yshift=-1pt}] & & & & {}^{12}V_0 \ar[llll,swap,"\textstyle Z_{10}\partial^*\hspace{-4pt}_1 s^*\hspace{-4pt}_2" {yshift=-1pt}] & & &\\
			& & & & & & & & & & & & & & &\\
			& & & {}^{11}V_1 & & & & {}^{14}V_2 \ar[llll,swap,"\textstyle s_2Z_{12}s_1" {yshift=-1pt}] & & & & {}^{16}V_2 \ar[llll,swap,"\textstyle s^*\hspace{-4pt}_1 Q_2 \partial_1" {yshift=-1pt}] & & & & {}^{15}V_1 \ar[llll,swap,"\textstyle \partial^*\hspace{-4pt}_1 Z_{21}\partial^*\hspace{-4pt}_2" {yshift=-1pt}]\\
			& & {}^{10}V_2 & & & & {}^{13}V_3 \ar[llll,swap,"\textstyle Z_{23}s_2s_1" {yshift=-1pt}] & & & & {}^{15}V_3 \ar[llll,swap,"\textstyle s^*\hspace{-4pt}_1 s^*\hspace{-4pt}_2 Q_3 \partial_2\partial_1" {yshift=-1pt}] & & & & {}^{14}V_2 \ar[llll,swap,"\textstyle \partial^*\hspace{-4pt}_1 \partial^*\hspace{-4pt}_2 Z_{32}" {yshift=-1pt}] &\\
			& {}^{8}V_2 & & & & {}^{9}V_3 \ar[llll,swap,"\textstyle Z_{23}s_2\partial_1" {yshift=-1pt}] & & & & {}^{11}V_3 \ar[llll,swap,"\textstyle \partial^*\hspace{-4pt}_1 s^*\hspace{-4pt}_2 Q_3 \partial_2 s_1" {yshift=-1pt}] & & & & {}^{12}V_2 \ar[llll,swap,"\textstyle s^*\hspace{-4pt}_1 \partial^*\hspace{-4pt}_2 Z_{32}" {yshift=-1pt}] & &\\
			{}^{5}V_1 & & & & {}^{6}V_2 \ar[llll,swap,"\textstyle \partial_2 Z_{12} s_1" {yshift=-1pt}] & & & & {}^{8}V_2 \ar[llll,swap,"\textstyle s^*\hspace{-4pt}_1 Q_2 \partial_1" {yshift=-1pt}] & & & & {}^{9}V_1 \ar[llll,swap,"\textstyle \partial^*\hspace{-4pt}_1 Z_{21} s^*\hspace{-4pt}_2" {yshift=-1pt}] & & &\\
			& & & & & & & & & & & & & & &\\
			& & & {}^{9}V_1 & & & & {}^{10}V_2 \ar[llll,swap,"\textstyle s_2 Z_{12}\partial_1" {yshift=-1pt}] & & & & {}^{12}V_2 \ar[llll,swap,"\textstyle \partial^*\hspace{-4pt}_1 Q_2 s_1" {yshift=-1pt}] & & & & {}^{13}V_1 \ar[llll,swap,"\textstyle s^*\hspace{-4pt}_1 Z_{21} \partial^*\hspace{-4pt}_2" {yshift=-1pt}]\\
			& & {}^{6}V_2 & & & & {}^{7}V_3 \ar[llll,swap,"\textstyle Z_{23}\partial_2 s_1" {yshift=-1pt}] & & & & {}^{9}V_3 \ar[llll,swap,"\textstyle s^*\hspace{-4pt}_1 \partial^*\hspace{-4pt}_2 Q_3 s_2 \partial_1" {yshift=-1pt}] & & & & {}^{10}V_2 \ar[llll,swap,"\textstyle \partial^*\hspace{-4pt}_1 s^*\hspace{-4pt}_2 Z_{32}" {yshift=-1pt}] &\\
			& {}^{4}V_2 & & & & {}^{3}V_3 \ar[llll,swap,"\textstyle Z_{23} \partial_2\partial_1" {yshift=-1pt}] & & & & {}^{5}V_3 \ar[llll,swap,"\textstyle \partial^*\hspace{-4pt}_1\partial^*\hspace{-4pt}_2 Q_3 s_2 s_1" {yshift=-1pt}] & & & & {}^{8}V_2 \ar[llll,swap,"\textstyle s^*\hspace{-4pt}_1 s^*\hspace{-4pt}_2 Z_{32}" {yshift=-1pt}] & &\\
			{}^{3}V_1 & & & & {}^{2}V_2 \ar[llll,swap,"\textstyle \partial_2 Z_{12}\partial_1" {yshift=-1pt}] & & & & {}^{4}V_2 \ar[llll,swap,"\textstyle \partial^*\hspace{-4pt}_1 Q_2 s_1" {yshift=-1pt}] & & & & {}^{7}V_1 \ar[llll,swap,"\textstyle s^*\hspace{-4pt}_1 Z_{21}s^*\hspace{-4pt}_2" {yshift=-1pt}] & & &\\
			& & & & & & & & & & & & & & &\\
			& & & {}^{6}V_0 & & & & {}^{7}V_1 \ar[llll,swap,"\textstyle s_2\partial_1 Z_{01}" {yshift=-1pt}] & & & & {}^{9}V_1 \ar[llll,swap,"\textstyle Q_1" {yshift=-1pt}] & & & & {}^{10}V_0 \ar[llll,swap,"\textstyle Z_{10}s^*\hspace{-4pt}_1 \partial^*\hspace{-4pt}_2" {yshift=-1pt}]\\
			& & {}^{3}V_1 & & & & {}^{4}V_2 \ar[llll,swap,"\textstyle \partial_2 Z_{12} s_1" {yshift=-1pt}] & & & & {}^{6}V_2 \ar[llll,swap,"\textstyle s^*\hspace{-4pt}_1 Q_2 \partial_1" {yshift=-1pt}] & & & & {}^{7}V_1 \ar[llll,swap,"\textstyle \partial^*\hspace{-4pt}_1 Z_{21}s^*\hspace{-4pt}_2" {yshift=-1pt}] &\\
			& {}^{1}V_1 & & & & {}^{0}V_2 \ar[llll,swap,"\textstyle \partial_2 Z_{12} \partial_1" {yshift=-1pt}] & & & & {}^{2}V_2 \ar[llll,swap,"\textstyle \partial^*\hspace{-4pt}_1 Q_2 s_1" {yshift=-1pt}] & & & & {}^{5}V_1 \ar[llll,swap,"\textstyle s^*\hspace{-4pt}_1 Z_{21}s^*\hspace{-4pt}_2" {yshift=-1pt}] & &\\
			{}^{0}V_0 & & & & {}^{-1}V_1 \ar[llll,swap,"\textstyle \partial_2\partial_1 Z_{01}" {yshift=-1pt}] & & & & {}^{1}V_1 \ar[llll,swap,"\textstyle Q_1" {yshift=-1pt}] & & & & {}^{4}V_0 \ar[llll,swap,"\textstyle Z_{10}s^*\hspace{-4pt}_1 s^*\hspace{-4pt}_2" {yshift=-1pt}] & & &
		\end{tikzcd}
	\]\vspace{-5pt}
	\captionsetup{width=.8\linewidth}
	\caption{The tricomplex $\sr{K}$ when $a = b = c = d = 3$. Cohomological degree shifts are omitted, and the symbol ${}^i V_r$ is shorthand for $q^iV_r$. Components of the differential that decrement the second or third coordinate of $\varepsilon \in [0,3]^b \cap \Z^b$ are shown at the top, while those that decrement the first coordinate are shown below.}
	\label{fig:finitetricomplex}
\end{figure}
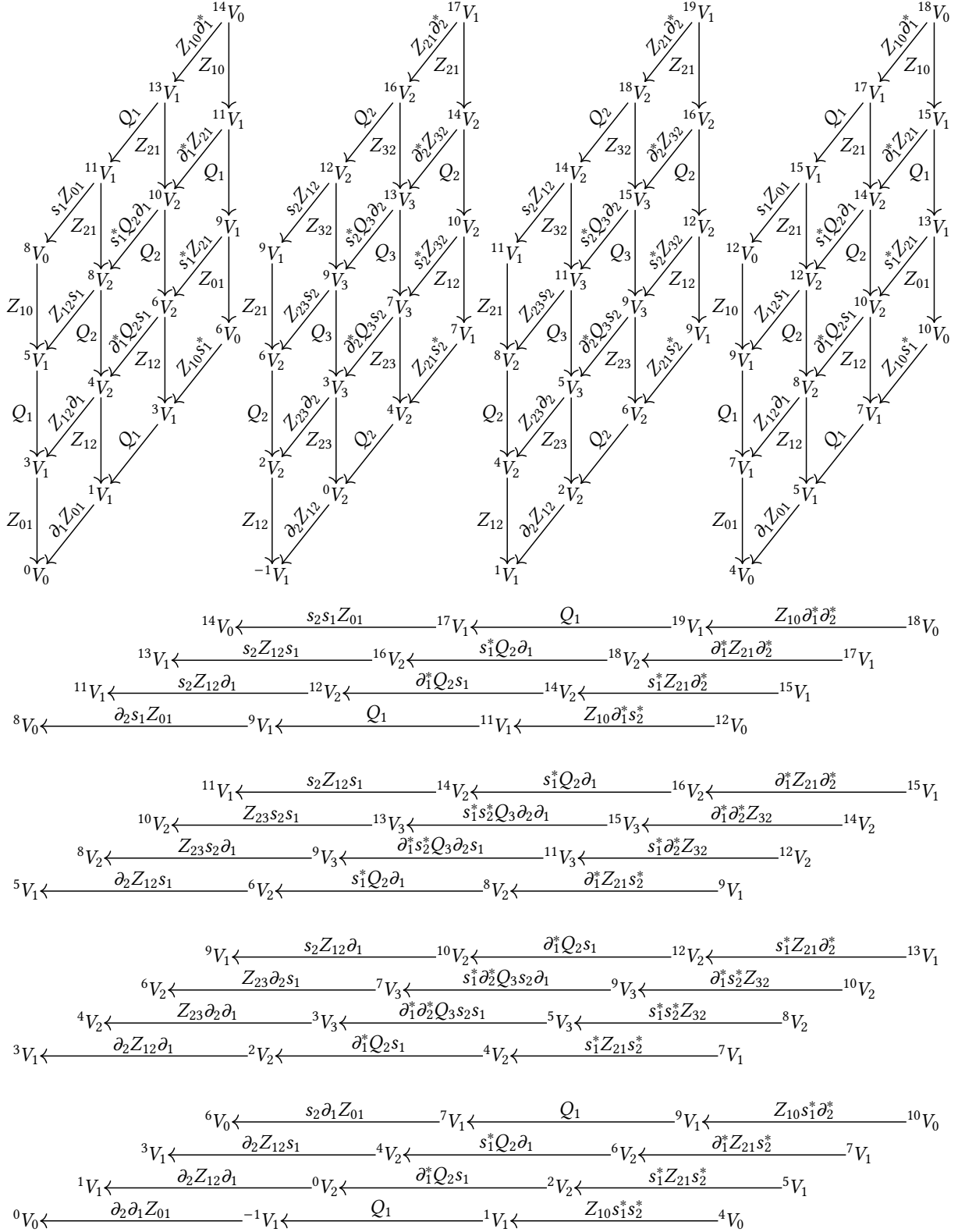

\begin{rem}
	Although $s^*\hspace{-4pt}_i = -s_i$ and $\partial^*\hspace{-4pt}_i = \partial_i$ by Lemma~\ref{lem:relationsK}, we continue to write $s^*\hspace{-4pt}_i$ and $\partial^*\hspace{-4pt}_i$ as a way of keeping track of signs and to make certain symmetries more apparent.
\end{rem}

\begin{lem}\label{lem:selfadjoint}
	The components of the differential assigned to dual edges are adjoint. 
\end{lem}
\begin{proof}
	Let $\varepsilon^j = (\varepsilon_1,\ldots,\varepsilon_{i-1},j,\varepsilon_{i+1},\ldots,\varepsilon_b)$ for $j = 0,1,2,3$, and consider the components of the differential \[
		\begin{tikzcd}[column sep = 100pt]
			V({\varepsilon^{0}}) & V({\varepsilon^{1}}) \ar[l,swap,"\textstyle \phi = \alpha\,Z_{(r-1)r}\,\beta"] & V({\varepsilon^{2}}) \ar[l,swap,"\textstyle \psi = \beta^*\,Q_r\, \hat{\beta}"] & V({\varepsilon^{3}}) \ar[l,swap,"\textstyle \chi = \hat{\beta}^*\, Z_{r(r-1)}\,\hat{\alpha}^*"]
		\end{tikzcd}
	\]where $r = r(\varepsilon^1)$. Recall that $\alpha$ is the descending string with smallest subscript $r$ obtained from $\varepsilon_{i+1},\ldots,\varepsilon_b$ by deleting $1$'s and $2$'s and replacing each $0$ by $\partial$ and $3$ by $s$, while $\beta^*$ is the ascending string with largest subscript $r-1$ obtained from $\varepsilon_{i+1},\ldots,\varepsilon_b$ by deleting $0$'s and $3$'s and replacing each $1$ by $\partial^*$ and $2$ by $s^*$. 

	Now let $\eta^j \coloneqq (\varepsilon^{3-j})^* = (3 - \varepsilon_1,\ldots,3-\varepsilon_{i-1},j,3-\varepsilon_{i+1},\ldots,3-\varepsilon_b)$ for $j = 0,1,2,3$. Note that $r(\varepsilon^1) = r(\eta^1)$. Consider the components of the differential \[
		\begin{tikzcd}[column sep = 100pt]
			V({\eta^{0}}) & V({\eta^{1}}) \ar[l,swap,"\textstyle \phi' = \gamma\,Z_{(r-1)r}\,\delta"] & V({\eta^{2}}) \ar[l,swap,"\textstyle \psi' = \delta^*\,Q_r\, \hat{\delta}"] & V({\eta^{3}}) \ar[l,swap,"\textstyle \chi' = \hat{\delta}^*\, Z_{r(r-1)}\,\hat{\gamma}^*"]
		\end{tikzcd}
	\]where $\gamma$ is the descending string with smallest subscript $r$ obtained from $3-\varepsilon_{i+1},\ldots,3-\varepsilon_b$ by deleting $1$'s and $2$'s and replacing each $0$ by $\partial$ and $3$ by $s$, while $\delta^*$ is the ascending string with largest subscript $r-1$ obtained from $3-\varepsilon_{i+1},\ldots,3-\varepsilon_b$ by deleting $0$'s and $3$'s and replacing each $1$ by $\partial^*$ and $2$ by $s^*$. Then $\gamma = \hat{\alpha}$ and $\delta^* = \hat{\beta}^*$. From the identities $Z^*\hspace{-4pt}_{(r-1)r} = Z_{r(r-1)}$ and $Q^*\hspace{-4pt}_r = Q_r$, it follows that $\phi^* = \chi'$, $\psi^* = \psi'$, and $\chi^* = \phi'$ as required.
\end{proof}

\begin{lem}
	The components of the differential are homogeneous with respect to the $q$-grading.
\end{lem}
\begin{proof}
	Let $\varepsilon^j \coloneqq (\varepsilon_1,\ldots,\varepsilon_{i-1},j,\varepsilon_{i+1},\ldots,\varepsilon_b)$ for $j = 0,1,2,3$, and consider the components of the differential \[
		\begin{tikzcd}[column sep = 100pt]
			V({\varepsilon^{0}}) & V({\varepsilon^{1}}) \ar[l,swap,"\textstyle \phi = \alpha\,Z_{(r-1)r}\,\beta"] & V({\varepsilon^{2}}) \ar[l,swap,"\textstyle \psi = \beta^*\,Q_r\, \hat{\beta}"] & V({\varepsilon^{3}}) \ar[l,swap,"\textstyle \chi = \hat{\beta}^*\, Z_{r(r-1)}\,\hat{\alpha}^*"]
		\end{tikzcd}
	\]where $r = r(\varepsilon^1)$. We must show that $\deg\phi = -\upsilon(\varepsilon^{[0,1]})$, $\deg\psi = -\upsilon(\varepsilon^{[1,2]})$, and $\deg\chi = -\upsilon(\varepsilon^{[2,3]})$ where $\upsilon$ is the $1$-cochain defined in the previous section. Since $\deg \partial_k = \deg \partial^*\hspace{-4pt}_k = -2$ and $\deg s_k = \deg s^*\hspace{-4pt}_k = 0$, we find that \begin{align*}
		&\deg \alpha = \deg \alpha^* = -2\sum_{k=i+1}^b \delta(\varepsilon_k-0) &&\deg \beta = \deg \beta^* = -2\sum_{k=i+1}^b \delta(\varepsilon_k-1)\\
		&\deg \hat{\alpha} = \deg \hat{\alpha}^* = -2\sum_{k=i+1}^b \delta(\varepsilon_k-3) &&\deg \hat{\beta} = \deg \hat{\beta}^* = -2\sum_{k=i+1}^b \delta(\varepsilon_k-2).
	\end{align*}Since $\deg Z_{(r-1)r} = \deg Z_{r(r-1)} = d$ and $\deg Q_r = 2l + 2r$, it follows that \begin{align*}
		\deg \phi &= -2\left(\sum_{k=i+1}^{b} \delta(\varepsilon_{k} - 0) + \delta(\varepsilon_{k} - 1) \right) + d = -\upsilon(\varepsilon^{[0,1]})\\
		\deg \psi &= -2\left(\sum_{k=i+1}^{b} \delta(\varepsilon_{k} - 1) + \delta(\varepsilon_{k} - 2) \right) + 2l + 2r = -\upsilon(\varepsilon^{[1,2]})\\
		\deg \chi &= -2\left(\sum_{k=i+1}^{b} \delta(\varepsilon_{k} - 2) + \delta(\varepsilon_{k} - 3)\right) + d = -\upsilon(\varepsilon^{[2,3]})
	\end{align*}where the second line uses the fact that $r = (\sum_{k=1}^{i-1} \delta(\varepsilon_k - 1) + \delta(\varepsilon_k - 2)) + 1 + (\sum_{k=i+1}^b \delta(\varepsilon_k - 1) + \delta(\varepsilon_k - 2))$. 
\end{proof}

\begin{lem}\label{lem:consecutiveK}
	Consecutive components of the differential that are assigned to edges that are parallel compose to zero. 
\end{lem}
\begin{proof}
	Let $\varepsilon^j = (\varepsilon_1,\ldots,\varepsilon_{i-1},j,\varepsilon_{i+1},\ldots,\varepsilon_b)$ for $j = 0,1,2,3$, and consider the components of the differential \[
		\begin{tikzcd}[column sep = 100pt]
			V({\varepsilon^{0}}) & V({\varepsilon^{1}}) \ar[l,swap,"\textstyle \phi = \alpha\,Z_{(r-1)r}\,\beta"] & V({\varepsilon^{2}}) \ar[l,swap,"\textstyle \psi = \beta^*\,Q_r\, \hat{\beta}"] & V({\varepsilon^{3}}) \ar[l,swap,"\textstyle \chi = \hat{\beta}^*\, Z_{r(r-1)}\,\hat{\alpha}^*"]
		\end{tikzcd}
	\]where $r = r(\varepsilon^1)$. We must show that $\phi\circ\psi = 0$ and $\psi\circ\chi = 0$. Note that for any descending string $\gamma$, the composite map $\gamma\gamma^*$ is either $0$ or $\pm\Id$ because $s_k\, s^*\hspace{-4pt}_k = -\Id$ and $\partial_k \, \partial^*\hspace{-4pt}_k = 0$. It therefore suffices to show that $Z_{(r-1)r}\,Q_r = 0$ and $Q_r\,Z_{r(r-1)} = 0$. By Lemma~\ref{lem:relationsK}, we have that $Z_{(r-1)r}\, Q_{r} = Q_{r}\, Z_{(r-1)r} = 0$ and $Q_{r}\, Z_{r(r-1)} = Z_{r(r-1)}\, Q_{r} = 0$ because $Q_{r}$ acts as the zero endomorphism of the web $V_{r-1}$.
\end{proof}

\begin{prop}\label{prop:commutativeSquaresK}
	The square associated to each face of the fine cubulation of $[0,3]^b$ is commutative. 
\end{prop}
\begin{proof}
	Fix $1 \leq i_1 < i_2\leq b$ and let $\varepsilon^{j_1,j_2} = (\varepsilon_1,\ldots,\varepsilon_{i_1-1},j_1,\varepsilon_{i_1+1},\ldots,\varepsilon_{i_2-1},j_2,\varepsilon_{i_2+1},\ldots,\varepsilon_b)$ for $j_1,j_2 \in \{0,1,2,3\}$. We must verify that the following nine squares are commutative \[
		\begin{tikzcd}[column sep={70pt,between origins}, row sep={45pt,between origins}]
			V(\varepsilon^{03}) \ar[d,swap,"\textstyle \chi^{0*}"] & V(\varepsilon^{13}) \ar[l,swap,"\textstyle \phi^{*3}"] \ar[d,swap,"\textstyle \chi^{1*}"] & V(\varepsilon^{23}) \ar[l,swap,"\textstyle \psi^{*3}"] \ar[d,swap,"\textstyle \chi^{2*}"] & V(\varepsilon^{33}) \ar[l,swap,"\textstyle \chi^{*3}"] \ar[d,swap,"\textstyle \chi^{3*}"]\\
			V(\varepsilon^{02}) \ar[d,swap,"\textstyle \psi^{0*}"] & V(\varepsilon^{12}) \ar[l,swap,"\textstyle \phi^{*2}"] \ar[d,swap,"\textstyle \psi^{1*}"] & V(\varepsilon^{22}) \ar[l,swap,"\textstyle \psi^{*2}"] \ar[d,swap,"\textstyle \psi^{2*}"] & V(\varepsilon^{32}) \ar[l,swap,"\textstyle \chi^{*2}"] \ar[d,swap,"\textstyle \psi^{3*}"]\\
			V(\varepsilon^{01}) \ar[d,swap,"\textstyle \phi^{0*}"] & V(\varepsilon^{11}) \ar[l,swap,"\textstyle \phi^{*1}"] \ar[d,swap,"\textstyle \phi^{1*}"] & V(\varepsilon^{21}) \ar[l,swap,"\textstyle \psi^{*1}"] \ar[d,swap,"\textstyle \phi^{2*}"] & V(\varepsilon^{31}) \ar[l,swap,"\textstyle \chi^{*1}"] \ar[d,swap,"\textstyle \phi^{3*}"]\\
			V(\varepsilon^{00}) & V(\varepsilon^{10}) \ar[l,swap,"\textstyle \phi^{*0}"] & V(\varepsilon^{20}) \ar[l,swap,"\textstyle \psi^{*0}"] & V(\varepsilon^{30}) \ar[l,swap,"\textstyle \chi^{*0}"]
		\end{tikzcd}
	\]where we have added superscripts to the maps that indicate their sources and targets. 
	We will use the following definition. Given a descending string $\alpha$, its \emph{incremented string} $\alpha_+$ is obtained by increasing each subscript by $1$ and keeping the symbols the same. The incremented string of $\alpha = \partial_8\,\partial_7\,\partial_6\,s_5\,\partial_4\,s_3$ is $\alpha_+ = \partial_9\,\partial_8\,\partial_7\,s_6\,\partial_5\,s_4$. The incremented string of an ascending string is defined in the same way. The three operations of forming the opposite, adjoint, and incremented strings pairwise commute so the notation $\hat{\alpha}^*_+$ is unambiguous. 

	Throughout, we let $r \coloneqq r(\varepsilon^{10})$. Also, let $t \ge 0$ be the number of $0$'s and $3$'s in the sequence $\varepsilon_{i_2+1},\ldots,\varepsilon_b$ and let $u \ge 0$ be the number of $1$'s and $2$'s so that $t + u = b - i_2$.  We now define four strings $\alpha,\beta^*,\gamma,\delta^*$ using the fixed data. \begin{itemize}[noitemsep]
		\item Let $\alpha$ be the descending string with smallest subscript $r$ obtained from $\varepsilon_{i_2+1},\ldots,\varepsilon_b$ by deleting $1$'s and $2$'s and replacing each $0$ by $\partial$ and $3$ by $s$. The largest subscript of $\alpha$ is $r + t - 1$. 
		\item Let $\beta^*$ be the ascending string with largest subscript $r-1$ obtained from $\varepsilon_{i_2+1},\ldots,\varepsilon_b$ by deleting $0$'s and $3$'s and replacing each $1$ by $\partial^*$ and $2$ by $s^*$. The smallest subscript of $\beta^*$ is $r - u$.
		\item Let $\gamma$ be the descending string with smallest subscript $r + t + 1$ obtained from $\varepsilon_{i_1+1},\ldots,\varepsilon_{i_2-1}$ by deleting $1$'s and $2$'s and replacing each $0$ by $\partial$ and $3$ by $s$. 
		\item Let $\delta^*$ be the ascending string with largest subscript $r - u - 1$ obtained from $\varepsilon_{i_1+1},\ldots,\varepsilon_{i_2-1}$ by deleting $0$'s and $3$'s and replacing each $1$ by $\partial^*$ and $2$ by $s^*$.
	\end{itemize}

	\vspace{-5pt}

	\noindent\hrulefill
	\vspace{2pt}

	\textbf{The bottom left square}. \[
		\begin{tikzcd}[column sep=75pt,row sep=30pt]
			V(\varepsilon^{01}) \ar[d,swap,"\textstyle \phi^{0*} = \alpha \, Z_{(r-1)r} \, \beta"] & & V(\varepsilon^{11}) \ar[ll,swap,"\textstyle \phi^{*1} = \gamma \,\alpha_+\,Z_{r(r+1)}\,\beta_+ \,\partial_{r-u}\,\delta"] \ar[d,"\textstyle \phi^{1*} = \alpha_+\, Z_{r(r+1)}\,\beta_+"]\\
			V(\varepsilon^{00}) & & V(\varepsilon^{10}) \ar[ll,swap,"\textstyle \phi^{*0} = \gamma\, \partial_{r+t}\, \alpha\, Z_{(r-1)r}\, \beta \,\delta"]
		\end{tikzcd}
	\]Lemma~\ref{lem:relationsK} implies \begin{align*}
		\phi^{0*}\phi^{*1} &= \alpha\, Z_{(r-1)r}\, \beta\,\gamma\,\alpha_+\, Z_{r(r+1)}\, \beta_+\, \partial_{r-u}\,\delta = \gamma\,\alpha\,\alpha_+\,Z_{(r-1)r} \, Z_{r(r+1)}\, \beta\,\beta_+ \,\partial_{r-u}\,\delta\\
		\phi^{*0}\phi^{1*} &= \gamma\, \partial_{r+t} \,\alpha\, Z_{(r-1)r}\,\beta\,\delta\,\alpha_+\,Z_{r(r+1)}\,\beta_+ = \gamma\,\partial_{r+t}\,\alpha\,\alpha_+\, Z_{(r-1)r} \, Z_{r(r+1)}\,\beta\,\beta_+\,\delta.
	\end{align*}We show that $\beta\,\beta_+\,\partial_{r-u} = \partial_r \,\beta\,\beta_+$ by induction on the length $u$ of $\beta$. The base case $u = 0$ is vacuous. For $\theta \in \{\partial,s\}$, we must show that $\beta \,\theta_{r-u-1}\,\beta_+\,\theta_{r-u}\,\partial_{r-u-1} = \partial_r \,\beta\,\theta_{r-u-1}\,\beta_+\,\theta_{r-u}$ assuming that $\beta\,\beta_+\,\partial_{r-u} = \partial_r\,\beta\,\beta_+$. Far commutativity implies that $\theta_{r-u-1}$ and $\beta_+$ commute so the desired equation follows from the identity $\theta_{r-u-1}\,\theta_{r-u}\,\partial_{r-u-1} = \partial_{r-u}\,\theta_{r-u-1}\,\theta_{r-u}$ of Lemma~\ref{lem:mixedBraidRelation} and the inductive hypothesis. The same argument gives $\partial_{r+t}\,\alpha\,\alpha_+ = \alpha\,\alpha_+ \,\partial_r$. Commutativity of the square now follows from the identity $Z_{(r-1)r}\,Z_{r(r+1)}\,\partial_r = \partial_r\,Z_{(r-1)r}\,Z_{r(r+1)}$ of Lemma~\ref{lem:relationsK}. 

	\noindent\hrulefill
	\vspace{2pt}

	\textbf{The bottom middle square}. \[
		\begin{tikzcd}[column sep=75pt,row sep=30pt]
			V(\varepsilon^{11}) \ar[d,swap,"\textstyle \phi^{1*} = \alpha_+\, Z_{r(r+1)}\,\beta_+"] & & V(\varepsilon^{21}) \ar[ll,swap,"\textstyle \psi^{*1} = \delta^*\,\partial^*\hspace{-4pt}_{r-u}\:\beta^*_+\,Q_{r+1}\,\hat{\beta}_+\,s_{r-u}\,\hat{\delta}"] \ar[d,"\textstyle \phi^{2*} = \alpha_+\, Z_{r(r+1)}\,\beta_+"]\\
			V(\varepsilon^{10}) & & V(\varepsilon^{20}) \ar[ll,swap,"\textstyle \psi^{*0} = \delta^*\,\beta^*\,Q_r\,\hat{\beta}\,\hat{\delta}"]
		\end{tikzcd}
	\]We have \begin{align*}
		\phi^{1*}\psi^{*1} &= \alpha_+\,Z_{r(r+1)}\,\beta_+\,\delta^*\,\partial^*\hspace{-4pt}_{r-u}\:\beta^*_+\,Q_{r+1}\,\hat{\beta}_+\, s_{r-u}\,\hat{\delta} = \delta^*\,\alpha_+\,Z_{r(r+1)}\,\beta_+\,\partial^*\hspace{-4pt}_{r-u}\:\beta^*_+\,Q_{r+1}\,\hat{\beta}_+\, s_{r-u}\,\hat{\delta}\\
		\psi^{*0}\phi^{2*} &= \delta^*\,\beta^*\,Q_r\,\hat{\beta}\,\hat{\delta}\,\alpha_+\,Z_{r(r+1)}\,\beta_+ = \delta^*\,\alpha_+\,Z_{r(r+1)}\,\beta^*\,Q_r\,\hat{\beta}\,\beta_+\,\hat{\delta}
	\end{align*}so it suffices to show that $Z_{r(r+1)}\,\beta_+\,\partial^*\hspace{-4pt}_{r-u}\:\beta^*_+\,Q_{r+1}\,\hat{\beta}_+\,s_{r-u} = Z_{r(r+1)}\,\beta^*\, Q_r\,\hat{\beta}\,\beta_+$. 

	We show that $\beta_+\,\partial^*\hspace{-4pt}_{r-u}\:\beta^*_+ = \beta^*\,\partial^*\hspace{-4pt}_r\:\beta$ and $\beta\,\hat{\beta}_+\,s_{r-u} = s_r\,\hat{\beta}\,\beta_+$ by induction on the length $u$ of $\beta$. The base cases $u = 0$ are vacuous. For the first identity, we must show that $\beta_+\,\theta_{r-u}\,\partial^*\hspace{-4pt}_{r-u-1}\:\theta^*\hspace{-4pt}_{r-u}\:\beta^*_+ = \theta^*\hspace{-4pt}_{r-u-1}\:\beta^*\,\partial^*\hspace{-4pt}_r\:\beta\,\theta_{r-u-1}$ assuming $\beta_+\,\partial^*\hspace{-4pt}_{r-u}\:\beta^*_+ = \beta^*\,\partial^*\hspace{-4pt}_r\:\beta$ for $\theta\in\{\partial,s\}$. For this, we use the identity $\theta_{r-u}\,\partial^*\hspace{-4pt}_{r-u-1}\:\theta^*\hspace{-4pt}_{r-u} = \theta^*\hspace{-4pt}_{r-u-1}\:\partial^*\hspace{-4pt}_{r-u}\:\theta_{r-u-1}$ of Lemma~\ref{lem:mixedBraidRelation}, far commutativity of $\theta^*\hspace{-4pt}_{r-u-1}$ and $\beta_+$, and the inductive hypothesis. For the second identity, we must show that $\beta\,\theta_{r-u-1}\,\hat{\beta}_+\,\hat{\theta}_{r-u}\,s_{r-u-1} = s_r\,\hat{\beta}\,\hat{\theta}_{r-u-1}\,\beta_+\,\theta_{r-u}$ assuming $\beta\,\hat{\beta}_+\,s_{r-u} = s_r\,\hat{\beta}\,\beta_+$ for $\theta \in \{\partial,s\}$. For this, we use far commutativity of $\theta_{r-u-1}$ and $\hat{\beta}_+$, the identity $\theta_{r-u-1}\,\hat{\theta}_{r-u}\,s_{r-u-1} = s_{r-u}\,\hat{\theta}_{r-u-1}\,\theta_{r-u}$ of Lemma~\ref{lem:mixedBraidRelation}, and the inductive hypothesis. Altogether, we have \begin{align*}
		Z_{r(r+1)}\,\beta_+\,\partial^*\hspace{-4pt}_{r-u}\:\beta^*_+\,Q_{r+1}\,\hat{\beta}_+\,s_{r-u} &= Z_{r(r+1)}\,\beta^*\,\partial^*\hspace{-4pt}_r\,Q_{r+1}\,s_r\,\hat{\beta}\,\beta_+\\
		&= Z_{r(r+1)} \,\beta^* \,Q_r\,\hat{\beta}\,\beta_+ + Z_{r(r+1)}\,\beta^*\, Q_{r+1}\,\partial^*\hspace{-4pt}_r\,s_r\,\hat{\beta}\,\beta_+
	\end{align*}where the first equality follows from the two identities just established together with commutativity of $\beta$ and $Q_{r+1}$. The second equality uses $\partial^*\hspace{-4pt}_r\, Q_{r+1} = Q_r\, s_r + Q_{r+1}\,\partial^*\hspace{-4pt}_r$ of Lemma~\ref{lem:relationsK}. Lastly, the second term in the sum vanishes by commutativity of $\beta^*$ and $Q_{r+1}$ together with the identity $Z_{r(r+1)}Q_{r+1} = 0$ verified in the proof of Lemma~\ref{lem:consecutiveK}. 

	\noindent\hrulefill
	\vspace{2pt}

	\textbf{The bottom right square}. \[
		\begin{tikzcd}[column sep=75pt,row sep=30pt]
			V(\varepsilon^{21}) \ar[d,swap,"\textstyle \phi^{2*} = \alpha_+\, Z_{r(r+1)}\,\beta_+"] & & V(\varepsilon^{31}) \ar[ll,swap,"\textstyle \chi^{*1} = \hat{\delta}^*\,s^*\hspace{-4pt}_{r-u}\:\hat{\beta}^*_+\,Z_{(r+1)r}\,\hat{\alpha}^*_+\,\hat{\gamma}^*"] \ar[d,"\textstyle \phi^{3*} = \alpha\, Z_{(r-1)r}\,\beta"]\\
			V(\varepsilon^{20}) & & V(\varepsilon^{30}) \ar[ll,swap,"\textstyle \chi^{*0} = \hat{\delta}^*\,\hat{\beta}^*\,Z_{r(r-1)} \,\hat{\alpha}^*\,s^*\hspace{-4pt}_{r+t}\:\hat{\gamma}^*"]
		\end{tikzcd}
	\]Commutativity between $\phi^{2*}$ and $\hat{\delta}^*$ and between $\hat{\gamma}^*$ and $\phi^{3*}$ implies that it suffices to show that \[
		\alpha_+\,Z_{r(r+1)}\,\beta_+\,s^*\hspace{-4pt}_{r-u}\:\hat{\beta}^*_+\,Z_{(r+1)r}\,\hat{\alpha}^*_+ = \hat{\beta}^*\,Z_{r(r-1)} \,\hat{\alpha}^*\,s^*\hspace{-4pt}_{r+t}\:\alpha\,Z_{(r-1)r}\,\beta.
	\]We show that $\beta_+\, s^*\hspace{-4pt}_{r-u}\:\hat{\beta}^*_+ = \hat{\beta}^* s^*\hspace{-4pt}_{r}\: \beta$ by induction on the length $u$ of $\beta$. The base case $u = 0$ is vacuous. For $\theta\in\{\partial,s\}$, we must show that $\beta_+\,\theta_{r-u}\,s^*\hspace{-4pt}_{r-u-1}\:\hat{\theta}^*\hspace{-4pt}_{r-u}\:\hat{\beta}^*_+ = \hat{\theta}^*\hspace{-4pt}_{r-u-1}\:\hat{\beta}^*\,s^*\hspace{-4pt}_r\:\beta\,\theta_{r-u-1}$ assuming $\beta_+\, s^*\hspace{-4pt}_{r-u}\:\hat{\beta}^*_+ = \hat{\beta}^* s^*\hspace{-4pt}_{r}\: \beta$. For this, we use the identity $\theta_{r-u}\,s^*\hspace{-4pt}_{r-u-1}\:\hat{\theta}^*\hspace{-4pt}_{r-u} = \hat{\theta}^*\hspace{-4pt}_{r-u-1}\:s^*\hspace{-4pt}_{r-u}\:\theta_{r-u-1}$ of Lemma~\ref{lem:mixedBraidRelation}, far commutativity of $\beta_+$ and $\hat{\theta}^*\hspace{-4pt}_{r-u-1}$, and the inductive hypothesis. A similar argument gives $\hat{\alpha}^*\, s^*\hspace{-4pt}_{r+t}\:\alpha = \alpha_+\,s^*\hspace{-4pt}_r\:\hat{\alpha}^*_+$. These two identities, far commutativity, and the identity $Z_{r(r+1)}\,s^*\hspace{-4pt}_r\:Z_{(r+1)r} = Z_{r(r-1)}\,s^*\hspace{-4pt}_r\:Z_{(r-1)r}$ of Lemma~\ref{lem:relationsK} establish the result. 

	\noindent\hrulefill
	\vspace{2pt}

	\textbf{The middle left square}. \[
		\begin{tikzcd}[column sep=75pt,row sep=30pt]
			V(\varepsilon^{02}) \ar[d,swap,"\textstyle \psi^{0*} = \beta^* \, Q_r \, \hat{\beta}"] & & V(\varepsilon^{12}) \ar[ll,swap,"\textstyle \phi^{*2} = \gamma\,\alpha_+\,Z_{r(r+1)}\,\beta_+\,s_{r-u}\,\delta "] \ar[d,"\textstyle \psi^{1*} = \beta^*_+\, Q_{r+1}\,\hat{\beta}_+"]\\
			V(\varepsilon^{01}) & & V(\varepsilon^{11}) \ar[ll,swap,"\textstyle \phi^{*1} = \gamma \,\alpha_+\,Z_{r(r+1)}\,\beta_+ \,\partial_{r-u}\,\delta"]
		\end{tikzcd}
	\]The commutativity relations imply that \begin{align*}
		\psi^{0*} \phi^{*2} &= \beta^*\,Q_r\,\hat{\beta}\, \gamma\,\alpha_+\,Z_{r(r+1)}\,\beta_+\,s_{r-u}\,\delta = \gamma\,\alpha_+\,\beta^* \,Z_{r(r+1)}\,Q_r\,\hat{\beta}\,\beta_+\,s_{r-u}\,\delta\\
		\phi^{*1} \psi^{1*} &= \gamma \,\alpha_+\,Z_{r(r+1)}\,\beta_+ \,\partial_{r-u}\,\delta \, \beta^*_+\, Q_{r+1}\,\hat{\beta}_+ = \gamma\,\alpha_+\,Z_{r(r+1)}\,\beta_+\,\partial_{r-u}\,\beta^*_+\,Q_{r+1}\,\hat{\beta}_+\,\delta
	\end{align*}As established in the case of the bottom middle square, we have $\hat{\beta}\,\beta_+\,s_{r-u} = s_r\,\beta\,\hat{\beta}_+$ and $\beta_+ \,\partial_{r-u}\,\beta^*_+ = \beta^*\,\partial_r\,\beta$. These identities and the commutativity relations imply that it suffices to show that $Z_{r(r+1)}\,\partial_r\,Q_{r+1} = Z_{r(r+1)}\,Q_r\,s_r$. This follows from the identity $\partial_r\,Q_{r+1} = Q_r\,s_r + Q_{r+1}\,\partial_r$ of Lemma~\ref{lem:relationsK} and the identity $Z_{r(r+1)}\,Q_{r+1} = 0$ verified in Lemma~\ref{lem:consecutiveK}.

	\noindent\hrulefill
	\vspace{2pt}

	\textbf{The center square}. \[
		\begin{tikzcd}[column sep=75pt,row sep=30pt]
			V(\varepsilon^{12}) \ar[d,swap,"\textstyle \psi^{1*} = \beta^*_+\, Q_{r+1}\,\hat{\beta}_+"] & & V(\varepsilon^{22}) \ar[ll,swap,"\textstyle \psi^{*2} = \delta^*\,s^*\hspace{-4pt}_{r-u}\:\beta^*_+\,Q_{r+1}\,\hat{\beta}_+\,\partial_{r-u}\,\hat{\delta}"] \ar[d,"\textstyle \psi^{2*} = \beta^*_+\, Q_{r+1}\,\hat{\beta}_+"]\\
			V(\varepsilon^{11}) & & V(\varepsilon^{21}) \ar[ll,swap,"\textstyle \psi^{*1} = \delta^*\,\partial^*\hspace{-4pt}_{r-u}\:\beta^*_+\,Q_{r+1}\,\hat{\beta}_+\,s_{r-u}\,\hat{\delta}"]
		\end{tikzcd}
	\]Commutativity between $\psi^{1*}$ and $\delta^*$ and between $\hat{\delta}$ and $\psi^{2*}$ implies that it suffices to show that \[
		\beta^*_+\,Q_{r+1}\,\hat{\beta}_+\,s^*\hspace{-4pt}_{r-u}\,\beta^*_+\,Q_{r+1}\,\hat{\beta}_+\,\partial_{r-u} = \partial^*\hspace{-4pt}_{r-u}\,\beta^*_+\,Q_{r+1}\,\hat{\beta}_+\,s_{r-u}\,\beta^*_+\,Q_{r+1}\,\hat{\beta}_+.
	\]As established in the case of the bottom right square, we have $\hat{\beta}_+\,s_{r-u}\,\beta^*_+ = \beta^*\, s_r\,\hat{\beta}$. Applying this and the commutativity relations to the two sides of the desired equality, it thereby suffices to show that \[
		\beta^*_+\,\beta^*\,Q_{r+1}\,s^*\hspace{-4pt}_r\:Q_{r+1}\,\hat{\beta}\,\hat{\beta}_+\,\partial_{r-u} = \partial^*\hspace{-4pt}_{r-u}\:\beta^*_+\,\beta^*\, Q_{r+1}\,s_r\,Q_{r+1}\,\hat{\beta}\,\hat{\beta}_+.
	\]From the case of the bottom left square, we have $\hat{\beta}\,\hat{\beta}_+\,\partial_{r-u} = \partial_r\,\hat{\beta}\,\hat{\beta}_+$ and $\partial^*\hspace{-4pt}_{r-u}\:\beta^*_+\,\beta^* = \beta^*_+\,\beta^*\,\partial^*\hspace{-4pt}_r$. So it now suffices to show that $Q_{r+1}\,s^*\hspace{-4pt}_r\:Q_{r+1}\,\partial_r = \partial^*\hspace{-4pt}_r\:Q_{r+1}\,s_r\,Q_{r+1}$. This follows from the computation \begin{align*}
		\partial^*\hspace{-4pt}_r\:Q_{r+1}\,s_r\,Q_{r+1} &= (Q_r\,s_r + Q_{r+1}\,\partial_r)\,s_r\,Q_{r+1}\\
		&= Q_r\,Q_{r+1} - Q_{r+1}\,\partial_r\,Q_{r+1}\\
		&= Q_r\,Q_{r+1} - Q_{r+1}\,(Q_r + s_r\,Q_{r+1}\,\partial_r) = Q_{r+1}\, s^*\hspace{-4pt}_r\:Q_{r+1}\,\partial_r
	\end{align*}where we have used the identities $\partial_r\,Q_{r+1} = Q_r\,s_r + Q_{r+1}\,\partial_r = Q_r + s_r\,Q_{r+1}\,\partial_r$ and $\partial_r\,s_r = -\partial_r$.

	\noindent\hrulefill
	\vspace{2pt}

	\textbf{The remaining squares}. Commutativity of the remaining squares follows from the cases already established by the symmetry of Lemma~\ref{lem:selfadjoint}. This completes the proof.
\end{proof}


\section{Construction of $\sr{P}$}\label{sec:constructionOfP}


Just as in the previous section, we have positive integers $a,b,c,d$ for which $a + b = c + d$ and $b = \min(a,b,c,d)$. We also set $n \coloneq a + b = c + d$ and $l \coloneq c - b = a - d$. The purpose of this section is to construct $\sr{P} \coloneq {}^b_a \sr{P}^c_d$. It is a bounded-above chain complex of singular Bott--Samelson bimodules with boundary data $c_L = (a,b)$ and $c_R = (d,c)$. A key ingredient in the construction is the complex $\sr{K} \coloneqq {}^b_a \sr{K}^c_d$. In section~\ref{subsec:webs_and_foams_in_P}, we introduce all of the relevant webs and foams needed to construct $\sr{P}$. We also relate these webs and foams to those appearing in $\sr{K}$, described in section~\ref{subsec:webs_and_foams_in_K}. In section~\ref{subsec:the_shape_of_P}, we explain the shape of $\sr{P}$, similar to how we described the shape of $\sr{K}$ in section~\ref{subsec:the_shape_of_K}. In sections~\ref{subsec:the_objects_of_P} and \ref{subsec:the_differential_of_P}, we define the objects and differential of $\sr{P}$, respectively. 

\subsection{The webs and foams in $\sr{P}$}\label{subsec:webs_and_foams_in_P}

Recall that a composition of a number $r$ is a sequence of positive integers $(g_1,\ldots,g_m)$ for which $g_1 + \cdots + g_m = r$. By abuse of notation, we let $0$ denote the empty composition of $0$.

\begin{df}\label{def:websfoamsP}
	Given a composition $g = (g_1,\ldots,g_m)$ of $r \in \{0,\ldots,b\}$, let $W^g_r = W^{g_1,\ldots,g_m}_r$ be the following web with boundary data $c_L = (a,b)$ and $c_R = (d,c)$. \vspace{5pt}\[
		\begin{gathered}
			\labellist
			\pinlabel $a$ at -0.5 0.1
			\pinlabel $b$ at -0.5 7.7
			\pinlabel $c$ at 14.8 7.7
			\pinlabel $d$ at 14.8 0.1
			\pinlabel $g_1$ at 1.2 6
			\pinlabel $g_2$ at 2.75 6
			\pinlabel $\cdots$ at 4.35 6
			\pinlabel $g_m$ at 5.6 6
			\pinlabel $b-r$ at 9.4 6.9
			\pinlabel $l+r$ at 13.7 6
			\pinlabel $a+r$ at 9.4 {-.6}
			\endlabellist
			\includegraphics[width=.295\textwidth]{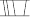}
		\end{gathered}\vspace{5pt}
	\]The edges labeled $g_1,g_2,\ldots,g_m$, and $b-r$ are assigned the alphabets \[
		\{x_1,\ldots,x_{g_1}\}, \{x_{g_1+1},\ldots,x_{g_1+g_2}\},\ldots,\{x_{r-g_m+1},\ldots,x_r\}, \{x_{r+1},\ldots,x_b\},
	\]respectively, while the edges labeled $a,b,c,d,l+r,a+r$ are assigned the alphabets \vspace{5pt}\[
		\begin{gathered}
			\labellist
			\pinlabel $\A$ at -0.6 0.1
			\pinlabel $\B$ at -0.6 7.7
			\pinlabel $\C$ at 14.9 7.7
			\pinlabel $\mathbf{D}$ at 14.9 0.1
			\pinlabel $\cdots$ at 5 5.9
			\pinlabel $\mathbf{E}_r$ at 13.4 6
			\pinlabel $\F_r$ at 9.4 {-.8}
			\endlabellist
			\includegraphics[width=.25\textwidth]{webW}
		\end{gathered}\vspace{5pt}
	\]Let $\xi(g) \coloneq \binom{g_1}{2} + \cdots + \binom{g_m}{2} + \binom{b-r}{2}$, and let $\iota^g \in \Hom^{-\xi(g)}(W^g_r,V_r)$ be the map induced by the natural inclusion map\[
		\Z[x_1,\ldots,x_b]^{\fk{S}_{g_1} \x \cdots \fk{S}_{g_m} \x \fk{S}_{b-r}} \hookrightarrow \Z[x_1,\ldots,x_b]
	\]and is linear over the elementary symmetric polynomials in $\A,\B,\C,\mathbf{D},\mathbf{E}_r,\mathbf{F}_r$. Let $\pi^g \in \Hom^{-\xi(g)}(V_r,W^g_r)$ be the adjoint of $\iota$. A foam representing $\pi$ can be obtained by gluing together local foams that merge adjacent rungs \[
		\begin{gathered}
			\includegraphics[width=.13\textwidth]{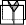}
		\end{gathered}
	\]in any order. A foam representing $\iota$ is obtained in a similar way, or by reflecting a foam representing $\pi$ across a horizontal plane.
\end{df}

\begin{lem}\label{lem:relationsP}
	There is a homogeneous polynomial $p_g \in \Z[x_1,\ldots,x_b]$ of $q$-degree $2\xi(g)$ with the property that $\pi^g\,p_g\,\iota^g = \Id$ on $W^g_r$. A bimodule map $f\colon B \to V_r$ from a singular Bott--Samelson bimodule $B$ factors through $q^{-\xi(g)}W^g_r$ as $f = \iota^g \tilde{f}$ \[
		\begin{tikzcd}
			B \ar[rr,"f"] \ar[dr,swap,dotted,"\tilde{f}"] & & V_r\\
			& q^{-\xi(g)}W^g_r \ar[ur,swap,"\iota^g"] &
		\end{tikzcd}
	\]if and only if $\partial_i f = 0$ for all $i\in\{1,\ldots,b-1\}\setminus\{g_1,g_1+g_2,\ldots,g_1+\cdots+g_m\}$. The map $\tilde{f}$ can be expressed as $\tilde{f} = \pi^g \, p_g \, f$.
\end{lem}
\begin{proof}
	The $\U(b)$-equivariant cohomology of the partial flag manifold $\mathrm{Fl}(g_1,\ldots,g_m,b-r;b)$ of $\C^b$ is naturally identified with $\Z[x_1,\ldots,x_b]^{\fk{S}_{g_1} \x \cdots \x \fk{S}_{g_m} \x \fk{S}_{b-r}}$. The inclusion map of this ring into $\Z[x_1,\ldots,x_b]$ is induced by the forgetful map from the full flag manifold of $\C^b$ to this partial flag manifold. The polynomial $p_g \in \Z[x_1,\ldots,x_b]$ is the equivariant fundamental class of this partial flag manifold, which has the stated property. 

	Clearly, a bimodule map $f$ factors as $f = \iota^g\circ\tilde{f}$ if and only if its image consists of polynomials that are invariant under $\fk{S}_{g_1} \x \cdots \x \fk{S}_{g_m} \x \fk{S}_{b-r}$. The second assertion follows from the fact that $s_if = f$ if and only if $\partial_i f = 0$. Lastly, if $f = \iota^g \tilde{f}$, then $\tilde{f} = \pi^g\,p_g\,\iota^g\tilde{f} = \pi^g\,p_g\,f$. 
\end{proof}

\subsection{The shape of $\sr{P}$}\label{subsec:the_shape_of_P}

Consider the convex set $T \coloneqq \{(x_1,\ldots,x_b) \in \R^b \,|\, x_1 \ge \cdots \ge x_b \ge 0\}$. For each integer $k \ge 0$, let $T_k \subset T$ be the subset consisting of points $(x_1,\ldots,x_b)$ for which $k \ge x_1 \ge \cdots \ge x_b \ge 0$. Then $T_k$ is a $b$-dimensional simplex, and the sequence $T_0 \subset T_1 \subset \cdots$ forms an exhaustive filtration of $T$. Observe that a lattice point $\lambda \in T \cap \Z^b$ within $T$ is a sequence of integers $\lambda = (\lambda_1,\ldots,\lambda_b)$ for which $\lambda_1 \ge \cdots \ge \lambda_b \ge 0$. So $T \cap \Z^b$ may be thought of as the set of partitions with at most $b$ parts. Similarly, the lattice points $T_k \cap \Z^b$ within $T_k$ may be thought of as the set of partitions with at most $b$ parts, each of size at most $k$. In addition to the lattice points, we are also interested in the edges and faces of the standard cubulation of $\R^b$ that are contained within $T$. Convexity of $T$ implies that such an edge lies within $T$ if and only if both of its endpoints do. See Figure~\ref{fig:infinitetetrahedron}. 

\begin{figure}[!ht]
	\[
		\begin{tikzcd}[nodes={inner sep=0pt},row sep=20pt]
			& & & \bullet \ar[r, no head] & \:\cdots\\
			& & \bullet \ar[r, no head] & \bullet \ar[u, no head] \ar[r, no head] & \:\cdots\\
			& \bullet \ar[r, no head] & \bullet \ar[r, no head] \ar[u, no head] & \bullet \ar[u, no head] \ar[r, no head] & \:\cdots \\
			\bullet \ar[r, no head] & \bullet \ar[r, no head] \ar[u, no head] & \bullet \ar[r, no head] \ar[u, no head] & \bullet \ar[u, no head] \ar[r, no head] & \:\cdots
		\end{tikzcd} \hspace{40pt} \begin{tikzcd}[nodes={inner sep=0pt},column sep={16pt,between origins},row sep={10pt,between origins}]
			& & & & & & & & & \bullet \ar[no head,ddd] & & \cdots \ar[no head,ll]\\
			& & & & & & & & &\\
			& & & & & & & & &\\
			& & & & & & & & & \bullet \ar[no head,ddd] \ar[no head,ddl] & & \cdots \ar[no head,ll]\\
			& & & & & & & & &\\
			& & & & & & \bullet \ar[no head,ddd] & & \bullet \ar[no head,ddd] \ar[no head,ll] & & \cdots \ar[no head,ll,crossing over]\\
			& & & & & & & & & \bullet \ar[no head,ddd] \ar[no head,ddl] & & \cdots \ar[no head,ll]\\
			& & & & & & & & &\\
			& & & & & & \bullet \ar[no head,ddd] \ar[no head,ddl] & & \bullet \ar[no head,ddd] \ar[no head,ll] \ar[no head,ddl] & & \cdots \ar[no head,ll,crossing over]\\
			& & & & & & & & & \bullet \ar[no head,ddl] & & \cdots \ar[no head,ll]\\
			& & & \bullet \ar[no head,ddd] & & \bullet \ar[no head,ddd] \ar[no head,ll] & & \bullet \ar[no head,ll,crossing over] & & \cdots \ar[no head,ll,crossing over]\\
			& & & & & & \bullet \ar[no head,ddl] & & \bullet \ar[no head,ll] \ar[no head,ddl] & & \cdots \ar[no head,ll]\\
			& & & & & & & & &\\
			& & & \bullet \ar[no head,ddl] & & \bullet \ar[no head,ll] \ar[no head,ddl] & & \bullet \ar[no head,ll] \ar[no head,ddl] \ar[no head,from=uuu,crossing over] & & \cdots \ar[no head,ll]\\
			& & & & & & & & &\\
			\bullet & & \bullet \ar[no head,ll] & & \bullet \ar[no head,ll] & & \bullet \ar[no head,ll] & & \cdots \ar[no head,ll]
		\end{tikzcd}
	\]
	\captionsetup{width=.8\linewidth}
	\caption{The vertices and edges of the standard cubulation of $\R^b$ contained within $T$.}
	\label{fig:infinitetetrahedron}
\end{figure}
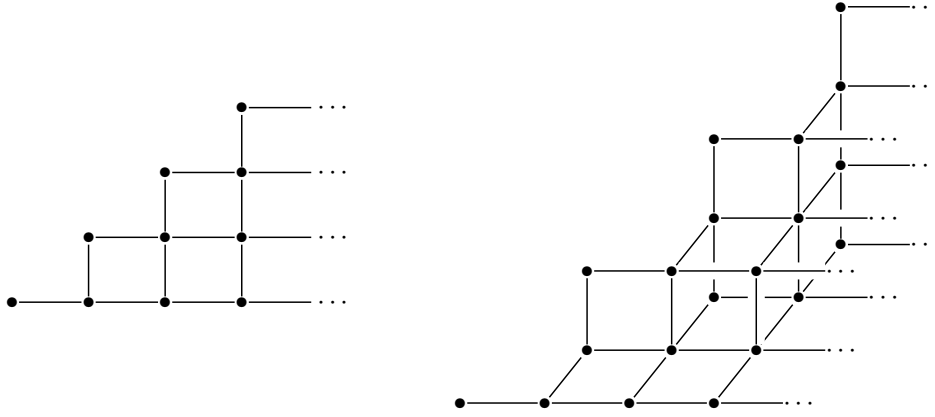

We will define $\sr{P} = {}^b_a \sr{P}^c_d$ as a $b$-fold complex in the following way. To each $\lambda \in T \cap \Z^b$, we assign an object $W(\lambda)$, which is just one of the webs $W^g_r$ from section~\ref{subsec:webs_and_foams_in_P} with a quantum grading shift. The complex $\sr{P}$ is the direct sum $\bigoplus_\lambda t^{-|\lambda|} W(\lambda)$ over all $\lambda \in T \cap \Z^b$ where $|\lambda| \coloneq \sum_{i=1}^b \lambda_i$. The nontrivial components of the differential lie along the edges of the standard cubulation of $\R^b$ that are contained within $T$. In particular, the differential decrements a coordinate of $\lambda$ by one. The differential squares to zero when traveling along consecutive edges in the same direction. Each face of the standard cubulation of $\R^b$ that is contained in $T$ yields a commutative square. 

The filtration $\sr{F}^0(\sr{P}) \subset \sr{F}^1(\sr{P}) \subset \cdots$ will be defined in the following way. For each integer $k \ge 0$, the subcomplex $\sr{F}^k(\sr{P})$ is the direct sum $\bigoplus_\lambda t^{-|\lambda|} W(\lambda)$ over all $\lambda \in T_k \cap \Z^b$, with differential agreeing with that of $\sr{P}$ along all edges in $T_k$. So $\sr{F}^k(\sr{P})$ is modeled on the vertices and edges of the standard cubulation of $\R^b$ that are contained within the simplex $T_k$. 

\subsection{The objects of $\sr{P}$}\label{subsec:the_objects_of_P}

Fix a partition $\lambda = (\lambda_1,\ldots,\lambda_b) \in T \cap \Z^b$ with at most $b$ parts, and let $r(\lambda) \in \{0,\ldots,b\}$ be the largest index for which $\lambda_{r(\lambda)} \neq 0$. So $\lambda_1,\ldots,\lambda_{r(\lambda)}$ are the (nonzero) parts of $\lambda$ and $\lambda_{r(\lambda)+1} = \cdots = \lambda_b = 0$. Let $g(\lambda) = (g_1,\ldots,g_m)$ be the multiplicities of the (nonzero) parts of $\lambda$, and note that $g_1 + \cdots + g_m = r(\lambda)$. Set \[
	W(\lambda) \coloneq q^{H(\lambda)}W^{g(\lambda)}_{r(\lambda)}
\]where the grading shift function $H\colon T \cap \Z^b \to \Z$ is given by the formula \[
	H(\lambda) \coloneqq \left(\sum_{i=1}^{b} 2i \lambda_i \right) - \frac12\left(r(\lambda)^2 + \sum_{j=1}^m g_j^2\right) + (a-b)|\lambda| - (c-b)\#\{\text{odd parts of }\lambda\} - (d-b)\#\{\text{even parts of } \lambda\}.
\]where $|\lambda| = \lambda_1 +\cdots + \lambda_{r(\lambda)}$ and $\#\{\text{even parts of }\lambda\} = \#\{i\in\{1,\ldots,b\}\:|\: \lambda_i \text{ is even and positive }\}$. 

We relate $H$ to the grading shift function $G\colon [0,3]^b \cap \Z^b \to \Z$ and its coboundary $\upsilon$ defined in section~\ref{subsec:the_objects_of_K}. Let $\varepsilon\colon \Z_{\ge 0} \to \{0,1,2\}$ be the function \[
	\varepsilon(j) \coloneqq \begin{cases}
		0 & j = 0\\
		1 & j \text{ is odd }\\
		2 & j \text{ is even and positive}
	\end{cases} 
\]and extend it coordinate-wise to a function $\varepsilon \colon T \cap \Z^b \to [0,2]^b \cap \Z^b$ by $\varepsilon(\lambda)\coloneqq (\varepsilon(\lambda_1),\ldots,\varepsilon(\lambda_b))$. Both functions are named $\varepsilon$ by abuse of notation. We view $\varepsilon(\lambda)$ as a vertex of the fine cubulation of $[0,3]^b$ that happens to lie in the subcube $[0,2]^b$. Next, fix $i \in \{1,\ldots,b\}$ and consider integers satisfying \[
	\lambda_1 \ge \cdots \ge \lambda_{i-1} \ge j+1 > j \ge \lambda_{i+1} \ge \cdots \ge \lambda_b \ge 0.
\]Let $\lambda^{[j,j+1]}$ denote the edge of the standard cubulation of $\R^b$ whose endpoints are $\lambda^j \coloneq (\lambda_1,\ldots,\lambda_{i-1},j,\lambda_{i+1},\ldots,\lambda_b)$ and $\lambda^{j+1} \coloneq (\lambda_1,\ldots,\lambda_{i-1},j+1,\lambda_{i+1},\ldots,\lambda_b)$, which lie in $T \cap \Z^b$. Any edge contained in $T \cap \Z^b$ is of this form. Set $\varepsilon_k \coloneqq \varepsilon(\lambda_k)$ for $k \in \{1,\ldots,b\}\setminus \{i\}$. Let $\varepsilon^{[0,1]}$, $\varepsilon^{[1,2]}$, and $\varepsilon^{[2,3]}$ be the following edges of the fine cubulation of $[0,3]^b$ \[
	\begin{tikzcd}[column sep={70pt,between origins}]
		\varepsilon^0 & \varepsilon^1 \ar[l,swap,"\varepsilon^{[0,1]}"] & \varepsilon^2 \ar[l,swap,"\varepsilon^{[1,2]}"] & \varepsilon^3 \ar[l,swap,"\varepsilon^{[2,3]}"]
	\end{tikzcd}
\]where $\varepsilon^k \coloneq (\varepsilon_1,\ldots,\varepsilon_{i-1},k,\varepsilon_{i+1},\ldots,\varepsilon_b)$ for $k \in \{0,1,2,3\}$. Note that \[
	(\varepsilon(\lambda^j),\varepsilon(\lambda^{j+1})) = \begin{cases}
		(\varepsilon^0,\varepsilon^1) & j = 0\\
		(\varepsilon^1,\varepsilon^2) & j \text{ is odd}\\
		(\varepsilon^2,\varepsilon^1) & j \text{ is even and positive.}
	\end{cases}
\]Now define \[
	\omega(\lambda^{[j,j+1]}) \coloneq \begin{cases}
		\upsilon(\varepsilon^{[0,1]}) + \xi(g(\lambda^{j+1})) - \xi(g(\lambda^j)) & j = 0\\
		\upsilon(\varepsilon^{[1,2]}) + \xi(g(\lambda^{j+1})) - \xi(g(\lambda^j)) & j \text{ is odd }\\
		\upsilon(\varepsilon^{[2,3]}) + \upsilon(\varepsilon^{[0,1]}) + \xi(g(\lambda^{j+1})) - \xi(g(\lambda^j)) & j \text{ is even and positive.}
	\end{cases}
\]

\begin{lem}\label{lem:gradingFunctionP}
	The function $H\colon T \cap \Z^b \to \Z$ satisfies $H(\lambda^j) - H(\lambda^{j+1}) = \omega(\lambda^{[j,j+1]})$ for all edges $\lambda^{[j,j+1]}$ of the standard cubulation of $\R^b$ that are contained in $T$. 
\end{lem}
\begin{proof}
	Set $H_\xi(\lambda)\coloneqq H(\lambda) + \xi(g(\lambda))$. Since $\xi(g(\lambda)) - \frac12(r(\lambda)^2 + \sum_{j=1}^m g_j^2) = \binom{b}{2} - br(\lambda)$, we have \begin{align*}
		H_\xi(\lambda) &= \left(\sum_{h=1}^b 2h\lambda_h\right) + \binom{b}{2} - br(\lambda) + (a-b)|\lambda| - (c-b)\#\{\text{odd parts of }\lambda\} - (d-b)\#\{\text{even parts of } \lambda\}\\
		&= \left(\sum_{h=1}^b 2h\lambda_h\right) + \binom{b}{2} + (a-b)|\lambda| - c\#\{\text{odd parts of }\lambda\} - d\#\{\text{even parts of } \lambda\}
	\end{align*}We must show that $H_\xi(\lambda^j) - H_\xi(\lambda^{j+1})$ is equal to \[
		\begin{cases}
			\upsilon(\varepsilon^{[0,1]}) = - 2i + 2b  - d & j = 0\\
			\upsilon(\varepsilon^{[1,2]}) = -2i + 2b - 2c & j \text{ is odd}\\
			\upsilon(\varepsilon^{[2,3]})+\upsilon(\varepsilon^{[0,1]}) = -2i + 2b - 2d & j \text{ is even and positive}
		\end{cases}
	\]where we have evaluated $\upsilon(\varepsilon^{[j,j+1]})$ using the definition of $\varepsilon$ above and the definition of $\upsilon$ given section~\ref{subsec:the_objects_of_K}. From the explicit formula for $H_\xi(\lambda^{j+1})$ above, we see that if we decrease its $i$th entry, which is $j+1$, by $1$, then the first term decreases by $2i$, the second term is unchanged, and the third term decreases by $a - b$. If $j = 0$, the fourth and fifth terms collectively increase by $c$. If $j$ is odd, the fourth and fifth terms collectively increase by $d-c$, while if $j$ is even and positive, then the fourth and fifth terms collectively increase by $c - d$. Using the identity $a = c + d - b$, we obtain \[
		H_\xi(\lambda^j) - H_\xi(\lambda^{j+1}) = -2i - (a - b) + \begin{cases}
			c & j = 0\\
			d-c & j \text{ is odd }\\
			c - d & j \text{ is even and positive }
		\end{cases} = \begin{cases}
			-2i + 2b - d & j = 0\\
			-2i + 2b - 2c & j \text{ is odd }\\
			-2i + 2b - 2d & j \text{ is even }
		\end{cases}
	\]as required.
\end{proof}

\subsection{The differential of $\sr{P}$}\label{subsec:the_differential_of_P}

Fix $i \in \{1,\ldots,b\}$ and integers $\lambda_1 \ge \cdots \ge \lambda_{i-1} \ge j+1 > j \ge \lambda_{i+1} \ge \cdots \ge \lambda_b \ge 0$. Let $\lambda^{[j,j+1]}$ be the edge of the standard cubulation of $\R^b$ contained in $T$ with endpoints $\lambda^j$ and $\lambda^{j+1}$ as before. Set $\varepsilon_k \coloneq \varepsilon(\lambda_k)$ for $k \in \{1,\ldots,b\}\setminus\{i\}$, and consider the following edges of the fine cubulation of $[0,3]^b$ \[
	\begin{tikzcd}[column sep={70pt,between origins}]
		\varepsilon^0 & \varepsilon^1 \ar[l,swap,"\varepsilon^{[0,1]}"] & \varepsilon^2 \ar[l,swap,"\varepsilon^{[1,2]}"] & \varepsilon^3 \ar[l,swap,"\varepsilon^{[2,3]}"]
	\end{tikzcd}
\]where $\varepsilon^k \coloneq (\varepsilon_1,\ldots,\varepsilon_{i-1},k,\varepsilon_{i+1},\ldots,\varepsilon_b)$ for $k \in \{0,1,2,3\}$. Consider the components of the differential of $\sr{K}$ associated with these edges \[
	\begin{tikzcd}[column sep=60pt]
		V({\varepsilon^{0}}) & V({\varepsilon^{1}}) \ar[l,swap,"\textstyle \phi"] & V({\varepsilon^{2}}) \ar[l,swap,"\textstyle \psi"] & V({\varepsilon^{3}}) \ar[l,swap,"\textstyle \chi"]
	\end{tikzcd}
\]We define the component $\zeta\colon W(\lambda^{j+1}) \to W(\lambda^j)$ of the differential of $\sr{P}$ assigned to the edge $\lambda^{[j,j+1]}$ to be the following composite\[
	\begin{cases}
		\begin{aligned}
		\begin{tikzcd}[column sep=60pt,row sep=25pt]
			W(\lambda^j) & W(\lambda^{j+1}) \ar[d,swap,"\textstyle\iota^{g(\lambda^{j+1})}"] \ar[l,dotted,swap,"\textstyle\zeta \coloneq \pi^{g(\lambda^j)}\,p_{g(\lambda^j)}\,\phi\,\iota^{g(\lambda^{j+1})}"]\\
			q^{H(\lambda^j) - G(\varepsilon^0) + \xi(g(\lambda^j))}V(\varepsilon^0) \ar[u,"\textstyle\pi^{g(\lambda^j)}\,p_{g(\lambda^j)}"] & q^{H(\lambda^{j+1})-G(\varepsilon^1) + \xi(g(\lambda^{j+1}))} V(\varepsilon^1) \ar[l,swap,"\textstyle\phi"]
		\end{tikzcd} \hspace{15pt} & j = 0\\[10pt]
		\begin{tikzcd}[column sep=60pt,row sep=25pt]
			W(\lambda^j) & W(\lambda^{j+1}) \ar[d,swap,"\textstyle\iota^{g(\lambda^{j+1})}"] \ar[l,dotted,swap,"\textstyle\zeta\coloneqq \pi^{g(\lambda^j)}\,p_{g(\lambda^j)}\,\psi\,\iota^{g(\lambda^{j+1})}"]\\
			q^{H(\lambda^j) - G(\varepsilon^1) + \xi(g(\lambda^j))} V(\varepsilon^1) \ar[u,"\textstyle\pi^{g(\lambda^j)}\,p_{g(\lambda^j)}"] & q^{H(\lambda^{j+1})-G(\varepsilon^2) + \xi(g(\lambda^{j+1}))} V(\varepsilon^2) \ar[l,swap,"\textstyle\psi"]
		\end{tikzcd} \hspace{15pt} & j \text{ is odd}\\[10pt]
		\begin{tikzcd}[column sep=60pt,row sep=25pt]
			W(\lambda^j) & W(\lambda^{j+1}) \ar[d,swap,"\textstyle\iota^{g(\lambda^{j+1})}"] \ar[l,dotted,swap,"\textstyle\zeta \coloneqq \pi^{g(\lambda^j)}\,p_{g(\lambda^j)}\,\chi\,\phi\,\iota^{g(\lambda^{j+1})}"]\\
			q^{H(\lambda^{j})-G(\varepsilon^2) + \xi(g(\lambda^{j}))} V(\varepsilon^2) \ar[u,"\textstyle\pi^{g(\lambda^j)}\,p_{g(\lambda^j)}"] & q^{H(\lambda^{j+1})-G(\varepsilon^1) + \xi(g(\lambda^{j+1}))} V(\varepsilon^1) \ar[l,swap,"\textstyle\chi\,\phi"]
		\end{tikzcd} \hspace{15pt} & j \text{ is even and positive.}
	\end{aligned}
	\end{cases}
\]See Lemma~\ref{lem:relationsP} for the notation $\pi^g \,p_g$. The differential is $q$-homogeneous by Lemma~\ref{lem:gradingFunctionP}. 

\begin{lem}\label{lem:differentialFactorsP}
	The component $\zeta$ of the differential of $\sr{P}$ assigned to $\lambda^{[j,j+1]}$ is the unique map that makes the following diagram commute \[
		\begin{cases}
		\begin{aligned}
		\begin{tikzcd}[column sep=60pt,row sep=25pt]
			W(\lambda^j) \ar[d,swap,"\textstyle \iota^{g(\lambda^j)}"] & W(\lambda^{j+1}) \ar[d,swap,"\textstyle\iota^{g(\lambda^{j+1})}"] \ar[l,dotted,swap,"\textstyle\zeta"]\\
			q^{H(\lambda^j) - G(\varepsilon^0) + \xi(g(\lambda^j))}V(\varepsilon^0) & q^{H(\lambda^{j+1})-G(\varepsilon^1) + \xi(g(\lambda^{j+1}))} V(\varepsilon^1) \ar[l,swap,"\textstyle\phi"]
		\end{tikzcd} \hspace{15pt} & j = 0\\[10pt]
		\begin{tikzcd}[column sep=60pt,row sep=25pt]
			W(\lambda^j) \ar[d,swap,"\textstyle \iota^{g(\lambda^j)}"] & W(\lambda^{j+1}) \ar[d,swap,"\textstyle\iota^{g(\lambda^{j+1})}"] \ar[l,dotted,swap,"\textstyle\zeta"]\\
			q^{H(\lambda^j) - G(\varepsilon^1) + \xi(g(\lambda^j))} V(\varepsilon^1) & q^{H(\lambda^{j+1})-G(\varepsilon^2) + \xi(g(\lambda^{j+1}))} V(\varepsilon^2) \ar[l,swap,"\textstyle\psi"]
		\end{tikzcd} \hspace{15pt} & j \text{ is odd}\\[10pt]
		\begin{tikzcd}[column sep=60pt,row sep=25pt]
			W(\lambda^j) \ar[d,swap,"\textstyle \iota^{g(\lambda^j)}"] & W(\lambda^{j+1}) \ar[d,swap,"\textstyle\iota^{g(\lambda^{j+1})}"] \ar[l,dotted,swap,"\textstyle\zeta"]\\
			q^{H(\lambda^{j})-G(\varepsilon^2) + \xi(g(\lambda^{j}))} V(\varepsilon^2) & q^{H(\lambda^{j+1})-G(\varepsilon^1) + \xi(g(\lambda^{j+1}))} V(\varepsilon^1) \ar[l,swap,"\textstyle\chi\,\phi"]
		\end{tikzcd} \hspace{15pt} & j \text{ is even and positive.}
		\end{aligned}
		\end{cases}
	\]
\end{lem}
\begin{proof}
	We show that the diagram is commutative. Uniqueness follows from injectivity of the vertical maps. 

	\underline{Case}: $j = 0$. Then $\lambda^{j+1} = (\lambda_1,\ldots,\lambda_{i-1},1,0,\ldots,0)$ and $\lambda^j = (\lambda_1,\ldots,\lambda_{i-1},0,0,\ldots,0)$. We note that $\varepsilon(\lambda^{j+1}) = \varepsilon^1$ and $\varepsilon(\lambda^j) = \varepsilon^0$. By definition of the component $\phi$ of the differential of $\sr{K}$, we have that \[
		\phi = \partial_{b-1}\,\cdots\,\partial_{i+1}\,\partial_i\,Z_{(i-1)i}
	\]since $r(\varepsilon^1) = r(\lambda^{j+1}) = i$. By Lemma~\ref{lem:relationsP}, it suffices to show that $\partial_k\,\phi\,\iota^{g(\lambda^{j+1})} = 0$ for all $k \in \{1,\ldots,b-1\}$ for which the $k$th and $(k+1)$th coordinates of $\lambda^j$ are equal. 

	If $k$ is such an index and satisfies $k < i-1$, then $\partial_k\,\phi = \phi\,\partial_k$ by Lemma~\ref{lem:relationsK}. Since the $k$th and $(k+1)$th coordinates of $\lambda^{j+1}$ are the same as those of $\lambda^j$, we have that $\partial_k\,\iota^{g(\lambda^{j+1})} = 0$, so $\partial_k\,\phi\,\iota^{g(\lambda^{j+1})} = 0$ as required. If $k \ge i$, then \begin{align*}
		\partial_k\,\phi &= \partial_{b-1}\,\cdots\,\partial_{k+2}\,\partial_k\,\partial_{k+1}\,\partial_k\,\partial_{k-1}\,\cdots\,\partial_i\,Z_{(i-1)i} = \partial_{b-1}\,\cdots\,\partial_{k+2}\,\partial_{k+1}\,\partial_{k}\,\partial_{k+1}\,\partial_{k-1}\,\cdots\,\partial_i\,Z_{(i-1)i} = \phi\,\partial_{k+1}
	\end{align*}Since $\partial_{k+1}\,\iota^{g(\lambda^{j+1})} = 0$, we see that $\partial_k\,\phi\,\iota^{g(\lambda^{j+1})} = 0$ as required. 

	\underline{Case}: $j$ is odd. Let $r\coloneqq r(\lambda^{j+1}) = r(\lambda^j)$ and note that \begin{align*}
		\varepsilon(\lambda^{j+1}) = \varepsilon^2 &= (\varepsilon_1,\ldots,\varepsilon_{i-1},2,\varepsilon_{i+1},\ldots,\varepsilon_r,0,\ldots,0)\\
		\varepsilon(\lambda^{j}) = \varepsilon^1 &= (\varepsilon_1,\ldots,\varepsilon_{i-1},1,\varepsilon_{i+1},\ldots,\varepsilon_r,0,\ldots,0)
	\end{align*}where $\varepsilon_1,\ldots,\varepsilon_{i-1},\varepsilon_{i+1},\ldots,\varepsilon_r \in \{1,2\}$. Let $\beta^*$ be the ascending string with smallest subscript $i$ and largest subscript $r-1$ obtained from $\varepsilon_{i+1},\ldots,\varepsilon_r$ by replacing $1$ by $\partial^*$ and $2$ by $s^*$. Then $\psi = \beta^*\,Q_r\,\hat{\beta}$. By Lemma~\ref{lem:relationsP}, it suffices to show that $\partial_k\,\psi\,\iota^{g(\lambda^{j+1})} = 0$ for all $k \in \{1,\ldots,b-1\}$ for which the $k$th and $(k+1)$th coordinates of $\lambda^j$ are equal. 

	If $k > r$ or $k < i-1$, then $\partial_k\,\psi\,\iota^{g(\lambda^{j+1})} = \psi\,\partial_k\,\iota^{g(\lambda^{j+1})} = 0$ by Lemma~\ref{lem:relationsK}. If $i < k < r$, then $\lambda_k = \lambda_{k+1}$ and $\varepsilon_k = \varepsilon_{k+1}$. It follows that the two symbols within $\beta^*$ with subscripts $k-1$ and $k$ are either both $\partial^*$ or both $s^*$. Since $\partial_k\,\theta^*\hspace{-4pt}_{k-1}\,\theta^*\hspace{-4pt}_k = \theta^*\hspace{-4pt}_{k-1}\,\theta^*\hspace{-4pt}_k\,\partial_{k-1}$ for $\theta \in \{\partial,s\}$ by Lemma~\ref{lem:mixedBraidRelation}, we find that $\partial_k\,\psi\,\iota^{g(\lambda^{j+1})} = \psi\,\partial_k\,\iota^{g(\lambda^{j+1})} = 0$ again. 

	The last possible value of $k \in \{1,\ldots,b\}\setminus \{g_1,g_1+g_2,\ldots,r\}$ is $k = i$. In this case, $j = \lambda_{i+1}$ so $\varepsilon_{i+1} = 1$ because $j$ is odd. The first symbol of $\beta^*$ is therefore $\partial^*\hspace{-4pt}_i$ so $\partial_i\,\psi\,\iota^{g(\lambda^{j+1})} = 0$ because $\partial_i \,\partial^*\hspace{-4pt}_i = 0$. Thus $\partial_k \,\psi\,\iota^{g(\lambda^{j+1})} = 0$ for all $k \in \{1,\ldots,b\}\setminus\{g_1,g_1+g_2,\ldots,r\}$ as required. 

	\underline{Case}: $j$ is even and positive. Again let $r\coloneq r(\lambda^{j+1}) = r(\lambda^j)$ and note that \begin{align*}
		\varepsilon(\lambda^{j+1}) = \varepsilon^1 &= (\varepsilon_1,\ldots,\varepsilon_{i-1},1,\varepsilon_{i+1},\ldots,\varepsilon_r,0,\ldots,0)\\
		\varepsilon(\lambda^{j}) = \varepsilon^2 &= (\varepsilon_1,\ldots,\varepsilon_{i-1},2,\varepsilon_{i+1},\ldots,\varepsilon_r,0,\ldots,0)
	\end{align*}where $\varepsilon_1,\ldots,\varepsilon_{i-1},\varepsilon_{i+1},\ldots,\varepsilon_r \in \{1,2\}$. 
	Let $\beta^*$ be the ascending string from the previous case. Then \[
		\chi\,\phi = \hat{\beta}^*\,Z_{r(r-1)} \,s^*\hspace{-4pt}_r \,\cdots\, s^*\hspace{-4pt}_{b-1}\, \partial_{b-1}\,\cdots\,\partial_r\,Z_{(r-1)r} \,\beta.
	\]Again by Lemma~\ref{lem:relationsP}, it suffices to show that $\partial_k \,\chi\,\phi\,\iota^{g(\lambda^{j+1})} = 0$ for all $k \in \{1,\ldots,b\}$ for which the $k$th and $(k+1)$th coordinates of $\lambda^j$ are equal. 

	If $k < i-1$, then $\partial_k\,\chi\,\phi\,\iota^{g(\lambda^{j+1})} = \chi\,\phi\,\partial_k\,\iota^{g(\lambda^{j+1})} = 0$. The $(i-1)$th and $i$th coordinates of $\lambda^j$ are different so we do not need to consider the possibility that $k = i-1$. If $k = i$, then $j = \lambda_{i+1}$ so $\varepsilon_{i+1} = 2$ because $j$ is even and positive. The first symbol of $\hat{\beta}^*$ is therefore $\partial^*\hspace{-4pt}_i$ so $\partial_i\,\chi\,\phi\,\iota^{g(\lambda^{j+1})} = 0$. If $i < k < r$, then $\lambda_k = \lambda_{k+1}$ and $\varepsilon_k = \varepsilon_{k+1}$ so the symbols in $\hat{\beta}^*$ with subscripts $k-1$ and $k$ are either both $\partial^*$ or both $s^*$. In either case, we have $\partial_k\,\theta^*\hspace{-4pt}_{k-1}\,\theta^*\hspace{-4pt}_k = \theta^*\hspace{-4pt}_{k-1}\,\theta^*\hspace{-4pt}_k\,\partial_{k-1}$ for $\theta \in \{\partial,s\}$ by Lemma~\ref{lem:mixedBraidRelation} which implies that $\partial_k\,\chi\,\phi\,\iota^{g(\lambda^{j+1})} = \chi\,\phi\,\partial_k\,\iota^{g(\lambda^{j+1})} = 0$. The $r$th and $(r+1)$th coordinates of $\lambda^j$ are different so we do not need to consider the possibility that $k = r$. Lastly, if $r < k$, then the mixed braid relation again gives us $\partial_k\,\chi\,\phi\,\iota^{g(\lambda^{j+1})} = \chi\,\phi\,\partial_k\,\iota^{g(\lambda^{j+1})} = 0$ as required.
\end{proof}

\begin{lem}\label{lem:consecutiveP}
	Consecutive components of the differential that are assigned to parallel edges compose to zero. 
\end{lem}
\begin{proof}
	Consider two consecutive edges $\lambda^{[j,j+1]}$ and $\lambda^{[j+1,j+2]}$ of the standard cubulation of $\R^b$ that are contained in $T$, and let $\zeta^{[j,j+1]}$ and $\zeta^{[j+1,j+2]}$ be their associated components of the differential. If $j = 0$, then Lemma~\ref{lem:differentialFactorsP} gives the following commutative diagram \[
		\begin{tikzcd}[column sep=55pt,row sep=25pt]
			W(\lambda^j) \ar[d,swap,"\textstyle \iota^{g(\lambda^j)}"] & W(\lambda^{j+1}) \ar[d,swap,"\textstyle\iota^{g(\lambda^{j+1})}"] \ar[l,dotted,swap,"\textstyle\zeta^{[j,j+1]}"] & W(\lambda^{j+2}) \ar[l,dotted,swap,"\textstyle\zeta^{[j+1,j+2]}"] \ar[d,swap,"\textstyle\iota^{g(\lambda^{j+2})}"]\\
			q^{H(\lambda^j) - G(\varepsilon^0) + \xi(g(\lambda^j))}V(\varepsilon^0) & q^{H(\lambda^{j+1})-G(\varepsilon^1) + \xi(g(\lambda^{j+1}))} V(\varepsilon^1) \ar[l,swap,"\textstyle\phi"] & q^{H(\lambda^{j+2})-G(\varepsilon^2) + \xi(g(\lambda^{j+2}))} V(\varepsilon^2) \ar[l,swap,"\textstyle\psi"]
		\end{tikzcd} 
	\]By Lemma~\ref{lem:consecutiveK}, we know that $\phi\,\psi = 0$ so commutativity of the diagram gives $\iota^{g(\lambda^j)}\,\zeta^{[j,j+1]}\,\zeta^{[j+1,j+2]} = 0$. By Lemma~\ref{lem:relationsP}, we have $\zeta^{[j,j+1]}\,\zeta^{[j+1,j+2]} = \pi^{g(\lambda^j)}\,p_{g(\lambda^j)}\,\iota^{g(\lambda^j)}\,\zeta^{[j,j+1]}\,\zeta^{[j+1,j+2]} = 0$ as required. The other cases follow in the same manner. 
\end{proof}

\begin{prop}\label{prop:commutativeSquaresP}
	The square associated to each face of the standard cubulation of $\R^b$ that is contained in $T$ is commutative. 
\end{prop}
\begin{proof}
	Just as in the proof of Lemma~\ref{lem:gradingFunctionP}, fix $1 \leq i_1 < i_2 \leq b$ and integers satisfying \[
		\lambda_1 \ge \cdots \lambda_{i_1-1} \ge j_1 + 1 > j_1 \ge \lambda_{i_1+1} \ge \cdots \ge \lambda_{i_2-1} \ge j_2+1 \ge j_2 \ge \lambda_{i_2+1} \ge \cdots \ge \lambda_b \ge 0.
	\]Consider the following four edges of the standard cubulation of $\R^b$ that are contained in $T$, and their associated components of the differential of $\sr{P}$. \[
		\begin{tikzcd}[column sep={120pt,between origins},row sep={45pt,between origins}]
			\lambda^{j_1,j_2+1} \ar[d,swap,"\lambda^{j_1,[j_2,j_2+1]}"] & \lambda^{j_1+1,j_2+1} \ar[l,swap,"\lambda^{[j_1,j_1+1],j_2+1}"] \ar[d,swap,"\lambda^{j_1+1,[j_2,j_2+1]}"]\\
			\lambda^{j_1,j_2} & \lambda^{j_1+1,j_2} \ar[l,swap,"\lambda^{[j_1,j_1+1],j_2}"]
		\end{tikzcd} \hspace{30pt} \begin{tikzcd}[column sep={120pt,between origins},row sep={45pt,between origins}]
			W(\lambda^{j_1,j_2+1}) \ar[d,swap,"\zeta^{j_1,[j_2,j_2+1]}"] & W(\lambda^{j_1+1,j_2+1}) \ar[l,swap,"\zeta^{[j_1,j_1+1],j_2+1}"] \ar[d,swap,"\zeta^{j_1+1,[j_2,j_2+1]}"]\\
			W(\lambda^{j_1,j_2}) & W(\lambda^{j_1+1,j_2}) \ar[l,swap,"\zeta^{[j_1,j_1+1],j_2}"]
		\end{tikzcd}
	\]Set $\varepsilon_k \coloneq \varepsilon(\lambda_k)$ for $k \in\{1,\ldots,b\}\setminus\{i_1,i_2\}$ and consider the following slice of $\sr{K}$. \[
		\begin{tikzcd}[column sep={80pt,between origins}, row sep={50pt,between origins}]
			V(\varepsilon^{03}) \ar[d,swap,"\textstyle \chi^{0,[2,3]}"] & V(\varepsilon^{13}) \ar[l,swap,"\textstyle \phi^{[0,1],3}"] \ar[d,swap,"\textstyle \chi^{1,[2,3]}"] & V(\varepsilon^{23}) \ar[l,swap,"\textstyle \psi^{[1,2],3}"] \ar[d,swap,"\textstyle \chi^{2,[2,3]}"] & V(\varepsilon^{33}) \ar[l,swap,"\textstyle \chi^{[2,3],3}"] \ar[d,swap,"\textstyle \chi^{3,[2,3]}"]\\
			V(\varepsilon^{02}) \ar[d,swap,"\textstyle \psi^{0,[1,2]}"] & V(\varepsilon^{12}) \ar[l,swap,"\textstyle \phi^{[0,1],2}"] \ar[d,swap,"\textstyle \psi^{1,[1,2]}"] & V(\varepsilon^{22}) \ar[l,swap,"\textstyle \psi^{[1,2],2}"] \ar[d,swap,"\textstyle \psi^{2,[1,2]}"] & V(\varepsilon^{32}) \ar[l,swap,"\textstyle \chi^{[2,3],2}"] \ar[d,swap,"\textstyle \psi^{3,[1,2]}"]\\
			V(\varepsilon^{01}) \ar[d,swap,"\textstyle \phi^{0,[0,1]}"] & V(\varepsilon^{11}) \ar[l,swap,"\textstyle \phi^{[0,1],1}"] \ar[d,swap,"\textstyle \phi^{1,[0,1]}"] & V(\varepsilon^{21}) \ar[l,swap,"\textstyle \psi^{[1,2],1}"] \ar[d,swap,"\textstyle \phi^{2,[0,1]}"] & V(\varepsilon^{31}) \ar[l,swap,"\textstyle \chi^{[2,3],1}"] \ar[d,swap,"\textstyle \phi^{3,[0,1]}"]\\
			V(\varepsilon^{00}) & V(\varepsilon^{10}) \ar[l,swap,"\textstyle \phi^{[0,1],0}"] & V(\varepsilon^{20}) \ar[l,swap,"\textstyle \psi^{[1,2],0}"] & V(\varepsilon^{30}) \ar[l,swap,"\textstyle \chi^{[2,3],0}"]
		\end{tikzcd}
	\]

	\underline{Case}: $j_1$ is odd and $j_2 = 0$. Consider the following cube \[
		\begin{tikzcd}[column sep={80pt,between origins}, row sep={40pt,between origins}]
			& W(\lambda^{j_1,j_2+1}) \ar[ld,swap,"\zeta^{j_1,[j_2,j_2+1]}" {xshift=8pt}] \ar[dd,swap,"\iota^{g(\lambda^{j_1,j_2+1})}" {yshift=-15pt}] & & W(\lambda^{j_1+1,j_2+1}) \ar[ll,swap,"\zeta^{[j_1,j_1+1],j_2+1}"] \ar[ld,swap,"\zeta^{j_1+1,[j_2,j_2+1]}" {xshift=8pt}] \ar[dd,swap,"\iota^{g(\lambda^{j_1+1,j_2+1})}" {yshift=-15pt}]\\
			W(\lambda^{j_1,j_2}) \ar[dd,swap,"\iota^{g(\lambda^{j_1,j_2})}" {yshift=-15pt}] & & W(\lambda^{j_1+1,j_2}) \ar[ll,swap,crossing over,"\zeta^{[j_1,j_1+1],j_2}" {xshift=30pt}] &\\
			& q^{K(\lambda^{j_1,j_2+1})}V(\varepsilon^{11}) \ar[dl,swap,"\phi^{1,[0,1]}" {xshift=8pt}] & & q^{K(\lambda^{j_1+1,j_2+1})}V(\varepsilon^{21}) \ar[ll,swap,"\psi^{[1,2],1}" {xshift=30pt}] \ar[dl,swap,"\phi^{2,[0,1]}" {xshift=8pt}] \\
			q^{K(\lambda^{j_1,j_2})}V(\varepsilon^{10}) & & q^{K(\lambda^{j_1+1,j_2})}V(\varepsilon^{20}) \ar[ll,swap,"\psi^{[1,2],0}"] \ar[from=uu,crossing over,swap,"\iota^{g(\lambda^{j_1+1,j_2})}" {yshift=-15pt}] &
		\end{tikzcd}
	\]where $K(\lambda^{k_1,k_2}) \coloneqq H(\lambda^{k_1,k_2}) - G(\varepsilon^{\varepsilon(k_1),\varepsilon(k_2)}) + \xi(g(\lambda^{k_1,k_2}))$. The four vertical faces are commutative by Lemma~\ref{lem:differentialFactorsP}, and the bottom face is commutative by Proposition~\ref{prop:commutativeSquaresK}. Injectivity of the vertical maps implies that the top face is commutative.

	\underline{Case}: $j_1$ is even and positive and $j_2 = 0$. Just as in the previous case, it suffices to verify that \[
		\begin{tikzcd}[column sep=60pt]
			q^{K(\lambda^{j_1,j_2+1})} V(\varepsilon^{21}) \ar[d,swap,"\phi^{2,[0,1]}"] & q^{K(\lambda^{j_1+1,j_2+1})} V(\varepsilon^{11}) \ar[d,swap,"\phi^{1,[0,1]}"] \ar[l,swap,"\chi^{[2,3],1}\,\phi^{[0,1],1}"] \\
			q^{K(\lambda^{j_1,j_2})} V(\varepsilon^{20}) & q^{K(\lambda^{j_1+1,j_2})} V(\varepsilon^{10}) \ar[l,swap,"\chi^{[2,3],0}\,\phi^{[0,1],0}"]
		\end{tikzcd}
	\]is commutative. The key observation is that $\phi^{3,[0,1]}$ and $\phi^{0,[0,1]}$ are equal because $r(\varepsilon^{0,1}) = r(\varepsilon^{3,1})$ and the two edges $\varepsilon^{0,[0,1]}$ and $\varepsilon^{3,[0,1]}$ have the same last $b-i_2$ coordinates. Hence \begin{align*}
		\phi^{2,[0,1]}\,\chi^{[2,3],1}\,\phi^{[0,1],1} &= \chi^{[2,3],0}\,\phi^{3,[0,1]}\,\phi^{[0,1],1} = \chi^{[2,3],0}\,\phi^{0,[0,1]}\,\phi^{[0,1],1} = \chi^{[2,3],0}\,\phi^{[0,1],0}\,\phi^{1,[0,1]}
	\end{align*}as required.

	\underline{Case}: $j_1$ is even and positive and $j_2$ is odd. It suffices to verify that \[
		\begin{tikzcd}[column sep=60pt]
			q^{K(\lambda^{j_1,j_2+1})} V(\varepsilon^{22}) \ar[d,swap,"\psi^{2,[1,2]}"] & q^{K(\lambda^{j_1+1,j_2+1})} V(\varepsilon^{12}) \ar[d,swap,"\psi^{1,[1,2]}"] \ar[l,swap,"\chi^{[2,3],2}\,\phi^{[0,1],2}"] \\
			q^{K(\lambda^{j_1,j_2})} V(\varepsilon^{21}) & q^{K(\lambda^{j_1+1,j_2})} V(\varepsilon^{11}) \ar[l,swap,"\chi^{[2,3],1}\,\phi^{[0,1],1}"]
		\end{tikzcd}
	\]is commutative. The key here is to observe that $\psi^{3,[1,2]} = \psi^{0,[1,2]}$ so \[
		\psi^{2,[1,2]}\,\chi^{[2,3],2}\,\phi^{[0,1],2} = \chi^{[2,3],1}\,\psi^{3,[1,2]}\,\phi^{[0,1],2} = \chi^{[2,3],1}\,\psi^{0,[1,2]}\,\phi^{[0,1],2} = \chi^{[2,3],1}\,\phi^{[0,1],1}\,\psi^{1,[1,2]}.
	\]

	\underline{Case}: $j_1$ and $j_2$ are odd. This follows from $\psi^{1,[1,2]}\,\psi^{[1,2],2} = \psi^{[1,2],1}\,\psi^{2,[1,2]}$. 

	\underline{Case}: $j_1$ is odd and $j_2$ is even and positive. It suffices to verify that \[
		\begin{tikzcd}[column sep=50pt,row sep=25pt]
			q^{K(\lambda^{j_1,j_2+1})} V(\varepsilon^{11}) \ar[d,swap,"\chi^{1,[2,3]}\,\phi^{1,[0,1]}"] & q^{K(\lambda^{j_1+1,j_2+1})} V(\varepsilon^{21}) \ar[d,swap,"\chi^{2,[2,3]}\,\phi^{2,[0,1]}"] \ar[l,swap,"\psi^{[1,2],1}"] \\
			q^{K(\lambda^{j_1,j_2})} V(\varepsilon^{12}) & q^{K(\lambda^{j_1+1,j_2})} V(\varepsilon^{22}) \ar[l,swap,"\psi^{[1,2],2}"]
		\end{tikzcd}
	\]is commutative. We observe that $\psi^{[1,2],0} = \psi^{[1,2],3}$ because the last $b-i_1$ coordinates of $\varepsilon^{[1,2],0}$ and $\varepsilon^{[1,2],3}$ become the same after deleting the $0$'s and $3$'s. Thus \[
		\chi^{1,[2,3]}\,\phi^{1,[0,1]}\,\psi^{[1,2],1} = \chi^{1,[2,3]}\,\psi^{[1,2],0}\,\phi^{2,[0,1]} =  \chi^{1,[2,3]}\,\psi^{[1,2],3}\,\phi^{2,[0,1]} = \psi^{[1,2],2}\,\chi^{2,[2,3]}\,\phi^{2,[0,1]}. 
	\]

	\underline{Case}: $j_1$ and $j_2$ are even and positive. It suffices to show that \[
		\begin{tikzcd}[column sep=60pt,row sep=25pt]
			q^{K(\lambda^{j_1,j_2+1})} V(\varepsilon^{21}) \ar[d,swap,"\chi^{2,[2,3]}\,\phi^{2,[0,1]}"] & q^{K(\lambda^{j_1+1,j_2+1})} V(\varepsilon^{11}) \ar[d,swap,"\chi^{1,[2,3]}\,\phi^{1,[0,1]}"] \ar[l,swap,"\chi^{[2,3],1}\,\phi^{[0,1],1}"]\\
			q^{K(\lambda^{j_1,j_2})} V(\varepsilon^{22}) & q^{K(\lambda^{j_1+1,j_2})} V(\varepsilon^{12}) \ar[l,swap,"\chi^{[2,3],2}\,\phi^{[0,1],2}"]
		\end{tikzcd}
	\]is commutative. Observe that $\phi^{3,[0,1]} = \phi^{0,[0,1]}$ because the two edges have the same last $b - i_2$ coordinates. Hence \begin{align*}
		\chi^{2,[2,3]}\,\phi^{2,[0,1]}\,\chi^{[2,3],1}\,\phi^{[0,1],1} &= \chi^{2,[2,3]}\,\chi^{[2,3],0}\,\phi^{3,[0,1]}\,\phi^{[0,1],1}\\
		&= \chi^{2,[2,3]}\,\chi^{[2,3],0}\,\phi^{0,[0,1]}\,\phi^{[0,1],1} = \chi^{2,[2,3]}\,\chi^{[2,3],0}\,\phi^{[0,1],0}\,\phi^{1,[0,1]}
	\end{align*}Similarly, we have $\chi^{0,[2,3]} = \chi^{3,[2,3]}$ so \begin{align*}
		\chi^{[2,3],2}\,\phi^{[0,1],2}\,\chi^{1,[2,3]}\,\phi^{1,[0,1]} &= \chi^{[2,3],2}\,\chi^{0,[2,3]}\,\phi^{[0,1],3}\,\phi^{1,[0,1]}\\
		&= \chi^{[2,3],2}\,\chi^{3,[2,3]}\,\phi^{[0,1],3}\,\phi^{1,[0,1]} = \chi^{2,[2,3]}\,\chi^{[2,3],3}\,\phi^{[0,1],3}\,\phi^{1,[0,1]}
	\end{align*}It now suffices to show that $\chi^{[2,3],0}\,\phi^{[0,1],0} = \chi^{[2,3],3}\,\phi^{[0,1],3}$. Let $r \coloneq r(\lambda^{j_1,j_2}) - 1$, and let $\beta^*$ be the ascending string with largest subscript $r - 1$ obtained from the sequence $\varepsilon_{i_1+1},\ldots,\varepsilon_{i_2-1},\varepsilon_{i_2+1},\ldots,\varepsilon_b$ by deleting the $0$'s and $3$'s and replacing $1$ by $\partial^*$ and $2$ by $s^*$. For $k \in \{0,3\}$, let $\alpha_k$ be the descending string with smallest subscript $r$ obtained from the sequence $\varepsilon_{i_1+1},\ldots,\varepsilon_{i_2-1},k,\varepsilon_{i_2+1},\ldots,\varepsilon_b$ by deleting the $1$'s and $2$'s and replacing $0$ by $\partial$ and $3$ by $s$. It follows that \[
		\alpha_0 = \partial_{b-1}\,\partial_{b-2}\,\cdots\,\partial_r \hspace{30pt} \alpha_3 = s_{b-1}\,\partial_{b-2}\,\cdots\,\partial_r
	\]because $\varepsilon_{i_1+1},\ldots,\varepsilon_{i_2-1},\varepsilon_{i_2+1},\ldots,\varepsilon_{r+1} \in \{1,2\}$ and $\varepsilon_{r+2} = \cdots = \varepsilon_b = 0$. Thus \begin{align*}
		\chi^{[2,3],0}\,\phi^{[0,1],0} &= \hat{\beta}^* \, Z_{r(r-1)} \hat{\alpha}^*\hspace{-4pt}_0 \,\alpha_0\,Z_{(r-1)r}\,\beta = \hat{\beta}^*\,Z_{r(r-1)} s^*\hspace{-4pt}_r\,\cdots\,s^*\hspace{-4pt}_{b-2}\,s^*\hspace{-4pt}_{b-1}\,\partial_{b-1}\,\partial_{b-2}\,\cdots\,\partial_r\,Z_{(r-1)r}\,\beta\\
		\chi^{[2,3],3}\,\phi^{[0,1],3} &= \hat{\beta}^* \, Z_{r(r-1)} \hat{\alpha}^*\hspace{-4pt}_3 \,\alpha_3\,Z_{(r-1)r}\,\beta = \hat{\beta}^*\,Z_{r(r-1)} s^*\hspace{-4pt}_r\,\cdots\,s^*\hspace{-4pt}_{b-2}\,\partial^*\hspace{-4pt}_{b-1}\,s_{b-1}\,\partial_{b-2}\,\cdots\,\partial_r\,Z_{(r-1)r}\,\beta.
	\end{align*}The identity $s^*\hspace{-4pt}_{b-1}\,\partial_{b-1} = -\partial_{b-1} = \partial^*\hspace{-4pt}_{b-1}\, s_{b-1}$ finishes the proof.
\end{proof}


\section{Main theorem}\label{sec:main_theorem}

In this section, we prove the following theorem. As before, $a,b,c,d$ are positive integers for which $a + b = c + d$ and $b = \min(a,b,c,d)$. Set $n \coloneqq a + b = c + d$ and $l \coloneqq c - b = a - d$. 

\begin{thm}\label{thm:mainThm}
	The complex $\sr{P} \coloneqq {}^b_a\sr{P}^c_d$ with its filtration $\sr{F}^0(\sr{P}) \subset \sr{F}^1(\sr{P}) \subset \cdots$ satisfies the following properties. \begin{enumerate}[noitemsep]
		\item The complexes $\sr{P}$ and $\sr{F}^k(\sr{P})$ for $k \ge 0$ are minimal.
		\item The subcomplex $\sr{F}^k(\sr{P})$ is homotopy equivalent to the Rickard complex of \[
			\begin{cases}
				\quad\begin{gathered}
					\labellist
					\pinlabel $b$ at -2 5.5
					\pinlabel $a$ at -2 0.5
					\endlabellist
					\includegraphics[width=.06\textwidth]{1halfTwist}
					\vspace{-10pt}
				\end{gathered}\hspace{3pt}\begin{gathered}
					\cdots\vspace{-5pt}
				\end{gathered}\hspace{3pt}\begin{gathered}
					\labellist
					\endlabellist
					\includegraphics[width=.06\textwidth]{1halfTwist}
					\vspace{-10pt}
				\end{gathered} \hspace{-3pt} \begin{gathered}
					\labellist
					\pinlabel $b$ at -1 7.2
					\pinlabel $a$ at -1 -1.5
					\pinlabel $d$ at 8.5 0.8
					\pinlabel $c$ at 8.5 5.5
					\endlabellist
					\includegraphics[width=0.034\textwidth]{rung_diag}
					\vspace{-10pt}
				\end{gathered}\vspace{20pt} \quad & k \text{ is even}\\
				\quad\begin{gathered}
					\labellist
					\pinlabel $b$ at -2 5.5
					\pinlabel $a$ at -2 0.5
					\endlabellist
					\includegraphics[width=.06\textwidth]{1halfTwist}
					\vspace{-3pt}
				\end{gathered}\hspace{3pt}\begin{gathered}
					\cdots
				\end{gathered}\hspace{3pt}\begin{gathered}
					\labellist
					\endlabellist
					\includegraphics[width=.06\textwidth]{1halfTwist}
					\vspace{-3pt}
				\end{gathered} \hspace{-3pt} \begin{gathered}
					\labellist
					\pinlabel $a$ at -1 7.1
					\pinlabel $b$ at -1 -1.9
					\pinlabel $d$ at 8.5 0.8
					\pinlabel $c$ at 8.5 5.5
					\endlabellist
					\includegraphics[width=0.034\textwidth]{rung}
					\vspace{-3pt}
				\end{gathered}\vspace{5pt} \quad & k \text{ is odd}
			\end{cases}
		\]where there are $k \ge 0$ positive crossings. The rung is colored by $c-b$ when $k$ is even and by $d - b$ when $k$ is odd.
		\item For $r \in \{1,\ldots,b\}$, the following four tensor products are contractible \[
			\sr{P} \otimes \quad \begin{gathered}
				\labellist
				\pinlabel $c$ at -2 5.3
				\pinlabel $d$ at -2 0.8
				\pinlabel $a+r$ at 11.5 0.5
				\pinlabel $b-r$ at 11.5 5.5
				\endlabellist
				\includegraphics[width=0.035\textwidth]{rung}
				\vspace{-3pt}
			\end{gathered} \hspace{25pt} \simeq 0 \hspace{20pt} \sr{P} \otimes \quad \begin{gathered}
				\labellist
				\pinlabel $c$ at -2 5.3
				\pinlabel $d$ at -2 0.8
				\pinlabel $b-r$ at 11.5 0.5
				\pinlabel $a+r$ at 11.5 5.5
				\endlabellist
				\includegraphics[width=0.035\textwidth]{rung_diag}
				\vspace{-3pt}
			\end{gathered} \hspace{25pt} \simeq 0 \hspace{20pt} \hspace{25pt}\begin{gathered}
				\labellist
				\pinlabel $b-r$ at -5 5.5
				\pinlabel $a+r$ at -5 0.5
				\pinlabel $a$ at 8.5 0.5
				\pinlabel $b$ at 8.5 5.5
				\endlabellist
				\includegraphics[width=0.035\textwidth]{rung_diag}
				\vspace{-3pt}
			\end{gathered} \quad \otimes \sr{P} \simeq 0 \hspace{20pt} \hspace{25pt}\begin{gathered}
				\labellist
				\pinlabel $a+r$ at -5 5.3
				\pinlabel $b-r$ at -5 0.8
				\pinlabel $a$ at 8.5 0.5
				\pinlabel $b$ at 8.5 5.5
				\endlabellist
				\includegraphics[width=0.035\textwidth]{rung}
				\vspace{-3pt}
			\end{gathered} \quad \otimes \sr{P} \simeq 0 \hspace{20pt} \vspace{2pt}
		\]where the rungs are colored by $c+r-b,d+r-b,b+r-b$, and $a+r-b$, respectively.
	\end{enumerate}
\end{thm}

Properties $1,2$, and $3$ are proven in sections~\ref{subsec:minimality}, \ref{subsec:Rickard_complexes}, and \ref{subsec:contractibility}, respectively. In section~\ref{subsec:proof_of_intro_main_theorem}, we show that Theorem~\ref{thm:mainThm} specializes to Theorem~\ref{thm:mainIntroThm} when $a = b = c = d$. For $r \in \{0,\ldots,b\}$, we use the following notational shorthand.\vspace{5pt} \[
	W_r \coloneqq W^r_r = \hspace{5pt}\generalwrr
\]

\subsection{Minimality}\label{subsec:minimality}

The following notion of minimality is valid for complexes over an additive category like the category of singular Bott--Samelson (or Soergel) bimodules with fixed boundary data. 

\begin{df}
	A chain complex $C$ is \textit{minimal} if every homotopy equivalence from $C$ to $C$ is an isomorphism. 
\end{df}
\begin{lem}
	If $g\colon C \to D$ is a homotopy equivalence between a minimal complex $C$ and a complex $D$, then there is a chain map $f\colon D \to C$ for which $f\,g = \Id_C$ and $g\,f$ is homotopic to $\Id_D$.
	In particular, any complex that is homotopy equivalent to a minimal one admits a deformation retract onto it. Also, any two minimal complexes that are homotopy equivalent are isomorphic.
\end{lem}
\begin{proof}
	The proof is routine and included here for the sake of completeness.
	By hypothesis, there exists a chain map $f'\colon D \to C$ and homotopies $h_C,h_D$ such that \[
		\Id_C -\, f'\, g = d\,h_C + h_C\,d \hspace{30pt} \Id_D -\, g\, f' = d\,h_D + h_D\,d.
	\]Then $f'\,g\colon C \to C$ is a homotopy equivalence and hence an isomorphism by the hypothesis that $C$ is minimal. Let $k\colon C \to C$ be its inverse chain map satisfying $k\,f'\,g = \Id_C$. Set $f \coloneq k\,f'$ and $h \coloneq h_D - g\,k\,h_C\,f'$. Then $f\,g = \Id_C$ and \begin{align*}
		d\,h + h\,d &= \left(d\,h_D + h_D\,d\right) - g\,k\,\left(d\,h_C + h_C\,d\right)\,f'\\
		&= \left(\Id_D - \,g\,f'\right) - g\,k\,(\Id_C - \,f'\,g)\,f'\\
		&= \Id_D - \,g\,f' - g\,k\,f' + g\,(k\,f'\,g)\,f' = \Id_D - \,g\,f
	\end{align*}which proves the first claim. If $D$ is also minimal, then $g\,f$ has an inverse $j$ satisfying $j\,g\,f = \Id_D$ so actually \[
		g\,f = (j\,g\,f)\,g\,f = j\,g\,(f\,g)\,f = j\,g\,f = \Id_D.\qedhere
	\]
\end{proof}

To prove that $\sr{P}$ and $\sr{F}^k(\sr{P})$ are minimal, we use the perverse filtration on singular Bott--Samelson bimodules, which we explain in our context of webs with boundary data $c_L = (a,b)$ and $c_R = (d,c)$. The following proposition follows from \cite[Theorem 2.8]{MR3177365}, \cite{MR2844932}, and \cite[Appendix B]{hogancamp2021skein}. 

\begin{prop}\label{prop:homSpaces}
	The singular Bott--Samelson bimodules $W_0,W_1,\ldots,W_b$ are indecomposable as bimodules. They are pairwise distinct in the sense that $q^i W_r \cong q^j W_s$ if and only if $r = s$ and $i = j$. Any singular Bott--Samelson bimodule with the boundary data $c_L,c_R$ is isomorphic to a finite direct sum of shifted copies of these indecomposable bimodules. The number of copies of each shifted indecomposable bimodule appearing in a decomposition is independent of the choice of decomposition. 
	Lastly $\Hom^{k}(W_r,W_s) = 0$ for $k < |r-s|(d-b+|r-s|)$ and \[
		\Hom^{|r-s|(d-b+|r-s|)}(W_r,W_s) \cong \Z.
	\]
\end{prop}

The result concerning morphism spaces can be verified directly using Proposition~\ref{prop:adjunction} and basic computations. We note that this morphism space computation implies that the bimodules $W_0,\ldots,W_b$ are indecomposable and distinct. Indeed, if $W_r$ were isomorphic to a nontrivial direct sum $B \oplus C$, then the idempotent projections onto each factor would be linearly independent in $\Hom^0(W_r,W_r) \cong \Z$. Furthermore, if there is an isomorphism in $\Hom^i(W_r,W_s)$ for some $i\in\Z$, then its inverse isomorphism would lie in $\Hom^{-i}(W_s,W_r)$. Hence both $i$ and $-i$ are nonnegative so $i = 0$ and $r = s$.

Given a finite direct sum \[
	W = \bigoplus_{m=1}^M q^{i_m} W_{r_m}
\]of shifted copies of $W_0,\ldots,W_b$, define an increasing filtration $\cdots \subseteq \sr{G}^k(W) \subseteq \sr{G}^{k+1}(W) \subseteq \cdots$ on $W$ by letting $\sr{G}^k(W)$ denote the direct sum of the summands $q^{i_m} W_{r_m}$ for which $i_m \leq k$. Any bimodule map from one such direct sum to another \[
	\bigoplus_{m=1}^M q^{i_m} W_{r_m} \to \bigoplus_{t=1}^T q^{j_t} W_{s_t}
\]has the property that the component from $q^{i_m} W_{r_m}$ to $q^{j_t}W_{s_t}$ is zero if $i_m < j_t$ by Proposition~\ref{prop:homSpaces}, so the map preserves the filtration. 

Any singular Bott--Samelson bimodule $W$ with boundary data $c_L,c_R$ is isomorphic to a direct sum of shifted copies of $W_0,\ldots,W_b$ by Proposition~\ref{prop:homSpaces}. The filtration $\sr{G}^k(W)$ on $W$ induced by a such an isomorphism does not depend on the choice of isomorphism by the above observation. Furthermore, every bimodule map preserves the filtration and has an induced associated graded map. As usual, the associated graded map of a composite is the composite of the associated graded maps, and a map is an isomorphism if its associated graded map is an isomorphism. 

\begin{proof}[Proof of property 1 in Theorem~\ref{thm:mainThm}]
	Suppose $f\colon \sr{P} \to \sr{P}$ and $g\colon \sr{P} \to \sr{P}$ are chain maps for which there are homotopies $h,h'$ of $\sr{P}$ with $\Id_{\sr{P}} - \,f\,g = d\,h + h\,d$ and $\Id_{\sr{P}} - \,g\,f = d\,h' + h'\,d$. To show that $f$ and $g$ are isomorphisms, it suffices to show that the composites $f\,g$ and $g\,f$ are isomorphisms. We do so by showing that in each cohomological degree, their associated graded maps with respect to the perverse filtration $\sr{G}^k$ are the identity maps. By the equations $\Id_{\sr{P}} - \,f\,g = d\,h + h\,d$ and $\Id_{\sr{P}} - \,g\,f = d\,h' + h'\,d$, it suffices to show that the associated graded map of $d$ is zero. 

	We show that the associated graded map of each component of the differential of $\sr{P}$ is zero. By its definition in section~\ref{subsec:the_differential_of_P}, each component is a composite of maps where at least one of the maps $Z_{r(r+1)},Z_{(r+1)r},Q_t$ appears in the composition. It suffices to show that the associated graded maps of $Z_{r(r+1)},Z_{(r+1)r},Q_t$ are zero.

	To see that the associated graded maps of $Z_{r(r+1)}$ and $Z_{(r+1)r}$ are zero, we note that the decomposition of $V_r$ into indecomposables has only shifted copies of $W_r$. In particular $V_r \cong [r]![b-r]! W_r$. The associated graded map of any map between shifts of indecomposables with different subscripts is zero by Proposition~\ref{prop:homSpaces}. 

	Recall that $Q_t$ is multiplication by \[
		e_{l+t}(\C - x_t - \cdots - x_b) = \sum_{j=0}^{l+t} (-1)^j e_{l+t-j}(\C) h_j(x_t,\ldots,x_b).
	\]Multiplication by $e_i(\C)$ commutes with every bimodule map, so when the map $e_i(\C)\colon q^{2i}V_r \to V_r$ is expressed in terms of a decomposition $V_r \cong [r]![b-r]!W_r$, every component map is either zero or multiplication by $e_i(\C)$. Hence the associated graded map of $e_i(\C)\colon q^{2i}V_r \to V_r$ is zero for $i > 0$ so the associated graded map of $Q_t$ is equal to the associated graded map of $(-1)^{l+t}h_{l+t}(x_1,\ldots,x_b)$. Next, note that by similar reasoning, the associated graded map of $e_{l+t}(\B - x_t - \cdots - x_b)$ is also equal to the associated graded map of $(-1)^{l+t}h_{l+t}(x_1,\ldots,x_b)$. However, we have that $e_{l+t}(\B - x_t - \cdots - x_b) = e_{l+t}(x_1,\ldots,x_{t-1}) = 0$ so the associated graded map of $Q_t$ is zero.

	The same reasoning applies to show that $\sr{F}^k(\sr{P})$ is minimal for $k \ge 0$. 
\end{proof}


\subsection{Rickard complexes}\label{subsec:Rickard_complexes}

Consider the subcomplex $\sr{F}^1(\sr{P})$ of $\sr{P}$. Recall from section~\ref{subsec:the_shape_of_P} that it consists of the objects $W(\lambda)$ of $\sr{P}$ for $\lambda$ in \begin{align*}
	T_1 \cap \Z^b &= \{(\lambda_1,\ldots,\lambda_b) \in \Z^b \:|\: 1 \ge \lambda_1 \ge \cdots \ge \lambda_b \ge 0\}\\
	&= \{(0,0,0,\ldots,0), (1,0,0,\ldots,0), (1,1,0,\ldots,0),\ldots,(1,1,1,\ldots,1)\}.
\end{align*}For $r \in \{0,\ldots,b\}$, let $1^r0^{b-r} \in T_1 \cap \Z^b$ be the tuple whose first $r$ entries are $1$ and whose remaining entries are $0$. 

\begin{lem}\label{lem:zetafoam}
	For $r \in \{0,\ldots,b\}$, we have $W(1^r0^{b-r}) = q^{r(d-b+1)}\,W_r$. The foam $\zeta^{(r-1)r}$ from $W(1^r0^{b-r})$ to $W(1^{r-1}0^{b-r+1})$ defined below as a composite makes the diagram on the right commute.\[
		\zeta^{(r-1)r} \coloneq\quad \begin{gathered}
			\includegraphics[width=.177\textwidth]{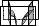}
			\\
			\includegraphics[width=.177\textwidth]{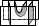}
			\vspace{-3pt}
		\end{gathered} \hspace{20pt}\begin{tikzcd}[column sep=75pt,row sep=30pt]
			q^{(r-1)(d-b+1)} \,W_{r-1} \ar[d,swap,"\textstyle\iota^{r-1}"] & q^{r(d-b+1)}\,W_r \ar[l,dotted,swap,"\textstyle\zeta^{(r-1)r}"] \ar[d,swap,"\textstyle\iota^{r}"]\\
			q^{(r-1)(d-b+1)+\binom{r-1}{2}+\binom{b-r+1}{2}} \,V_{r-1} & q^{r(d-b+1)+\binom{r}{2}+\binom{b-r}{2}}\, V_r \ar[l,swap,"\textstyle\partial_{b-1}\,\cdots\,\partial_r\,Z_{(r-1)r}" {xshift=5pt}]
		\end{tikzcd}
	\]
\end{lem}
\begin{proof}
	The foam $\zeta^{(r-1)r}$ is expressed as a composite for clarity where the intermediate web is \vspace{3pt}\[
		\begin{gathered}
			\labellist
			\pinlabel $b$ at -.7 4.8
			\pinlabel $a$ at -.7 0.2
			\pinlabel $r-1$ at 0.3 2.4
			\pinlabel $1$ at 3.9 2.4
			\pinlabel $b-r$ at 6.5 6
			\pinlabel $c$ at 14.1 4.8
			\pinlabel $d$ at 14.1 0.4
			\pinlabel $1$ at 9.3 2.6
			\pinlabel $l+r-1$ at 14.5 2.6
			\endlabellist
			\includegraphics[width=.135\textwidth]{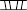}
			\vspace{-3pt}
		\end{gathered}
	\]Let $\{x_1,\ldots,x_{r-1}\}, \{x_r\},\{y\},\mathbf{E}_{r-1}$ be alphabets assigned to the four rungs colored by $r-1,1,1,l+r-1$, respectively, and let $\{x_{r+1},\ldots,x_b\}$ be the alphabet assigned to the top middle edge colored by $b-r$. Then the top foam in the composition is given by the inclusion of polynomials symmetric in $\{x_1,\ldots,x_{r-1},x_r\}$ and in $\{y\} \cup \mathbf{E}_{r-1}$ into polynomials separately symmetric in the four alphabets. Using the formulas provided in Example~\ref{example:bending}, the second foam in the composition is a quotient map identifying $x_r$ with $y$ followed $\partial_{b-1}\,\cdots\,\partial_r$. The composite $\iota^{r-1}\,\zeta^{(r-1)r}$ is therefore equal to $\partial_{b-1}\,\cdots\,\partial_r\,Z_{(r-1)r}\,\iota^r$.

	For the grading shift computations, note that $\deg(\partial_{b-1}\,\cdots\,\partial_r\,Z_{(r-1)r}) = d - 2(b-r)$ and $-\deg\iota^r = \binom{r}{2} + \binom{b-r}{2}$ so homogeneity of the maps in the commutative square gives \[
		H(1^r0^{b-r}) - H(1^{r-1}0^{b-r+1}) = \deg\iota^r + \deg(\partial_{b-1}\,\cdots\,\partial_r\,Z_{(r-1)r}) - \deg\iota^{r-1} = d-b+1.
	\]The value of $H(1^r0^{b-r})$ is then determined by $H(1^00^b) = H(0,\ldots,0) = 0$. 
\end{proof}

\begin{rem}
	By Proposition~\ref{prop:homSpaces}, $\Hom^{d-b+1}(W_r,W_{r-1}) \cong \Z$, and the foam $\zeta^{(r-1)r}\in\Hom^{d-b+1}(W_r,W_{r-1})$ is a generator. 
\end{rem}

By Lemmas~\ref{lem:differentialFactorsP} and \ref{lem:zetafoam}, the complex $\sr{F}^1(\sr{P})$ is \[
	\begin{tikzcd}[column sep=35pt]
		W_0 & t^{-1}q^{d-b+1}\,W_1 \ar[l,swap,"\zeta^{01}"] & t^{-2}q^{2(d-b+1)}\,W_2 \ar[l,swap,"\zeta^{12}"] & \cdots \ar[l,swap,"\zeta^{23}"] & t^{-b}q^{b(d-b+1)}\,W_b. \ar[l,swap,"\zeta^{(b-1)b}"]
	\end{tikzcd}
\]When $c = a$ and $d = b$, this complex is precisely the \emph{Rickard complex} assigned to the positive crossing \vspace{3pt} \[
	\quad\begin{gathered}
		\labellist
		\pinlabel $b$ at -2 5.5
		\pinlabel $a$ at -2 0.5
		\pinlabel $b$ at 14 0.8
		\pinlabel $a$ at 14 5.1
		\endlabellist
		\includegraphics[width=.06\textwidth]{1halfTwist}
		\vspace{-3pt}
	\end{gathered} \quad
\]When $c$ and $d$ are not specialized, it is the \emph{shifted} Rickard complex \cite{MR3412353,hogancamp2021skein} that Hogancamp--Rose--Wedrich show in \cite[Proposition 2.31]{hogancamp2021skein} is homotopy equivalent to the following tensor product complex. \vspace{3pt}\[
	\begin{gathered}
		\labellist
		\pinlabel $b$ at -2 5.5
		\pinlabel $a$ at -2 0.5
		\pinlabel $b$ at 14 0.8
		\pinlabel $a$ at 14 5.1
		\endlabellist
		\includegraphics[width=.06\textwidth]{1halfTwist}
		\vspace{-3pt}
	\end{gathered} \hspace{10pt} \otimes \hspace{10pt} \begin{gathered}
		\labellist
		\pinlabel $a$ at -2 5.1
		\pinlabel $b$ at -2 0.8
		\pinlabel $d$ at 8.5 0.8
		\pinlabel $c$ at 8.5 5.1
		\endlabellist
		\includegraphics[width=0.034\textwidth]{rung}
		\vspace{-3pt}
	\end{gathered} \quad = \quad \begin{gathered}
		\labellist
		\pinlabel $b$ at -2 5.5
		\pinlabel $a$ at -2 0.5
		\endlabellist
		\includegraphics[width=.06\textwidth]{1halfTwist}
		\vspace{-3pt}
	\end{gathered} \hspace{-3pt} \begin{gathered}
		\labellist
		\pinlabel $d$ at 8.5 0.8
		\pinlabel $c$ at 8.5 5.1
		\endlabellist
		\includegraphics[width=0.034\textwidth]{rung}
		\vspace{-3pt}
	\end{gathered}
\]Our convention for Rickard complexes places no grading shifts on the $W_0$ term which matches \cite{MR3412353,MR3545951}, while another standard convention places no shifts on the $W_b$ term. 

\begin{lem}\label{lem:rungSlide}
	If $a,b,c,d$ are nonnegative integers with $a + b = c + d$, then for each $k \ge 0$, there is a homotopy equivalence\[
		\left((t^{-1}q)^{\min(c,d) - \min(a,b)} q^{cd - ab}\right)^{-k} \hspace{15pt}\begin{gathered}
			\labellist
			\pinlabel $b$ at -2 5.5
			\pinlabel $a$ at -2 0.5
			\endlabellist
			\includegraphics[width=0.0343\textwidth]{rung_diag}
			\vspace{-3pt}
		\end{gathered} \hspace{-3pt} \begin{gathered}
			\labellist
			\endlabellist
			\includegraphics[width=.06\textwidth]{1halfTwist}
			\vspace{-3pt}
		\end{gathered}\hspace{3pt}\begin{gathered}
			\cdots\vspace{2pt}
		\end{gathered}\hspace{3pt}\begin{gathered}
			\labellist
			\pinlabel $d$ at 14 0.8
			\pinlabel $c$ at 14 5.5
			\endlabellist
			\includegraphics[width=.06\textwidth]{1halfTwist}
			\vspace{-3pt}
		\end{gathered} \quad\: \simeq \: \begin{cases}
			\quad\begin{gathered}
				\labellist
				\pinlabel $b$ at -2 5.5
				\pinlabel $a$ at -2 0.5
				\endlabellist
				\includegraphics[width=.06\textwidth]{1halfTwist}
				\vspace{-10pt}
			\end{gathered}\hspace{3pt}\begin{gathered}
				\cdots\vspace{-5pt}
			\end{gathered}\hspace{3pt}\begin{gathered}
				\labellist
				\endlabellist
				\includegraphics[width=.06\textwidth]{1halfTwist}
				\vspace{-10pt}
			\end{gathered} \hspace{-3pt} \begin{gathered}
				\labellist
				\pinlabel $b$ at -1 7.2
				\pinlabel $a$ at -1 -1.5
				\pinlabel $d$ at 8.5 0.8
				\pinlabel $c$ at 8.5 5.5
				\endlabellist
				\includegraphics[width=0.034\textwidth]{rung_diag}
				\vspace{-10pt}
			\end{gathered}\vspace{20pt} \quad & k \text{ is even}\\
			\quad\begin{gathered}
				\labellist
				\pinlabel $b$ at -2 5.5
				\pinlabel $a$ at -2 0.5
				\endlabellist
				\includegraphics[width=.06\textwidth]{1halfTwist}
				\vspace{-3pt}
			\end{gathered}\hspace{3pt}\begin{gathered}
				\cdots\vspace{0pt}
			\end{gathered}\hspace{3pt}\begin{gathered}
				\labellist
				\endlabellist
				\includegraphics[width=.06\textwidth]{1halfTwist}
				\vspace{-3pt}
			\end{gathered} \hspace{-3pt} \begin{gathered}
				\labellist
				\pinlabel $a$ at -1 7.1
				\pinlabel $b$ at -1 -1.9
				\pinlabel $d$ at 8.5 0.8
				\pinlabel $c$ at 8.5 5.5
				\endlabellist
				\includegraphics[width=0.034\textwidth]{rung}
				\vspace{-3pt}
			\end{gathered}\vspace{5pt} \quad & k \text{ is odd}
		\end{cases}
	\]where each diagram has $k$ positive crossings. 
\end{lem}
\begin{proof}
	We let $\left\llbracket - \right\rrbracket$ denote the other standard grading convention for the Rickard complex. In other words, we set \[
		\left\llbracket \quad\begin{gathered}
			\labellist
			\pinlabel $c$ at -2 5.2
			\pinlabel $d$ at -2 0.5
			\pinlabel $c$ at 13.8 0.2
			\pinlabel $d$ at 14 5.8
			\endlabellist
			\includegraphics[width=.06\textwidth]{1halfTwist}
			\vspace{-3pt}
		\end{gathered} \quad \right\rrbracket \coloneq (t^{-1}q)^{-\min(c,d)} \quad\begin{gathered}
			\labellist
			\pinlabel $c$ at -2 5.2
			\pinlabel $d$ at -2 0.5
			\pinlabel $c$ at 13.8 0.2
			\pinlabel $d$ at 14 5.8
			\endlabellist
			\includegraphics[width=.06\textwidth]{1halfTwist}
			\vspace{-3pt}
		\end{gathered} \quad
	\]which matches the notation in \cite{hogancamp2021skein}. The result now follows from the fork-sliding and fork-twisting homotopy equivalences \cite[Proposition 2.27]{hogancamp2021skein}. 
\end{proof}

\begin{rem}\label{rem:forkTwisting}
	If one of the four numbers $a,b,c,d$ is equal to $0$ in Lemma~\ref{lem:rungSlide}, then, the diagram on one side or the other is planar, following the convention that edges colored by zero are ignored. So in this case, the lemma reduces to the fork-twisting homotopy equivalence \cite[Proposition 2.27]{hogancamp2021skein}. 
\end{rem}

Now for any $k \ge 1$, consider the subcomplex $\sr{F}^k(\sr{P})$ which consists of the objects $W(\lambda)$ of $\sr{P}$ for $\lambda$ in \[
	T_k \cap \Z^b = \{(\lambda_1,\ldots,\lambda_b) \in \Z^b \:|\: k \ge \lambda_1 \ge\cdots\ge\lambda_b \ge 0\}.
\]The differential of $\sr{F}^k(\sr{P})$ consists of all components of the differential of $\sr{P}$ that are assigned to edges that connect vertices in $T_k \cap \Z^b$. Consider the following partition of $T_k \cap \Z^b$ into $b + 1$ disjoint sets \[
	T_k \cap \Z^b = U_{k,0} \sqcup U_{k,1} \sqcup \cdots \sqcup U_{k,b} \hspace{30pt} U_{k,r} \coloneq \{(\lambda_1,\ldots,\lambda_b) \in \Z^b \:|\: k = \lambda_1 = \cdots = \lambda_r > \lambda_{r+1} \ge \cdots \ge \lambda_b \ge 0 \}
\]Let $k^r0^{b-r} \in U_{k,r}$ be the lattice point whose first $r$ entries are $k$ and whose remaining entries are $0$. 
For $r \in \{0,\ldots,b\}$, let $\sr{U}_{k,r}(\sr{P})$ be the subquotient complex of $\sr{F}^k(\sr{P})$ consisting of the objects $W(\lambda)$ for $\lambda \in U_{k,r}$ with differential consisting of all components of the differential of $\sr{P}$ that are assigned to edges that connect vertices in $U_{k,r}$. 

\begin{example}
	When $a = b = c = d = 2$ and $k = 3$, the subquotient complexes $\sr{U}_{3,0}(\sr{P}), \sr{U}_{3,1}(\sr{P}),\sr{U}_{3,2}(\sr{P})$ of $\sr{F}^3(\sr{P})$ are\vspace{5pt} \[
		\begin{tikzpicture}[every node/.style={inner sep=2pt}, x=3cm, y=2cm]
			\node (a00) at (0,0) {$W_0$};
			\node (a10) at (1,0) {$t^{-1}q^1\,W_1$};
			\node (a20) at (2,0) {$t^{-2}q^3\,W_1$};
			\node (a30) at (3.3,0) {$t^{-3}q^5\,W_1$};

			\node (a11) at (1,1) {$t^{-2}q^2\,W_2$};
			\node (a21) at (2,1) {$t^{-3}q^5\,W_2^{1,1}$};
			\node (a31) at (3.3,1) {$t^{-4}q^7\,W_2^{1,1}$};

			\node (a22) at (2,2) {$t^{-4}q^8\,W_2$};
			\node (a32) at (3.3,2) {$t^{-5}q^{11}\,W_2^{1,1}$};

			\node (a33) at (3.3,3) {$t^{-6}q^{14}\,W_2$};

			\draw[{<[scale=1.5]}-] (a00)--(a10) node[midway,above] {$\pi Z_{01}$};
			\draw[{<[scale=1.5]}-] (a10)--(a20) node[midway,above] {$Q_1$};

			\draw[{<[scale=1.5]}-] (a10)--(a11) node[midway,left] {$Z_{12}\iota$};
			\draw[{<[scale=1.5]}-] (a20)--(a21) node[midway,left] {$Z_{12}$};

			\draw[{<[scale=1.5]}-] (a11)--(a21) node[midway,above] {$\pi Q_2s_1$};

			\draw[{<[scale=1.5]}-] (a21)--(a22) node[midway,left] {$Q_2\iota$};

			\begin{scope}[on background layer]
		    \path[
		      draw,
		      rounded corners, 
		      line width=0.6pt
		    ] 
		      ($ (a00.north west) + (-5pt, 0pt) $) --
		      ($ (a00.south west) + (-5pt,-5pt) $) --
		      ($ (a20.south east) + ( 5pt,-5pt) $) --
		      ($ (a22.north east) + ( 5pt, 5pt) $) -- 
		      ($ (a22.north west) + ( 0pt, 5pt) $) -- cycle;
		  	\end{scope}

		  	\draw[{<[scale=1.5]}-] (a30)--(a31) node[midway,left] {$Z_{12}$};
		  	\draw[{<[scale=1.5]}-] (a31)--(a32) node[midway,left] {$Q_{2}$};

		  	\node[draw, rounded corners, line width=0.6pt, inner sep=5pt, outer sep=2pt, fit=(a32) (a31) (a30)] (a3*g) {};

		  	\node[draw, rounded corners, line width=0.6pt, inner sep=5pt, outer sep=2pt, fit=(a33)] (a33g) {};

		  	\draw[{<[scale=1.5]}-, shorten <=8pt, shorten >=13pt] (a20)--(a30) node[midway,above] {$Z_{10}s^*\hspace{-4pt}_1\partial_1 Z_{01}$};

		  	\draw[{<[scale=1.5]}-, shorten <=4pt, shorten >=10pt] (a21)--(a31) node[midway,above] {$s^*\hspace{-4pt}_1Z_{21}Z_{12} \partial_1$};

		  	\draw[{<[scale=1.5]}-, shorten <=8pt, shorten >=8pt] (a22)--(a32) node[midway,above] {$\pi Z_{21}Z_{12} s_1$};

		  	\draw[{<[scale=1.5]}-, shorten <=8pt, shorten >=7pt] (a32)--(a33) node[midway,left] {$Z_{21}Z_{12}\iota$};
		\end{tikzpicture}\vspace{5pt}
	\]
\end{example}

For the following proposition, we set \[
	{}^0_n\sr{P}^c_d \coloneq \quad \begin{gathered}
			\labellist
			\pinlabel $0$ at -2 5.2
			\pinlabel $n$ at -2 0.3
			\pinlabel $d$ at 8.5 0.8
			\pinlabel $c$ at 8.5 5.5
			\endlabellist
			\includegraphics[width=0.035\textwidth]{rung_diag}
			\vspace{-3pt}
		\end{gathered} \quad = \quad \begin{gathered}
			\labellist
			\pinlabel $n$ at -2 0.3
			\pinlabel $d$ at 8.5 0.8
			\pinlabel $c$ at 8.5 5.5
			\endlabellist
			\includegraphics[width=0.035\textwidth]{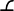}
			\vspace{-3pt}
		\end{gathered}
\]and $\sr{F}^k({}^0_n \sr{P}^c_d) = {}^0_n\sr{P}^c_d$ for all $k \ge 0$.

\begin{prop}\label{prop:tensorSplitting}
	For $r \in \{0,\ldots,b\}$ and $k \ge 1$, there is an identification \[
		\sr{U}_{k,r}\left({}^b_a \sr{P}^c_d\right) = t^{-kr}q^{H(k^r0^{b-r})} \quad \begin{gathered}
			\labellist
			\pinlabel $b$ at -2 5.5
			\pinlabel $a$ at -2 0.5
			\pinlabel $r$ at 5 2.7
			\pinlabel $a+r$ at 11.5 0.5
			\pinlabel $b-r$ at 11.5 5.5
			\endlabellist
			\includegraphics[width=0.035\textwidth]{rung}
			\vspace{-3pt}
		\end{gathered} \hspace{23pt} \otimes \sr{F}^{k-1}\left({}^{b-r}_{a+r}\sr{P}^c_d\right).
	\]Furthermore \[
		H(k^r0^{b-r}) = \begin{cases}
			r(k(a-b+r+1) - (c-b+r)) & k \text{ is odd }\\
			r(k(a-b+r+1) - (d-b+r)) & k \text{ is even.}
		\end{cases}
	\]
\end{prop}
\begin{proof}
	There is a bijection between $U_{k,r} \subset \Z^b$ and $T_{k-1} \cap \Z^{b-r}$ given by \[
		k^r\lambda \coloneq (k,\ldots,k,\lambda_{r+1},\ldots,\lambda_b) \longleftrightarrow \lambda\coloneq (\lambda_{r+1},\ldots,\lambda_b)
	\]where $k > \lambda_{r+1} \ge \cdots \ge \lambda_b \ge 0$. Note $r(k^r\lambda) = r + r(\lambda)$ and that if $g(\lambda) = (g_1,\ldots,g_m)$, then the multiplicities of the parts of $k^r\lambda$ are $(r,g_1,\ldots,g_m)$. We therefore have the desired identification\vspace{2pt}\[
		W^{g(k^r\lambda)}_{r(k^r\lambda)} = \quad\begin{gathered}
			\labellist
			\pinlabel $b$ at -2 5.5
			\pinlabel $a$ at -2 0.5
			\pinlabel $r$ at 5 2.7
			\pinlabel $a+r$ at 11.5 0.5
			\pinlabel $b-r$ at 11.5 5.5
			\endlabellist
			\includegraphics[width=0.035\textwidth]{rung}
			\vspace{-3pt}
		\end{gathered} \hspace{23pt} \otimes W^{g(\lambda)}_{r(\lambda)}\vspace{2pt}
	\]at the level of webs. The grading shift in the proposition statement makes the objects associated to $k^r0^{b-r} \in U_{k,r}$ and $0^{b-r} \in T_{k-1} \cap \Z^{b-r}$ match. The formula for $H(k^r0^{b-r})$ is obtained by directly from the formula for $H$ given in section~\ref{subsec:the_objects_of_P}. To show that the grading shifts for the remaining objects match, it suffices to identify the components of their differentials. Note that if $r = b$, then each side of the desired identification consists of just a single web with no differential, so we may assume $r < b$.

	First note that the web ${}^b_a (V_{r+t})^c_d$ with boundary data $c_L = (a,b)$ and $c_R = (d,c)$ factors as the tensor product \[
		{}^b_a(V_{r+t})^c_d = \quad\begin{gathered}
		 	\labellist
			\pinlabel $b$ at -1 4.9
			\pinlabel $a$ at -1 0.2
			\pinlabel $1$ at 1.5 2.4
			\pinlabel $1$ at 4.5 2.4
			\pinlabel $\cdots$ at 7.5 2.4
			\pinlabel $b-r$ at 15.5 4.8
			\pinlabel $a+r$ at 15.5 0.4
			\pinlabel $1$ at 9.5 2.4
			\endlabellist
		 	\includegraphics[width=.135\textwidth]{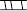}
		 	\vspace{-3pt}
		 \end{gathered} \hspace{25pt} \otimes {}^{b-r}_{a+r}(V_t)^c_d 
	\]for $t \in \{0,\ldots,b-r\}$. Fix $i \in \{r+1,\ldots,b\}$, an $r$-tuple $\eta = (\eta_1,\ldots,\eta_r) \in \{1,2\}^r$, and numbers $\varepsilon_{r+1},\ldots,\varepsilon_{i-1},\varepsilon_{i+1},\ldots,\varepsilon_b \in \{0,1,2,3\}$. For $j \in \{0,1,2,3\}$, let $\varepsilon^j \coloneq (\varepsilon_{r+1},\ldots,\varepsilon_{i-1},j,\varepsilon_{i+1},\ldots,\varepsilon_b)$ and $\eta\varepsilon^j = (\eta_1,\ldots,\eta_r,\varepsilon_{r+1},\ldots,\varepsilon_{i-1},j,\varepsilon_{i+1},\ldots,\varepsilon_b)$. It is straightforward to see that the component of the differential of ${}^b_a \sr{K}^c_d$ assigned to the edge $\eta\varepsilon^{[j,j+1]}$ of the cube $[0,3]^b$ connecting $\eta\varepsilon^j$ and $\eta\varepsilon^{j+1}$ respects the above tensor product decomposition. In particular, it is the tensor product of the identity on the left with the component of the differential of ${}^{b-r}_{a+r}\sr{K}^c_d$ assigned to the edge $\varepsilon^{[j,j+1]}$ of the cube $[0,3]^{b-r}$ connecting $\varepsilon^j$ and $\varepsilon^{j+1}$. 
	The key observation is that the components of the differentials are both defined in terms of the same sequence $\varepsilon_{i+1},\ldots,\varepsilon_b$. As for the indices of the expressions for the components of the differentials, we note that after tensoring with the identity on the left, the actions of $\partial_1,\ldots,\partial_{b-r-1}$ on ${}^{b-r}_{a+r}(V_t)^c_d$ become the actions of $\partial_{r+1},\ldots,\partial_{b-r}$ on ${}^b_a(V_{r+t})^c_d$. Additionally, the action of $Q_t$ on ${}^{b-r}_{a+r}(V_t)^c_d$ becomes the action of $Q_{r+t}$ on ${}^b_a(V_{r+t})^c_d$ because both are multiplication by $e_{c - (b-r) + t}(\C - x_{r+t} - \cdots - x_b)$. 

	Next, observe that the maps $\iota^{g(k^r\lambda)}$ and $\iota^{g(\lambda)}$ from section~\ref{subsec:webs_and_foams_in_P} satisfy $\iota^{g(k^r\lambda)} = \iota^r \otimes \iota^{g(\lambda)}$ where $\iota^r$ is the map \[
		\begin{gathered}
			\labellist
			\pinlabel $b$ at -1 4.9
			\pinlabel $a$ at -1 0.2
			\pinlabel $r$ at 1.5 2.4
			\pinlabel $b-r$ at 7.5 4.8
			\pinlabel $a+r$ at 7.5 0.4
			\endlabellist
			\includegraphics[width=0.052\textwidth]{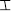}
			\vspace{-3pt}
		\end{gathered} \hspace{23pt} \to \quad\begin{gathered}
		 	\labellist
			\pinlabel $b$ at -1 4.9
			\pinlabel $a$ at -1 0.2
			\pinlabel $1$ at 1.5 2.4
			\pinlabel $1$ at 4.5 2.4
			\pinlabel $\cdots$ at 7.5 2.4
			\pinlabel $b-r$ at 15.5 4.8
			\pinlabel $a+r$ at 15.5 0.4
			\pinlabel $1$ at 9.5 2.4
			\endlabellist
		 	\includegraphics[width=.135\textwidth]{rung_exploded}
		 	\vspace{-3pt}
		 \end{gathered}
	\]induced by the inclusion $\Z[x_1,\ldots,x_r]^{\fk{S}_r} \hookrightarrow \Z[x_1,\ldots,x_r]$. The desired identification now follows from Lemma~\ref{lem:differentialFactorsP}. 
\end{proof}

We will use the following lemma in the proof of property 2 of Theorem~\ref{thm:mainThm}. Recall that a map $f\in \Hom^i(B,C)$ is \emph{primitive} if the equality $f = kg$ for an integer $k \in \Z$ and $g \in \Hom^i(B,C)$ implies that $k = \pm1$. Additionally, observe that we may consider the quotient of any singular Bott--Samelson bimodule with boundary data $c_L = (a,b)$ and $c_R = (d,c)$ by the ideal generated by $e_i(\C - \B)$ for $i > b-c$. Any bimodule map between such singular Bott--Samelson bimodules descends to a map on quotients. This procedure mimics forming a partial closure of the web in a rather weak way that is sufficient for our purposes. We may also instead quotient out by the ideal generated by $e_i(\mathbf{D} - \B)$ for $i > d-b$. 

\begin{lem}\label{lem:primitive}
	The maps $e_{d-b+1}(\mathbf{D} - \B) \in \Hom^{2(d-b+1)}(W_b,W_b)$ and $e_{c-b+1}(\C - \B) \in \Hom^{2(c-b+1)}(W_b,W_b)$ are primitive. Furthermore, the first map descends to a primitive map after quotienting by the ideal generated by $e_i(\C - \B)$ for $i > c-b$, while the second map descends to a primitive map after quotienting by the ideal generated by $e_i(\mathbf{D} - \B)$ for $i > d - b$.
\end{lem}
\begin{proof}
	We note that quotienting out by $e_i(\C - \B)$ for $i > c - b$ is equivalent to first tensoring with $\Sym(\mathbf{X})$ where $\mathbf{X}$ is an alphabet of size $c - b$, and then quotienting by $e_i(\C - \B) - e_i(\mathbf{X})$ for $i \ge 1$. From this description, we see that the quotient of $q^{cd}W_b$ by this ideal is \begin{align*}
		\frac{\Sym(\A) \otimes \Sym(\B) \otimes \Sym(\C) \otimes\Sym(\mathbf{D})\otimes\Sym(\mathbf{X})}{(e_i(\A + \B) - e_i(\C + \mathbf{D}), e_i(\C - \B) - e_i(\mathbf{X})\:|\: i\ge 1)} &= \frac{\Sym(\A) \otimes \Sym(\B) \otimes \Sym(\mathbf{D})\otimes \Sym(\mathbf{X})}{(e_i(\A + \B) - e_i(\B + \mathbf{X} + \mathbf{D}) \:|\: i \ge 1)}\\
		&= \Sym(\B) \otimes\Sym(\mathbf{D})\otimes\Sym(\mathbf{X})
	\end{align*}It is clear then that $e_{d-b+1}(\mathbf{D} - \B)$ descends to a primitive map on this quotient. Note that descending to a primitive map implies that the original map is primitive. A similar argument establishes the analogous results for the endomorphism $e_{c-b+1}(\C - \B)$.
\end{proof}

\begin{proof}[Proof of property 2 in Theorem~\ref{thm:mainThm}]
	We prove the result by strong induction on $b \ge 1$. While it is possible to formulate the statement so that $b = 0$ is the base case, we explain the case $b = 1$ to illustrate the argument. 

	\underline{Base case}: $b = 1$. We induct on $k\ge 0$. The cases $k = 0$ and $k = 1$ have already been established. For the inductive step, let $k \ge 2$ and consider the $b+1 = 2$ subquotient complexes $\sr{U}_{k,0}({}^1_a\sr{P}^c_d)$ and $\sr{U}_{k,1}({}^1_a\sr{P}^c_d)$ of $\sr{F}^k({}^1_a\sr{P}^c_d)$ \[
		\sr{F}^k({}^1_a\sr{P}^c_d) = \quad \begin{gathered}	
		\begin{tikzpicture}[every node/.style={inner sep=2pt}, x=2.5cm, y=2cm]
			\node (a0) at (0,0) {$W(0)$};
			\node (a1) at (1,0) {$t^{-1}W(1)$};
			\node (a2) at (2,0) {$\:\cdots\:$};
			\node (a3) at (3,0) {$t^{-k+1}W(k-1)$};
			\node (a4) at (4.3,0) {$t^{-k}W(k)$};

			\draw[{<[scale=1.5]}-] (a0)--(a1) node[midway,above] {};
			\draw[{<[scale=1.5]}-] (a1)--(a2) node[midway,above] {};
			\draw[{<[scale=1.5]}-] (a2)--(a3) node[midway,above] {};

			\node[draw, rounded corners, line width=0.6pt, inner sep=5pt, outer sep=2pt, fit=(a0) (a1) (a2) (a3)] (a*g) {};

		  	\node[draw, rounded corners, line width=0.6pt, inner sep=5pt, outer sep=2pt, fit=(a4)] (a4g) {};

		  	\draw[{<[scale=1.5]}-, shorten <=0pt, shorten >=0pt] (a*g)--(a4g) node[midway,above] {$\tau$};
		\end{tikzpicture}
		\vspace{-5pt}
		\end{gathered} 
	\]which we also write as \[
		\sr{F}^k({}^1_a\sr{P}^c_d) =\: \begin{tikzcd}
			\sr{U}_{k,0}({}^1_a\sr{P}^c_d) & \sr{U}_{k,1}({}^1_a\sr{P}^c_d). \ar[l,swap,"\mu_{01}"]
		\end{tikzcd}
	\]We may think of $\mu_{01}$ as a chain map from $t\sr{U}_{k,1}({}^1_a\sr{P}^c_d)$ to $\sr{U}_{k,0}({}^1_a\sr{P}^c_d)$. Since $\sr{U}_{k,0}({}^1_b\sr{P}^c_d) = \sr{F}^{k-1}({}^1_a\sr{P}^c_d)$, the inductive hypothesis and Lemma~\ref{lem:rungSlide} give a homotopy equivalence \[
		\sr{U}_{k,0}({}^1_a\sr{P}^c_d) \simeq \left((t^{-1}q)^{\min(c,d) - 1} q^{cd - a}\right)^{1-k} \hspace{15pt}\begin{gathered}
			\labellist
			\pinlabel $1$ at -2 5.3
			\pinlabel $a$ at -2 0.5
			\endlabellist
			\includegraphics[width=0.0343\textwidth]{rung_diag}
			\vspace{-3pt}
		\end{gathered} \hspace{-3pt} \begin{gathered}
			\labellist
			\endlabellist
			\includegraphics[width=.06\textwidth]{1halfTwist}
			\vspace{-3pt}
		\end{gathered}\hspace{3pt}\begin{gathered}
			\cdots\vspace{2pt}
		\end{gathered}\hspace{3pt}\begin{gathered}
			\labellist
			\pinlabel $d$ at 14 0.8
			\pinlabel $c$ at 14 5.2
			\endlabellist
			\includegraphics[width=.06\textwidth]{1halfTwist}
			\vspace{-3pt}
		\end{gathered} \quad
	\]where the diagram has $k-1$ positive crossings. Note that the fork-twisting equivalence of Remark~\ref{rem:forkTwisting} gives\begin{align*}
		\sr{U}_{k,1}({}^1_a\sr{P}^c_d) &= t^{-k}q^{H(k)} \quad \begin{gathered}
			\labellist
			\pinlabel $a$ at -2 0.8
			\pinlabel $1$ at -2 5.3
			\endlabellist
			\includegraphics[width=0.035\textwidth]{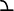}
			\vspace{-3pt}
		\end{gathered} \hspace{-3pt}\begin{gathered}
			\labellist
			\pinlabel $d$ at 8.5 0.8
			\pinlabel $c$ at 8.5 5.2
			\endlabellist
			\includegraphics[width=0.035\textwidth]{rung_partial}
			\vspace{-3pt}
		\end{gathered} \quad \simeq t^{-k}q^{H(k)}\left((t^{-1}q)^{\min(c,d)} q^{cd}\right)^{1-k} \quad\begin{gathered}
			\labellist
			\pinlabel $a$ at -2 0.8
			\pinlabel $1$ at -2 5.3
			\endlabellist
			\includegraphics[width=0.035\textwidth]{rung_partial_flip}
			\vspace{-3pt}
		\end{gathered} \hspace{-3pt}\begin{gathered}
			\labellist
			\endlabellist
			\includegraphics[width=0.035\textwidth]{rung_partial}
			\vspace{-3pt}
		\end{gathered} \hspace{-3pt} \begin{gathered}
			\labellist
			\endlabellist
			\includegraphics[width=.06\textwidth]{1halfTwist}
			\vspace{-3pt}
		\end{gathered}\hspace{3pt}\begin{gathered}
			\cdots\vspace{2pt}
		\end{gathered}\hspace{3pt}\begin{gathered}
			\labellist
			\pinlabel $d$ at 14 0.8
			\pinlabel $c$ at 14 5.2
			\endlabellist
			\includegraphics[width=.06\textwidth]{1halfTwist}
			\vspace{-3pt}
		\end{gathered} \quad
	\end{align*}where the diagram has $k-1$ positive crossings. Using the formula for $H(k)$ given in Proposition~\ref{prop:tensorSplitting} with $r = 1 = b$, we have for $k \ge 1$ \[
		H(k) = \begin{cases}
			kn-c & k \text{ is odd}\\
			kn-d & k \text{ is even}
		\end{cases}
	\]and $H(0) = 0$. We now tensor $\sr{F}^k({}^1_a\sr{P}^c_d)$ on the right with $k-1$ negative crossings and obtain \begin{align*}
		\sr{F}^k({}^1_a\sr{P}^c_d) \otimes \hspace{7pt}\begin{gathered}
			\labellist
			\pinlabel $c$ at -2 5.2
			\pinlabel $d$ at -2 0.8
			\endlabellist
			\includegraphics[width=.06\textwidth]{1negHalfTwist}
			\vspace{-3pt}
		\end{gathered}\hspace{3pt}\begin{gathered}
			\cdots\vspace{2pt}
		\end{gathered}&\hspace{3pt}\begin{gathered}
			\labellist
			\endlabellist
			\includegraphics[width=.06\textwidth]{1negHalfTwist}
			\vspace{-3pt}
		\end{gathered} \: = \begin{tikzcd}[ampersand replacement=\&]
			\sr{U}_{k,0}({}^1_a\sr{P}^c_d) \otimes \hspace{7pt}\begin{gathered}
				\labellist
				\pinlabel $c$ at -2 5.2
				\pinlabel $d$ at -2 0.8
				\endlabellist
				\includegraphics[width=.06\textwidth]{1negHalfTwist}
				\vspace{-3pt}
			\end{gathered}\hspace{3pt}\begin{gathered}
				\cdots\vspace{2pt}
			\end{gathered}\hspace{3pt}\begin{gathered}
				\labellist
				\endlabellist
				\includegraphics[width=.06\textwidth]{1negHalfTwist}
				\vspace{-3pt}
			\end{gathered} \& \sr{U}_{k,1}({}^1_a\sr{P}^c_d) \otimes \hspace{7pt}\begin{gathered}
				\labellist
				\pinlabel $c$ at -2 5.2
				\pinlabel $d$ at -2 0.8
				\endlabellist
				\includegraphics[width=.06\textwidth]{1negHalfTwist}
				\vspace{-3pt}
			\end{gathered}\hspace{3pt}\begin{gathered}
				\cdots\vspace{2pt}
			\end{gathered}\hspace{3pt}\begin{gathered}
				\labellist
				\endlabellist
				\includegraphics[width=.06\textwidth]{1negHalfTwist}
				\vspace{-3pt}
			\end{gathered}\ar[l,swap,"\mu_{01}\otimes\Id" {yshift=2pt,xshift=0pt}]
		\end{tikzcd}\\
		&\simeq \begin{cases}
			\left((t^{-1}q)^{\min(c,d) - 1} q^{cd-a}\right)^{1-k} \left(\hspace{3pt}\begin{tikzcd}[ampersand replacement=\&]
				\begin{gathered}
					\labellist
					\pinlabel $1$ at -2 5.3
					\pinlabel $a$ at -2 0.5
					\pinlabel $d$ at 8.5 0.8
					\pinlabel $c$ at 8.5 5.2
					\endlabellist
					\includegraphics[width=0.0343\textwidth]{rung_diag}
					\vspace{-3pt}
				\end{gathered}\hspace{7pt} \& t^{-1}q^d \quad\begin{gathered}
					\labellist
					\pinlabel $a$ at -2 0.8
					\pinlabel $1$ at -2 5.3
					\endlabellist
					\includegraphics[width=0.035\textwidth]{rung_partial_flip}
					\vspace{-3pt}
				\end{gathered} \hspace{-3pt}\begin{gathered}
					\labellist
					\pinlabel $d$ at 8.5 0.8
					\pinlabel $c$ at 8.5 5.2
					\endlabellist
					\includegraphics[width=0.035\textwidth]{rung_partial}
					\vspace{-3pt}
				\end{gathered}\hspace{3pt} \ar[l,swap,"\nu_{\,01}"]
			\end{tikzcd}\right) & k \text{ is odd }\vspace{7pt}\\
			\left((t^{-1}q)^{\min(c,d) - 1} q^{cd-a}\right)^{1-k} \left(\hspace{3pt}\begin{tikzcd}[ampersand replacement=\&]
				\begin{gathered}
					\labellist
					\pinlabel $1$ at -2 5.3
					\pinlabel $a$ at -2 0.5
					\pinlabel $c$ at 8.5 0.5
					\pinlabel $d$ at 8.5 5.8
					\endlabellist
					\includegraphics[width=0.0343\textwidth]{rung_diag}
					\vspace{-3pt}
				\end{gathered}\hspace{7pt} \& t^{-1}q^c \quad\begin{gathered}
					\labellist
					\pinlabel $a$ at -2 0.8
					\pinlabel $1$ at -2 5.3
					\endlabellist
					\includegraphics[width=0.035\textwidth]{rung_partial_flip}
					\vspace{-3pt}
				\end{gathered} \hspace{-3pt}\begin{gathered}
					\labellist
					\pinlabel $c$ at 8.5 0.5
					\pinlabel $d$ at 8.5 5.8
					\endlabellist
					\includegraphics[width=0.035\textwidth]{rung_partial}
					\vspace{-3pt}
				\end{gathered}\hspace{3pt} \ar[l,swap,"\nu_{\,01}"]
			\end{tikzcd}\right) & k \text{ is even }
		\end{cases}
	\end{align*}by homotopy invariance under the Reidemeister II move \cite[Proposition 2.25]{hogancamp2021skein} and the homological perturbation lemma (see for example \cite[section 3.1]{MR4903257}). By Proposition~\ref{prop:homSpaces}, the morphism space in which $\nu_{01}$ lives has rank $1$ and is generated by the foam $\zeta^{01}$ in Lemma~\ref{lem:zetafoam} \[
		\nu_{01} \in \begin{cases}
			\Hom^d(W_0,W_1) = \Z\cdot\zeta^{01} & k\text{ is odd }\\
			\Hom^c(W_0,W_1) = \Z\cdot\zeta^{01} & k\text{ is even }
		\end{cases}
	\]Consider the map sending $\mu_{01}$ to $\nu_{01}$ from the space of chain maps up to homotopy from $t\sr{U}_{k,1}({}^1_a\sr{P}^c_d)$ to $\sr{U}_{k,0}({}^1_a\sr{P}^c_d)$ to the morphism space from $W_0$ to $W_1$ of the appropriate degree depending on the parity of $k$. Again by homotopy invariance under the Reidemeister II move, this map is an isomorphism. See for example \cite[Lemma 3.14]{MR4903257} for more details.

	We now show that $\mu_{01}$ is a generator of the space of chain maps up to homotopy, which implies that $\nu_{01}$ is also a generator. Since $\sr{U}_{k,1}({}^1_a\sr{P}^c_d)$ has zero differential, the chain map $\mu_{01}$ is a generator if and only if the component map $\tau\colon W(k) \to W(k-1)$ is primitive. Note that \[
		\tau = \begin{cases}
			Z_{10}Z_{01} = e_d(\mathbf{D} - \B) & k\text{ is odd }\\
			Q_1 = e_{c}(\C - \B) & k \text{ is even }
		\end{cases}
	\]so $\tau$ is primitive by Lemma~\ref{lem:primitive}. Hence $\nu_{01}$ is a generator so $\nu_{01} = \pm\zeta^{01}$. It follows that the two term complex with differential $\nu_{01}$ is isomorphic to the shifted Rickard complex so \begin{align*}
		\sr{F}^k({}^1_a\sr{P}^c_d) \otimes \hspace{7pt}\begin{gathered}
			\labellist
			\pinlabel $c$ at -2 5.2
			\pinlabel $d$ at -2 0.8
			\endlabellist
			\includegraphics[width=.06\textwidth]{1negHalfTwist}
			\vspace{-3pt}
		\end{gathered}\hspace{3pt}\begin{gathered}
			\cdots\vspace{2pt}
		\end{gathered}\hspace{3pt}\begin{gathered}
			\labellist
			\endlabellist
			\includegraphics[width=.06\textwidth]{1negHalfTwist}
			\vspace{-3pt}
		\end{gathered} \:& \simeq\: \left((t^{-1}q)^{\min(c,d) - 1} q^{cd-a}\right)^{1-k} \quad \begin{gathered}
			\labellist
			\pinlabel $1$ at -2 5.3
			\pinlabel $a$ at -2 0.5
			\endlabellist
			\includegraphics[width=.06\textwidth]{1halfTwist}
			\vspace{-3pt}
		\end{gathered} \hspace{-3pt} \begin{gathered}
			\labellist
			\endlabellist
			\includegraphics[width=0.034\textwidth]{rung}
			\vspace{-3pt}
		\end{gathered}\\
		&\simeq\: \left((t^{-1}q)^{\min(c,d) - 1} q^{cd-a}\right)^{-k}\hspace{3pt} \quad \begin{gathered}
			\labellist
			\pinlabel $1$ at -2 5.3
			\pinlabel $a$ at -2 0.5
			\endlabellist
			\includegraphics[width=0.0343\textwidth]{rung_diag}
			\vspace{-3pt}
		\end{gathered} \hspace{-3pt} \begin{gathered}
			\labellist
			\endlabellist
			\includegraphics[width=.06\textwidth]{1halfTwist}
			\vspace{-3pt}
		\end{gathered}
	\end{align*}Finally, we tensor once more on the right, now with $k-1$ positive crossings, and obtain $\sr{F}^k({}^1_a\sr{P}^c_d)$ by Reidemeister II invariance. Thus \begin{align*}
		\sr{F}^k({}^1_a\sr{P}^c_d) &\simeq \: \left((t^{-1}q)^{\min(c,d) - 1} q^{cd-a}\right)^{-k}\hspace{3pt} \quad \begin{gathered}
			\labellist
			\pinlabel $1$ at -2 5.3
			\pinlabel $a$ at -2 0.5
			\endlabellist
			\includegraphics[width=0.0343\textwidth]{rung_diag}
			\vspace{-3pt}
		\end{gathered} \hspace{-3pt} \begin{gathered}
			\labellist
			\endlabellist
			\includegraphics[width=.06\textwidth]{1halfTwist}
			\vspace{-3pt}
		\end{gathered}\hspace{-3pt} \begin{gathered}
			\labellist
			\endlabellist
			\includegraphics[width=.06\textwidth]{1halfTwist}
			\vspace{-3pt}
		\end{gathered}\hspace{3pt}\begin{gathered}
			\cdots\vspace{2pt}
		\end{gathered}\hspace{3pt}\begin{gathered}
			\labellist
			\pinlabel $d$ at 14 0.8
			\pinlabel $c$ at 14 5.2
			\endlabellist
			\includegraphics[width=.06\textwidth]{1halfTwist}
			\vspace{-3pt}
		\end{gathered}
	\end{align*}which after one more application of Lemma~\ref{lem:rungSlide} proves the result. 

	\underline{\smash{Inductive step}}. Let $b \ge 2$. We again proceed by induction on $k \ge 0$. The cases $k = 0$ and $k = 1$ have already been established. Let $k \ge 2$, and consider the $b + 1$ subquotient complexes of $\sr{F}^k({}^b_a\sr{P}^c_d)$ \[
		\sr{F}^k({}^b_a\sr{P}^c_d) = \begin{tikzcd}[column sep=50pt]
			\sr{U}_{k,0}({}^b_a\sr{P}^c_d) & \sr{U}_{k,1}({}^b_a\sr{P}^c_d) \ar[l,swap,"\mu_{01}"] & \cdots \ar[l,swap,"\mu_{12}"] & \sr{U}_{k,b}({}^b_a\sr{P}^c_d) \ar[l,swap,"\mu_{(b-1)b}"]
		\end{tikzcd}
	\]We show that $\mu_{(r-1)r}$, viewed as a chain map from $t\sr{U}_{k,r}({}^b_a\sr{P}^c_d)$ to $\sr{U}_{k,r-1}({}^b_a\sr{P}^c_d)$, is primitive in the space of chain maps up to homotopy. Let $\tau_{(r-1)r}$ be the component of $\mu_{(r-1)r}$ from $W(k^r(k-1)^{b-r})$ to $W(k^{r-1}(k-1)^{b-r+1})$, which are the objects of $t\sr{U}_{k,r}({}^b_a \sr{P}^c_d)$ and $\sr{U}_{k,r}({}^b_a\sr{P}^c_d)$ of lowest cohomological degree. Suppose for the sake of contradiction that $\mu_{(r-1)r}$ is not primitive in the space of chain maps up to homotopy. Then there is a homotopy $h$ from $t\sr{U}_{k,r}({}^b_a\sr{P}^c_d)$ to $\sr{U}_{k,r-1}({}^b_a\sr{P}^c_d)$ such that $\mu_{(r-1)r} + dh + hd$ is not primitive as a chain map. Observe that $h$ must vanish on $W(k^r(k-1)^{b-r})$ by cohomological degree considerations, so in particular, $\tau_{(r-1)r} + hd\colon W(k^r(k-1)^{b-r}) \to W(k^{r-1}(k-1)^{b-r+1})$ is not a primitive bimodule map. Note that if $r = b$, the differential $d$ of $\sr{U}_{k,r}({}^b_a\sr{P}^c_d)$ is zero. When $r < b$, there is only one nontrivial component of the differential $d$ of $\sr{U}_{k,r}({}^b_a\sr{P}^c_d)$ out of $W(k^r(k-1)^{b-r})$. \[
		\begin{tikzcd}[column sep=80pt]
			W(k^{r-1}(k-1)^{b-r+1}) & W(k^r(k-1)^{b-r}) \ar[l,swap,"\textstyle\tau_{(r-1)r}"] \ar[d,swap,"\textstyle d"]\\
			& W(k^r(k-1)^{b-r-1}(k-2)) \ar[lu,"\textstyle h"]
		\end{tikzcd}
	\]

	\underline{Case}: $k$ is odd. Then $\tau_{(r-1)r}$ fits into the following commutative diagram by Lemma~\ref{lem:differentialFactorsP}. \[
		\begin{tikzcd}[column sep=120pt,row sep=35pt]
			W(k^{r-1}(k-1)^{b-r+1}) \ar[d,swap,"\textstyle\iota^{g(k^{r-1}(k-1)^{b-r+1})}"] & W(k^r(k-1)^{b-r}) \ar[l,swap,"\textstyle\tau_{(r-1)r}"] \ar[d,swap,"\textstyle\iota^{g(k^r(k-1)^{b-r})}"]\\
			q^{\bullet_{r-1}}V_b & q^{\bullet_r} V_b \ar[l,swap,"\textstyle\partial^*\hspace{-4pt}_r\,\cdots\,\partial^*\hspace{-4pt}_{b-1}\,Z_{b(b-1)}\,Z_{(b-1)b}\,s_{b-1}\,\cdots\,s_r"]
		\end{tikzcd}
	\]where $\bullet_{r-1} = H(k^{r-1}(k-1)^{b-r+1})+\xi(g(k^{r-1}(k-1)^{b-r+1}))$ and $\bullet_r = H(k^r(k-1)^{b-r}) + \xi(g(k^r(k-1)^{b-r}))$. Suppose $r = b$. The fact that $\tau_{(r-1)r}$ is not primitive implies that its composite with $\iota^{g(k^{b-1}(k-1))}$ is also not primitive. By commutativity of the above diagram and by further composing with $\partial_1\,\cdots\,\partial_{r-1}$ on the left, we find that \[
		\partial_1\,\cdots\,\partial_{b-1}\, e_d(\mathbf{D} - x_b) \,\iota^b = e_{d-b+1}(\mathbf{D} - \B)\,\iota^b
	\]is not primitive, where we have used the Leibniz rule for $\partial_i$, the identity $s_i\iota^b = \iota^b$, and the fact that $Z_{b(b-1)}\,Z_{(b-1)b} = e_d(\mathbf{D} - x_b)$. But then $e_{d-b+1}(\mathbf{D} - \B) = \pi^b\,p_b\,e_{d-b+1}(\mathbf{D} - \B)\,\iota^b \in \Hom(W_b,W_b)$ is not primitive, contradicting Lemma~\ref{lem:primitive}. 

	Now assume $r < b$, and note that Lemma~\ref{lem:differentialFactorsP} implies that $d\colon W(k^r(k-1)^{b-r}) \to W(k^r(k-1)^{b-r-1}(k-2))$ is $Q_b = e_{c}(\C - x_b)$. Consider quotients by the ideal generated by $e_i(\C - \B)$ for $i > c - b$. We claim that $e_c(\C - x_b)$ descends to zero in the quotient. Just as in the proof of Lemma~\ref{lem:primitive}, the quotient of $W_b$ by this ideal is a shifted copy of \[
		\Z[x_1,\ldots,x_b] \otimes \Sym(\mathbf{D}) \otimes \Sym(\mathbf{X})
	\]where $|\mathbf{X}| = c - b$. The element $e_c(\C - x_b)$ is sent to $e_c(x_1 + \cdots + x_b + \mathbf{X} - x_b) = e_c(x_1 + \cdots + x_{b-1} + \mathbf{X}) = 0$ because $\{x_1,\ldots,x_{b-1}\} \cup \mathbf{X}$ is an alphabet of size $c-1$. Next, because $\tau_{(r-1)r} + hd$ is not primitive, it descends to a non-primitive map on quotients. But $d$ descends to zero so $\tau_{(r-1)r}$ itself must descend to a non-primitive map on quotients. By an argument similar to the one used in the case of $r = b$, we find that this implies that $e_{d-b+1}(\mathbf{D} - \B) \in \Hom(W_b,W_b)$ descends to a non-primitive map on quotients, contradicting Lemma~\ref{lem:primitive}. 

	\underline{Case}: $k$ is even. Then $\tau_{(r-1)r}$ fits into the following commutative diagram by Lemma~\ref{lem:differentialFactorsP}. \[
		\begin{tikzcd}[column sep=100pt,row sep=35pt]
			W(k^{r-1}(k-1)^{b-r+1}) \ar[d,swap,"\textstyle\iota^{g(k^{r-1}(k-1)^{b-r+1})}"] & W(k^r(k-1)^{b-r}) \ar[l,swap,"\textstyle\tau_{(r-1)r}"] \ar[d,swap,"\textstyle\iota^{g(k^r(k-1)^{b-r})}"]\\
			q^{\bullet_{r-1}}V_b & q^{\bullet_r} V_b \ar[l,swap,"\textstyle\partial^*\hspace{-4pt}_r\,\cdots\,\partial^*\hspace{-4pt}_{b-1}\,Q_b\,s_{b-1}\,\cdots\,s_r"]
		\end{tikzcd}
	\]If $r = b$, then the non-primitivity of $\tau_{(r-1)r}$ implies that of \[
		\partial_1\,\cdots\,\partial_{b-1}\,e_c(\C - x_b)\,s_{b-1}\,\cdots\,s_r\,\iota^b = e_{c-b+1}(\C - \B)\,\iota^b
	\]and therefore $e_{c-b+1}(\C - \B) \in \Hom(W_b,W_b)$, contradicting Lemma~\ref{lem:primitive}. 

	Assume $r < b$, and note that $d\colon W(k^r(k-1)^{b-r}) \to W(k^r(k-1)^{b-r-1}(k-2))$ factors through $Z_{(b-1)b}$. Consider quotients by the ideal generated by $e_i(\mathbf{D} - \B)$ for $i > d - b$. Unfortunately, the map $Z_{(b-1)b}$ does not need to descend to the zero map. However, its adjoint $Z_{b(b-1)} \in \Hom^d(V_{b-1},V_b)$ does, which we now verify. Just as in the proof of Lemma~\ref{lem:primitive}, the quotient of $V_b$ by the ideal generated by $e_i(\mathbf{D} - \B)$ for $i > d -b$ is $\Z[x_1,\ldots,x_b] \otimes \Sym(\C) \otimes \Sym(\mathbf{Y})$ where $|\mathbf{Y}| = d - b$ and $e_i(\A) = e_i(\C + \mathbf{Y})$ and $e_i(\mathbf{D}) = e_i(x_1 + \cdots + x_b + \mathbf{Y})$. The map $Z_{b(b-1)}$ sends $1 \in V_{b-1}$ to $e_d(\mathbf{D} - x_b)\in V_b$ which is descends to $e_d(x_1 + \cdots + x_{b-1} + \mathbf{Y}) = 0$ since $\{x_1,\ldots,x_{b-1}\} \cup \mathbf{Y}$ is an alphabet of size $d-1$. 

	Because $\tau_{(r-1)r} + hd$ is not primitive, its adjoint $\tau_{(r-1)r}^* + d^*h^*$ is also not primitive. The adjoint $d^*$ factors through $Z_{b(b-1)}$ and hence descends to zero in the quotient, so $\tau_{(r-1)r}^*$ descends to a non-primitive map in the quotient. Hence, the adjoint of \[
		\iota^b\,p_b\,\partial_1\,\cdots\,\partial_{b-1}\,Q_b\, \iota^b = e_{c-b+1}(\C - \B) \in \Hom(W_b,W_b)
	\]descends to a non-primitive map in the quotient. The map $e_{c-b+1}(\C - \B)$ is self-adjoint, so the fact that it descends to a non-primitive map in the quotient contradicts Lemma~\ref{lem:primitive}. 

	Altogether, this establishes that $\mu_{(r-1)r}$ is primitive in the space of chain maps from $t\sr{U}_{k,r}({}^b_a\sr{P}^c_d)$ to $\sr{U}_{k,r-1}({}^b_a\sr{P}^c_d)$.
	Next, by Proposition~\ref{prop:tensorSplitting}, the inductive hypotheses, and Lemma~\ref{lem:rungSlide}, we have \begin{align*}
		\sr{U}_{k,r}({}^b_a\sr{P}^c_d) &= t^{-kr}q^{H(k^r0^{b-r})} \quad \begin{gathered}
			\labellist
			\pinlabel $b$ at -2 5.5
			\pinlabel $a$ at -2 0.5
			\pinlabel $r$ at 5 2.7
			\pinlabel $a+r$ at 11.5 0.5
			\pinlabel $b-r$ at 11.5 5.5
			\endlabellist
			\includegraphics[width=0.035\textwidth]{rung}
			\vspace{-3pt}
		\end{gathered} \hspace{23pt} \otimes \sr{F}^{k-1}\left({}^{b-r}_{a+r}\sr{P}^c_d\right)\\
		&\simeq t^{-kr}q^{H(k^r0^{b-r})}\left((t^{-1}q)^{\min(c,d) - b+r} q^{cd - (a+r)(b-r)}\right)^{1-k} \quad \begin{gathered}
			\labellist
			\pinlabel $b$ at -2 5.5
			\pinlabel $a$ at -2 0.5
			\pinlabel $r$ at 5 2.7
			\endlabellist
			\includegraphics[width=0.0343\textwidth]{rung}
			\vspace{-3pt}
		\end{gathered} \hspace{-3pt} \begin{gathered}
			\labellist
			\endlabellist
			\includegraphics[width=0.0343\textwidth]{rung_diag}
			\vspace{-3pt}
		\end{gathered} \hspace{-3pt} \begin{gathered}
			\labellist
			\endlabellist
			\includegraphics[width=.06\textwidth]{1halfTwist}
			\vspace{-3pt}
		\end{gathered}\hspace{3pt}\begin{gathered}
			\cdots\vspace{2pt}
		\end{gathered}\hspace{3pt}\begin{gathered}
			\labellist
			\pinlabel $d$ at 14 0.8
			\pinlabel $c$ at 14 5.2
			\endlabellist
			\includegraphics[width=.06\textwidth]{1halfTwist}
			\vspace{-3pt}
		\end{gathered} \quad
	\end{align*}where there are $k-1$ positive crossings in the diagram. We tensor on the right with $k-1$ negative crossings and obtain \begin{align*}
		\sr{F}^k({}^b_a&\sr{P}^c_d) \otimes \hspace{7pt}\begin{gathered}
			\labellist
			\pinlabel $c$ at -2 5.2
			\pinlabel $d$ at -2 0.8
			\endlabellist
			\includegraphics[width=.06\textwidth]{1negHalfTwist}
			\vspace{-3pt}
		\end{gathered}\hspace{3pt}\begin{gathered}
			\cdots\vspace{2pt}
		\end{gathered}\hspace{3pt}\begin{gathered}
			\labellist
			\endlabellist
			\includegraphics[width=.06\textwidth]{1negHalfTwist}
			\vspace{-3pt}
		\end{gathered}\\
		&\simeq \begin{cases}
			((t^{-1}q)^{\min(c,d)-b}q^{cd-ab})^{1-k}\left(\begin{tikzcd}[ampersand replacement=\&,column sep=30pt]
			 	W_0 \& q^{d-b+1}W_1 \ar[l,swap,"\nu_{\,01}"] \& \cdots \ar[l,swap,"\nu_{\,12}"] \& q^{b(d-b+1)}W_b \ar[l,swap,"\nu_{\,(b-1)b}"]
			\end{tikzcd}\right) & k \text{ is odd }\vspace{7pt}\\
			((t^{-1}q)^{\min(c,d)-b}q^{cd-ab})^{1-k}\left(\begin{tikzcd}[ampersand replacement=\&,column sep=30pt]
			 	W_0 \& q^{c-b+1}W_1 \ar[l,swap,"\nu_{\,01}"] \& \cdots \ar[l,swap,"\nu_{\,12}"] \& q^{b(c-b+1)}W_b \ar[l,swap,"\nu_{\,(b-1)b}"]
			\end{tikzcd}\right) & k \text{ is even }
		\end{cases}
	\end{align*}The grading shifts on the right-hand side are computed using Proposition~\ref{prop:tensorSplitting}, and the map $\nu_{(r-1)r}$ is induced by the map $\mu_{(r-1)r}$. Because $\mu_{(r-1)r}$ is primitive in the space of chain maps up to homotopy, it follows that $\mu_{(r-1)r}$ is primitive. By Lemma~\ref{prop:homSpaces}, it follows that $\mu_{(r-1)r}$ and $\zeta^{(r-1)r}$ agree up to a sign as they are both generators of a free abelian group of rank $1$. Hence, we have\begin{align*}
		\sr{F}^k({}^b_a\sr{P}^c_d) \otimes \hspace{7pt}\begin{gathered}
			\labellist
			\pinlabel $c$ at -2 5.2
			\pinlabel $d$ at -2 0.8
			\endlabellist
			\includegraphics[width=.06\textwidth]{1negHalfTwist}
			\vspace{-3pt}
		\end{gathered}\hspace{3pt}\begin{gathered}
			\cdots\vspace{2pt}
		\end{gathered}\hspace{3pt}\begin{gathered}
			\labellist
			\endlabellist
			\includegraphics[width=.06\textwidth]{1negHalfTwist}
			\vspace{-3pt}
		\end{gathered} \:& \simeq\: \left((t^{-1}q)^{\min(c,d) - b} q^{cd-ab}\right)^{1-k} \quad \begin{gathered}
			\labellist
			\pinlabel $b$ at -2 5.3
			\pinlabel $a$ at -2 0.5
			\endlabellist
			\includegraphics[width=.06\textwidth]{1halfTwist}
			\vspace{-3pt}
		\end{gathered} \hspace{-3pt} \begin{gathered}
			\labellist
			\endlabellist
			\includegraphics[width=0.034\textwidth]{rung}
			\vspace{-3pt}
		\end{gathered}\\
		&\simeq\: \left((t^{-1}q)^{\min(c,d) - b} q^{cd-ab}\right)^{-k}\hspace{3pt} \quad \begin{gathered}
			\labellist
			\pinlabel $b$ at -2 5.3
			\pinlabel $a$ at -2 0.5
			\endlabellist
			\includegraphics[width=0.0343\textwidth]{rung_diag}
			\vspace{-3pt}
		\end{gathered} \hspace{-3pt} \begin{gathered}
			\labellist
			\endlabellist
			\includegraphics[width=.06\textwidth]{1halfTwist}
			\vspace{-3pt}
		\end{gathered}
	\end{align*}so by tensoring once more on the right with $k-1$ positive crossings, we obtain \begin{align*}
		\sr{F}^k({}^b_a\sr{P}^c_d) &\simeq \: \left((t^{-1}q)^{\min(c,d) - b} q^{cd-ab}\right)^{-k}\hspace{3pt} \quad \begin{gathered}
			\labellist
			\pinlabel $b$ at -2 5.3
			\pinlabel $a$ at -2 0.5
			\endlabellist
			\includegraphics[width=0.0343\textwidth]{rung_diag}
			\vspace{-3pt}
		\end{gathered} \hspace{-3pt} \begin{gathered}
			\labellist
			\endlabellist
			\includegraphics[width=.06\textwidth]{1halfTwist}
			\vspace{-3pt}
		\end{gathered}\hspace{-3pt} \begin{gathered}
			\labellist
			\endlabellist
			\includegraphics[width=.06\textwidth]{1halfTwist}
			\vspace{-3pt}
		\end{gathered}\hspace{3pt}\begin{gathered}
			\cdots\vspace{2pt}
		\end{gathered}\hspace{3pt}\begin{gathered}
			\labellist
			\pinlabel $d$ at 14 0.8
			\pinlabel $c$ at 14 5.2
			\endlabellist
			\includegraphics[width=.06\textwidth]{1halfTwist}
			\vspace{-3pt}
		\end{gathered}
	\end{align*}which proves the result. 
\end{proof}

\subsection{Contractibility}\label{subsec:contractibility}

\begin{proof}[Proof of property 3 of Theorem~\ref{thm:mainThm}]
	Let $r \in \{1,\ldots,b\}$, and let \vspace{2pt}\[
		{}^{c}_{d}\sr{R}^{b-r}_{a+r} \coloneq \quad\begin{gathered}
			\labellist
			\pinlabel $c$ at -2 5.3
			\pinlabel $d$ at -2 0.8
			\pinlabel $a+r$ at 11.5 0.5
			\pinlabel $b-r$ at 11.5 5.5
			\endlabellist
			\includegraphics[width=0.035\textwidth]{rung}
			\vspace{-3pt}
		\end{gathered}\hspace{25pt} \hspace{20pt} {}^{c}_{d}\sr{R}^{a+r}_{b-r} \coloneq \quad\begin{gathered}
			\labellist
			\pinlabel $c$ at -2 5.3
			\pinlabel $d$ at -2 0.8
			\pinlabel $b-r$ at 11.5 0.8
			\pinlabel $a+r$ at 11.5 5.3
			\endlabellist
			\includegraphics[width=0.035\textwidth]{rung_diag}
			\vspace{-3pt}
		\end{gathered}\hspace{25pt} \hspace{20pt} {}^{b-r}_{a+r}\sr{R}^b_a \coloneq \hspace{25pt}\begin{gathered}
			\labellist
			\pinlabel $b-r$ at -5 5.5
			\pinlabel $a+r$ at -5 0.5
			\pinlabel $a$ at 8.5 0.5
			\pinlabel $b$ at 8.5 5.5
			\endlabellist
			\includegraphics[width=0.035\textwidth]{rung_diag}
			\vspace{-3pt}
		\end{gathered}\quad \hspace{20pt} {}^{a+r}_{b-r}\sr{R}^b_a \coloneq \hspace{25pt}\begin{gathered}
			\labellist
			\pinlabel $a+r$ at -5 5.3
			\pinlabel $b-r$ at -5 0.8
			\pinlabel $a$ at 8.5 0.5
			\pinlabel $b$ at 8.5 5.5
			\endlabellist
			\includegraphics[width=0.035\textwidth]{rung}
			\vspace{-3pt}
		\end{gathered}\quad\vspace{2pt}
	\]where the rungs are colored by $c+r-b,d+r-b,b+r-b,a+r-b$, respectively. We first show that ${}^{b-r}_{a+r}\sr{R}^b_a \otimes {}^b_a \sr{P}^c_d$ is contractible. 
	By property 2 of Theorem~\ref{thm:mainThm} and Lemma~\ref{lem:rungSlide}, we have homotopy equivalences \begin{align*}
		{}^{b-r}_{a+r}\sr{R}^b_a \otimes \sr{F}^k({}^b_a\sr{P}^c_d) &\simeq \begin{cases}
			\hspace{25pt}\begin{gathered}
				\labellist
				\pinlabel $b-r$ at -5 5.5
				\pinlabel $a+r$ at -5 0.5
				\pinlabel $r$ at 1.5 2.5
				\endlabellist
				\includegraphics[width=0.034\textwidth]{rung_diag}
				\vspace{-10pt}
			\end{gathered}\hspace{-3pt}\begin{gathered}
				\labellist
				\endlabellist
				\includegraphics[width=.06\textwidth]{1halfTwist}
				\vspace{-10pt}
			\end{gathered}\hspace{3pt}\begin{gathered}
				\cdots\vspace{-5pt}
			\end{gathered}\hspace{3pt}\begin{gathered}
				\labellist
				\endlabellist
				\includegraphics[width=.06\textwidth]{1halfTwist}
				\vspace{-10pt}
			\end{gathered} \hspace{-3pt} \begin{gathered}
				\labellist
				\pinlabel $b$ at -1 7.2
				\pinlabel $a$ at -1 -1.5
				\pinlabel $d$ at 8.5 0.8
				\pinlabel $c$ at 8.5 5.5
				\endlabellist
				\includegraphics[width=0.034\textwidth]{rung_diag}
				\vspace{-10pt}
			\end{gathered}\vspace{20pt} \quad & k \text{ is even}\\
			\hspace{25pt}\begin{gathered}
				\labellist
				\pinlabel $b-r$ at -5 5.5
				\pinlabel $a+r$ at -5 0.5
				\pinlabel $r$ at 1.5 2.5
				\endlabellist
				\includegraphics[width=0.034\textwidth]{rung_diag}
				\vspace{-3pt}
			\end{gathered}\hspace{-3pt}\begin{gathered}
				\labellist
				\endlabellist
				\includegraphics[width=.06\textwidth]{1halfTwist}
				\vspace{-3pt}
			\end{gathered}\hspace{3pt}\begin{gathered}
				\cdots
			\end{gathered}\hspace{3pt}\begin{gathered}
				\labellist
				\endlabellist
				\includegraphics[width=.06\textwidth]{1halfTwist}
				\vspace{-3pt}
			\end{gathered} \hspace{-3pt} \begin{gathered}
				\labellist
				\pinlabel $a$ at -1 7.1
				\pinlabel $b$ at -1 -1.9
				\pinlabel $d$ at 8.5 0.8
				\pinlabel $c$ at 8.5 5.5
				\endlabellist
				\includegraphics[width=0.034\textwidth]{rung}
				\vspace{-3pt}
			\end{gathered}\vspace{5pt} \quad & k \text{ is odd}
		\end{cases}\\[10pt]
		&\simeq \begin{cases}
			((t^{-1}q)^r q^{ab-(a+r)(b-r)})^k\hspace{25pt}\begin{gathered}
				\labellist
				\pinlabel $b-r$ at -5 5.5
				\pinlabel $a+r$ at -5 0.5
				\endlabellist
				\includegraphics[width=.06\textwidth]{1halfTwist}
				\vspace{-5pt}
			\end{gathered}\hspace{3pt}\begin{gathered}
				\cdots\vspace{-5pt}
			\end{gathered}\hspace{3pt}\begin{gathered}
				\labellist
				\endlabellist
				\includegraphics[width=.06\textwidth]{1halfTwist}
				\vspace{-5pt}
			\end{gathered} \hspace{-3pt} \begin{gathered}
				\labellist
				\pinlabel $r$ at 1.5 2.5
				\endlabellist
				\includegraphics[width=0.034\textwidth]{rung_diag}
				\vspace{-5pt}
			\end{gathered}\hspace{-3pt}\begin{gathered}
				\labellist
				\pinlabel $b$ at 0.5 7.2
				\pinlabel $a$ at 0.5 -1.5
				\pinlabel $d$ at 8.5 0.8
				\pinlabel $c$ at 8.5 5.5
				\endlabellist
				\includegraphics[width=0.034\textwidth]{rung_diag}
				\vspace{-5pt}
			\end{gathered}\vspace{20pt} \quad & k \text{ is even}\\
			((t^{-1}q)^r q^{ab-(a+r)(b-r)})^k\hspace{25pt}\begin{gathered}
				\labellist
				\pinlabel $b-r$ at -5 5.5
				\pinlabel $a+r$ at -5 0.5
				\endlabellist
				\includegraphics[width=.06\textwidth]{1halfTwist}
				\vspace{-3pt}
			\end{gathered}\hspace{3pt}\begin{gathered}
				\cdots
			\end{gathered}\hspace{3pt}\begin{gathered}
				\labellist
				\endlabellist
				\includegraphics[width=.06\textwidth]{1halfTwist}
				\vspace{-3pt}
			\end{gathered} \hspace{-3pt} \begin{gathered}
				\labellist
				\pinlabel $r$ at 1.5 2.5
				\endlabellist
				\includegraphics[width=0.034\textwidth]{rung}
				\vspace{-3pt}
			\end{gathered}\hspace{-3pt}\begin{gathered}
				\labellist
				\pinlabel $a$ at 0.5 7.1
				\pinlabel $b$ at 0.5 -1.9
				\pinlabel $d$ at 8.5 0.8
				\pinlabel $c$ at 8.5 5.5
				\endlabellist
				\includegraphics[width=0.034\textwidth]{rung}
				\vspace{-3pt}
			\end{gathered}\vspace{5pt} \quad & k \text{ is odd}
		\end{cases}
	\end{align*}Observe that the complex in the second line is bounded above in cohomological degree by $-kr$. Because ${}^{b-r}_{a+r}\sr{R}^b_a\otimes{}^b_a\sr{P}^c_d$ has an exhaustive filtration by subcomplexes \[
		{}^{b-r}_{a+r}\sr{R}^b_a\otimes\sr{F}^0({}^b_a\sr{P}^c_d) \subset {}^{b-r}_{a+r}\sr{R}^b_a\otimes\sr{F}^1({}^b_a\sr{P}^c_d) \subset \cdots
	\]where the $k$th term is homotopy equivalent to a complex bounded above in cohomological degree by $-kr$, which goes to $-\infty$ as $k\to\infty$, standard techniques imply that ${}^{b-r}_{a+r}\sr{R}^b_a\otimes{}^b_a\sr{P}^c_d$ is contractible. See for example \cite[Lemma 3.26]{MR3663594} and \cite[Lemma 2.38]{MR3811775}.

	Similar computations using Property 2 of Theorem~\ref{thm:mainThm} and Lemma~\ref{lem:rungSlide} show that each of \[
		\sr{F}^k({}^b_a\sr{P}^c_d) \otimes {}^c_d\sr{R}^{b-r}_{a+r} \qquad \sr{F}^k({}^b_a\sr{P}^c_d) \otimes {}^c_d\sr{R}^{a+r}_{b-r} \qquad {}^{a+r}_{b-r} \sr{R}^b_a \otimes {}^b_a \sr{P}^c_d
	\]are homotopy equivalence to a complex bounded above in cohomological degree by $-kr$, so all four tensor products in the theorem statement are contractible by the same reasoning.
\end{proof}

\subsection{Proof of Theorem~\ref{thm:mainIntroThm}}\label{subsec:proof_of_intro_main_theorem}

\begin{proof}[Proof of Theorem~\ref{thm:mainIntroThm}]
	By definition of the Rickard complex assigned to a positive crossing, we have \[
		\quad\begin{gathered}
			\labellist
			\pinlabel $b$ at -2 5
			\pinlabel $b$ at -2 0.5
			\pinlabel $b$ at 14 0.5
			\pinlabel $b$ at 14 5.1
			\endlabellist
			\includegraphics[width=.06\textwidth]{1negHalfTwist}
			\vspace{-3pt}
		\end{gathered}\quad \otimes \sr{P}_b = \begin{tikzcd}[column sep=35pt]
			W_0 \otimes \sr{P}_b & t^{-1}q^1\, W_1 \otimes \sr{P}_b \ar[l,swap,"\zeta^{01}"] & \cdots \ar[l,swap,"\zeta^{12}"] & t^{-b}q^b\, W_b \otimes \sr{P}_b \ar[l,swap,"\zeta^{(b-1)b}"]
		\end{tikzcd}
	\]By property 3 of Theorem~\ref{thm:mainThm}, we know that $W_r \otimes \sr{P}_b$ is contractible for $r \in \{1,\ldots,b\}$ so the complex retracts onto $W_0 \otimes \sr{P} = \sr{P}$. The same reasoning implies that $\sr{P}_b$ is also invariant under tensoring with the crossing on the right. Furthermore, the tensor square \[
		\sr{P}_b \otimes \sr{P}_b = \bigoplus_{\lambda \in T \cap \Z^b} t^{-|\lambda|}q^{H(\lambda)} \,W_{r(\lambda)}^{g(\lambda)}\otimes \sr{P}_b
	\]deformation retracts onto the term corresponding to $\lambda = (0,\ldots,0)$ so $\sr{P}_b$ is idempotent. 

	The Euler characteristic $p_b$ of $\sr{P}_b$, viewed as an endomorphism of $\Lambda^b(V) \otimes \Lambda^b(V)$, has the property that $W_r\,p_b = 0$ for $r\in \{1,\ldots,b\}$ and the coefficient of $W_0$ in the expression of $p_b$ in terms of the basis $W_0,W_1,\ldots,W_r$ is $1$. This characterizes the idempotent projection onto the highest-weight irreducible summand. 
\end{proof}

\begin{rem}
	By the same reasoning as in the proof of Theorem~\ref{thm:mainIntroThm}, the complex $\sr{P} = {}^b_a \sr{P}^{b}_a$ is idempotent and its Euler characteristic is the idempotent endomorphism of $\Lambda^a(V) \otimes \Lambda^b(V)$ that projects onto the irreducible summand corresponding to the $2$-column Young diagram whose column lengths are $a$ and $b$.
\end{rem}

\phantomsection
\addcontentsline{toc}{section}{\protect\numberline{}References}
\raggedright
\bibliography{minimal}
\bibliographystyle{alpha}

\vfill

\textit{Department of Mathematics, Princeton University, Princeton NJ 08540}

\textit{School of Mathematics, Institute for Advanced Study, Princeton NJ 08540}

\textit{Email addresses:} \hspace{2pt} {joshuaxw@princeton.edu} \hspace{2pt} {jxwang@ias.edu}

\end{document}